\documentclass{amsart}
\usepackage{comment}
\usepackage{amscd}
\usepackage{hyperref}
\usepackage{marginnote}
\usepackage{epsf,latexsym,graphicx,dsfont}
\usepackage[matrix,arrow,tips,curve]{xy}
\usepackage{color}
\usepackage{tikz}
\usepackage{euscript}
\newlength{\defbaselineskip} 
\usepackage{footnote}
\usepackage{color}
\usepackage{todonotes}
\usepackage{MnSymbol}

\usepackage{amsfonts}
\usepackage[centertags]{amsmath}

\usepackage{amscd}

\usepackage{fullpage}
\usepackage{color}
\usepackage{euscript}
\usepackage{tikz}
\usepackage{footnote}
\usepackage{color}
\usepackage{todonotes}
\usepackage{url}
\newtheorem{thm}{Theorem}[section]
\newtheorem{cor}[thm]{Corollary}
\newtheorem{corollary}[thm]{Corollary}
\newtheorem{corr}[thm]{Corollary}

\newtheorem{lemma}[thm]{Lemma}
\newtheorem{lem}[thm]{Lemma}
\newtheorem{prop}[thm]{Proposition}

\newtheorem{prob}[thm]{Problem}

\theoremstyle{definition}
\newtheorem{defi}[thm]{Definition}

\newtheorem{rem}[thm]{Remark}
\newtheorem{remark}[thm]{Remark}

\usetikzlibrary{matrix,arrows}
\makeatletter \tikzset{
  edge node/.code={%
      \expandafter\def\expandafter\tikz@tonodes\expandafter{\tikz@tonodes #1}}}
\makeatother \tikzset{
  subseteq/.style={
    draw=none,
    edge node={node [sloped, allow upside down, auto=false]{$\subseteq$}}},
  Subseteq/.style={
    draw=none,
    every to/.append style={
      edge node={node [sloped, allow upside down, auto=false]{$\subseteq$}}}
  }
}

 \numberwithin{equation}{section}
\numberwithin{equation}{section} \theoremstyle{definition}
\DeclareMathOperator{\Pic}{Pic}
\DeclareMathOperator{\NS}{NS}
\DeclareMathOperator{\T}{T}
\DeclareMathOperator{\Aut}{Aut}
\DeclareMathOperator{\Fix}{Fix}

\DeclareMathOperator{\Sym}{Sym}

\DeclareMathOperator{\id}{id}
\DeclareMathOperator{\divi}{div}

\DeclareMathOperator{\Hilb}{Hilb}

          \newcommand\PP{{\mathbb{P}}}
           
          \newcommand\C{{\mathbb{C}  }}

            \newcommand\ZZ{{\mathbb Z}}

          \newcommand\oo{\mathcal O}
             \newcommand\Q{\mathbb Q}
             \newcommand{\ra}{\rightarrow}
          \newcommand\Z{\mathbb{Z}}

          \newcommand\rk{\mathrm{rk}}

\definecolor{zielony}{rgb}{0.5, 0.9, 0.1}
\definecolor{czerwony}{rgb}{0.8, 0.2, 0.1}
\definecolor{niebieski}{rgb}{0.3, 0.1, 0.9}

\newcounter{appendice}

\title{Generalized Nikulin surfaces and irreducible symplectic fourfolds}

\author[C. Camere]{Chiara Camere}
\address{C. Camere: Dipartimento di Matematica, Università degli Studi di Milano, Milano, Italy}
\email{chiara.camere@unimi.it}

\author[A. Garbagnati]{Alice Garbagnati}
\address{A. Garbagnati: Dipartimento di Matematica, Università degli Studi di Milano, Milano, Italy}
\email{alice.garbagnati@unimi.it}

\author[G. Kapustka]{Grzegorz Kapustka}
\address{G. Kapustka: Faculty of Mathematics and Informatics, Jagiellonian University, ul. Łojasiewicza 6, 30-348 Krak\'ow, Poland}
\email{grzegorz.kapustka@uj.edu.pl}

\author[M. Kapustka]{Micha\l{} Kapustka}
\address{M. Kapustka: Institute of Mathematics of the Polish Academy of Sciences, ul. Śniadeckich 8, 00-656 Warszawa, Poland}
\email{michal.kapustka@impan.pl}

\begin{document}
\maketitle

\begin{abstract} A Nikulin surface is the minimal resolution of the quotient of a $K3$ surface $S$ by a symplectic involution $\iota_S$. Equivalently, it is the $2$-dimensional component of the fixed locus of the involution induced by $\iota_S$ on the Hilbert scheme $S^{[2]}$.
    We study $K3$ surfaces $F$ that are the $2$-dimensional component of the fixed locus of a symplectic involution $\iota$ on hyper-K\"ahler manifolds $X$ of $K3^{[2]}$-type; we call them generalized Nikulin surfaces.
    We show that a projective $K3$ surface is a generalized Nikulin surface if and only if its N\'eron-Severi lattice contains primitively the lattice $E_7(-2)$. 
    Moreover, we show that the transcendental lattices $\T_F$ and $\T_{\widetilde{X/ \iota}}$, where $\widetilde{X/ \iota}$ is the terminalization of the quotient $X/\iota$, are Hodge isometric.
    Finally, we describe projective models of generalized Nikulin surfaces of small degrees.
\end{abstract}
\section{Introduction}
Let $X$ be a manifold of $K3^{[2]}$-type admitting a symplectic involution $\iota$.
Then, the fixed locus of $\iota$ consists of a $K3$ surface $F$ and $28$ isolated points.
We will call $F$ a \emph{generalized Nikulin surface}. The reason for this name is that these surfaces naturally appear as deformation of the so called Nikulin surfaces, which are minimal resolutions of quotients of $K3$ surfaces by symplectic involutions.

A $K3$ surface (resp. a manifold of $K3^{[2]}$-type) admits a symplectic involution if and only if its N\'eron--Severi lattice contains primitively $E_8(-2)$, \cite{vGS, C,Mon}. The quotient of a $K3$ surface by its symplectic involution is singular in 8 points and its minimal resolution (the Nikulin surface) is obtained by blowing up once each of these singular points.

The N\'eron--Severi lattice $N$ of a general non projective Nikulin surface is called a Nikulin lattice and it is the extension of the lattice $\langle -2\rangle^{\oplus 8}$  with generators $E_i$, that are the exceptional curves of the resolution,
by the class $\frac{1}{2}(\sum_i E_i)$. A projective $K3$ surface is a Nikulin surface if and only if its N\'eron--Severi lattice contains primitively $N$, see \cite{GS,N}.  

There is another characterization of the Nikulin surfaces, which motivates the name generalized Nikulin surface for the surface in the fixed locus of a symplectic involution on a manifold of $K3^{[2]}$-type. Indeed, if $S$ is a K3 surface with a symplectic involution $\iota_S$,  then the Hilbert scheme of two points on $S$, $S^{[2]}$, is naturally endowed with a symplectic involution $\iota_S^{[2]}$ induced by $\iota_S$. The surface fixed by $\iota_S^{[2]}$ on $S^{[2]}$ is the minimal resolution of the quotient $S/\iota_S$, i.e. the Nikulin surface of $(S,\iota)$. Deforming $S^{[2]}$ to a manifold of $K3^{[2]}$-type $X$ and $\iota_S^{[2]}$ to a symplectic involution $\iota$ of $X$ one deforms also the fixed surface, from a Nikulin surface to a generalized Nikulin surface.  

We prove that the condition of being a generalized Nikulin surface is a lattice-theoretic condition which allows to describe their family in terms of lattice-polarized $K3$ surfaces.
\begin{thm}\label{main}The N\'eron--Severi lattice of a $K3$ surface $F$ contains primitively the lattice $E_7(-2)$  if and only if $F$ is isomorphic to the $2$-dimensional component of the fixed locus of a symplectic involution on a  manifold of $K3^{[2]}$-type, i.e.~it is isomorphic to a generalized Nikulin surface.
\end{thm}

There is another variety strictly related with the pair $(X,\iota)$, which is the terminalization of the quotient $X/\iota$. It will be denoted by $Y$ and it is called a Nikulin orbifold. The orbifold $Y$ is an irreducible symplectic orbifold and therefore its second cohomology group is endowed with a Beauville--Bogomolov  form. The following result relates its transcendental structure with the one of the generalized Nikulin surfaces, and proves a conjecture stated in \cite[Conjecture 3.12]{CGKK}. 
\begin{thm}\label{main2} Let $(X,\iota)$ be a manifold of $K3^{[2]}$-type with a symplectic involution $\iota$, $Y$ the terminalization of the quotient $X/\iota$ and $F$ the surface fixed by $\iota$ on $X$. Then there exists a Hodge isometry between the transcendental lattices of $F$ and $Y$.
\end{thm}

The idea of the proof of Theorem \ref{main}, given in Section \ref{section: proofs of 2nd theorem and definition of Generalized Nikulin lattice}, is the following:
as said before, deforming $(S^{[2]},\iota_S^{[2]})$ to a general $(X,\iota)$ the fixed locus $F_0$ of $\iota_S^{[2]}$ deforms to the fixed locus $F$ of $\iota$. 
In the deformation of $S^{[2]}$ to $X$, we loose the algebraicity of the diagonal $\Delta$ in $S^{[2]}$ (which indeed becomes a transcendental class) and so the algebraicity of $F_0\cap \Delta$ in $F_0$. Since $F_0\cap \Delta=\frac{1}{2}(\sum_i E_i)$, we consider the sublattice $G\subset N$ in the Nikulin lattice that is the orthogonal complement to the class $\sum_i E_i$. We call it the generalized Nikulin lattice, and we prove that it is isometric to $E_7(-2)$.
In order to show that any $K3$ surface $F$ containing primitively $E_7(-2)$ in its Picard lattice is a generalized Nikulin surface we relate this surface with the irreducible symplectic orbifold $Y$ by using Theorem \ref{main2} and this allows to reconstruct the manifold of $K3^{[2]}$-type $X$ and its symplectic involution $\iota$ such that $F$ is the fixed surface of $\iota$ on $X$. 
An important ingredient in the proof is the characterization of orbifolds bimeromorphic to Nikulin orbifolds inside the family of the orbifold of Nikulin type, as the ones that contain a $(-4)$-class of divisibility $2$ with a prescribed orthogonal lattice $\mathbb{T}$ (see Definition \ref{def:T}) in their N\'eron--Severi lattice (Theorem \ref{Thm1.1}).

After recalling some preliminary results in Section \ref{sec: preliminaries}, we present in more detail the three main results of our paper in Section \ref{section: The main results}, which are proved in the following two sections. In particular, in Section \ref{section: proofs of 1st theorem} we characterize orbifolds bimeromorphic to Nikulin orbifolds as the orbifolds of Nikulin type whose N\'eron--Severi lattice contains a class with self-intersection $-4$, divisibility 2 and orthogonal lattice isometric to $\mathbb{T}$,  and we show that the locus of marked Nikulin orbifolds is dense in any connected component of the moduli space of marked orbifolds of Nikulin type. This result is analogous to the one which characterizes the Hilbert scheme of points on a $K3$ surface in the family of the manifold of $K3^{[2]}$-type (see \cite[Theorem 1.1]{MarkmanMehrotra}). In Section \ref{section: proofs of 2nd theorem and definition of Generalized Nikulin lattice}, we first study the so-called generalized Nikulin lattice and 
we prove that such a lattice is isometric to $E_7(-2)$. These preliminary results are applied to prove first Theorem \ref{main2} and then Theorem \ref{main}. 
In Section \ref{section: Families of projective generalized Nikulin surfaces} we consider the family of projective generalized Nikulin surfaces.  
We observe that the family has countably many components, each of which is $12$-dimensional and is described as a family of lattice-polarized $K3$ surfaces. The lattices which are used to polarize these families are either $\langle 2d\rangle\oplus E_7(-2)$ for any positive integer $d$, or even finite-index overlattices (of index either 2 or 4) of $\langle 2d\rangle\oplus E_7(-2)$, see Proposition \ref{prop: Ns e T of projective F} for details.

In Section~\ref{section:projective models}, in analogy with the case of Nikulin
surfaces \cite{FV, vGS, GS, V}, we describe the geometry of the projective models
of generalized Nikulin surfaces for several values of the polarization degree. We
relate it with the geometry of the associated manifold of $K3^{[2]}$-type and Nikulin
orbifold. We show
in Corollary \ref{cor: isom between NS(F) and NS(S),NS(Z)}, that a $(\langle 2\rangle\oplus E_7(-2))$-polarized generalized Nikulin surface is a seven nodal quartic.
We prove in Proposition~\ref{prop 10.10} that a general
$(\langle 4\rangle\oplus E_7(-2))'$-polarized generalized Nikulin surface is the
double cover of a quadric in $\mathbb{P}^3$ branched along its intersection with a
Kummer quartic surface. Moreover, in Corollary~\ref{8'} we show that that the moduli space of
$(\langle 8\rangle \oplus E_7(-2))'$-polarized generalized Nikulin surfaces is
unirational and give a geometric description of its general elements.

In Section~\ref{section: related problems} we present some consequences and
generalizations of the study of generalized Nikulin surfaces.
In Section~\ref{subset:the residuals involution} we explain that generalized Nikulin
surfaces admit natural residual involutions \cite{El} on their derived categories and
we discuss their properties and relations with our geometric constructions.
In Section~\ref{subset: N generalized} we discuss problems analogous to the ones
considered above for the fixed loci of symplectic involutions on higher dimensional
 manifolds of $K3^{[n]}$-type: we propose a notion of $n$-generalized Nikulin surfaces,
which are surfaces appearing as components of fixed loci of symplectic involutions on
some  manifold of $K3^{[n]}$-type for some $n\geq 3$, and we prove interesting properties in the case $n=3$. This construction has many aspects in common with the construction of generalized Nikulin surfaces in the case $n=2$, especially from a lattice-theoretic point of view.

\subsection*{Acknowledgements}

We would like to dedicate this paper to B. van Geemen, whom we warmly thank for his invaluable suggestions, support, and discussions.

CC and AG are members of INdAM-GNSAGA and are supported by the Prin 2020 project ``Curves, Ricci flat Varieties and their Interactions''. CC is also supported by the Prin 2022 project ``Symplectic varieties: their interplay with Fano manifolds and derived categories''.
GK is supported by the project Narodowe Centrum Nauki 2024/53/B/ST1/00161.
MK is supported by the project Narodowe Centrum Nauki 2024/53/B/ST1/01413.

\section{Preliminaries}\label{preliminaries}\label{sec: preliminaries}
The aim of this section is to introduce notation and recall some known results, useful in the following.
\subsection{Lattices}
We recall that a lattice $(L,b)$ is a free finitely generated $\mathbb{Z}$-module $L$
endowed with a symmetric non-degenerate bilinear form $b\colon L\times L\to \mathbb{Z}$.
We say that $L$ is even if $b(\ell,\ell)\in 2\mathbb{Z}$ for every $\ell\in L$. The
discriminant group of $L$ is $A_L:=L^{\vee}/L$, where
$L^{\vee}=\mathrm{Hom}(L,\mathbb{Z})$ is isomorphic to
$$\{m\in L\otimes \mathbb{Q}\mid b(m,\ell)\in \mathbb{Z}\ \forall\,\ell\in L\}.$$
The bilinear form on $L$ induces, by $\mathbb{Q}$-linear extension, a bilinear form on
$L^{\vee}$ and hence on $L^{\vee}/L$. This endows $A_L=L^{\vee}/L$ with a quadratic
form, called the discriminant form. Observe that $A_L$ is a product of finite cyclic
groups and its order equals the discriminant $d(L):=|\det(B)|$, where $B$ is a matrix
representing $b$ with respect to a basis of $L$. A unimodular lattice is a lattice
which is isometric to its dual, or equivalently, a lattice with $d(L)=1$. The
signature of a lattice is the signature of the $\mathbb{R}$-linear extension of $b$
to $L\otimes \mathbb{R}$.

In the following we denote by:
\begin{itemize}
\item $U$ the unique even unimodular lattice of rank $2$ and signature $(1,1)$;
\item $E_8$ the unique even unimodular positive definite lattice of rank $8$;
\item $E_7$ the unique even positive definite lattice with discriminant group
$\mathbb{Z}/2\mathbb{Z}$, which is generated by its roots (i.e.~ by vectors on which
the bilinear form takes the value $2$); it is the orthogonal complement of a vector
with self-intersection $2$ in $E_8$;
\item $L(n)$, given a lattice $L$ and an integer $n\in\mathbb{Z}$, the lattice obtained
by multiplying the bilinear form of $L$ by $n$;
\item $u(n)$, for $n\in\mathbb{Z}$, the discriminant form of $U(n)$; for each
$m\in \mathbb{Z}$ and $\alpha\in \mathbb{Q}$, $\mathbb{Z}/m\Z(\alpha)$ is the
discriminant cyclic group $\mathbb{Z}/m\Z$ endowed with the quadratic form taking value
$\alpha$ on a generator; for short, the discriminant quadratic form of
$\mathbb{Z}/m\Z\!\left(\pm\tfrac{1}{m}\right)$ is denoted by $\left(\pm\tfrac{1}{m}\right)$;
\item $\Lambda_{K3}$ the unique even unimodular lattice of rank $22$ and signature
$(3,19)$; it is isometric to $U^{\oplus 3}\oplus E_8(-1)^{\oplus 2}$;
\item $\mathrm{div}(v)$, for $v\in L$, the generator of the ideal
$\{(v,w)\mid w\in L\}\subset \mathbb{Z}$.
\item We denote $q_k$ the finite quadratic form on $(\Z/2\Z)^{\oplus 2}$
\begin{eqnarray}\label{eq: form qk}q_k:=\left\lbrace\begin{array}{ll}
     \left(\frac{1}{2}\right)^{\oplus 2} &\mathrm{if}\ k\equiv 0\mod 4 \\
v(2)     & \mathrm{if}\ k\equiv 1\mod 4 \\
\left(-\frac{1}{2}\right)^{\oplus 2} &\mathrm{if}\ k\equiv 2\mod 4 \\
u(2)     & \mathrm{if}\ k\equiv 3\mod 4 
\end{array}\right.\end{eqnarray}
\item The lattice $K_t$ is the negative definite lattice represented by the matrix $\left[\begin{array}{cc}-\frac{t+1}{2}&1\\1&-2\end{array}\right]$ where $t\equiv 1\mod 2$
\item 
If $t=2k+1$, the discriminant quadratic form of $K_{t}(2)$ is $\Z/t\Z\left(-\frac{4}{t}\right)\oplus (-q_k)=\Z/t\Z\left(-\frac{t+1}{t}\right)\oplus (-q_k)$.

\end{itemize}

Given a lattice $L$ an overlattice $M$ of finite index of $L$ is a lattice such that there exists an injective map $f:L\hookrightarrow M$ which respects the bilinear forms and $M/f(L)$ is a finite group of order $r$. We will call $r$ the index of the inclusion. The finite index overlattices of $L$ are in 1:1 correspondence with the isotropic subgroups of $A_L$, see \cite{N}.

We will say that an inclusion of lattices $f:L\hookrightarrow M$ is primitive if $M/f(L)$ is a free $\Z$-module.
{If $R$ is a finite index overlattice of $L\oplus M$, then $R$ is generated by the generators of $L$ and $M$ and by some other vectors $v$, which are contained neither in $L$ nor in $M$ and are necessarily of the form $(\ell+m)/r$, $\ell\in L$, $m\in M$, $r\in \Z$. We call these vectors the \emph{gluing vectors} between $L$ and $M$.}

In the following we will apply the theory of lattices to the second cohomology groups of $K3$ surfaces, irreducible symplectic manifolds and varieties. The bilinear forms are defined as the cup product in the first case and as the Beauville--Bogomolov form in the other ones. We omit to denote explicitly the bilinear form, since it is clear from the context. We often call "self intersection" of $\ell$ the value $b(\ell,\ell)$.

\subsection{Symplectic involutions on $K3$ surfaces and Nikulin surfaces}\label{subsec: symplectic involution on K3s and Nikuli surfaces}
An automorphism $\alpha$ of a $K3$ surface $S$ is said to be symplectic, if its action on the symplectic form of $S$ is trivial. The quotient of a $K3$ surface by a symplectic automorphism is a singular surface, with ADE singularities, whose minimal resolution is another $K3$ surface.

In particular, if one considers a symplectic involution $\iota_S$ on a $K3$ surface $S$, then the quotient $S/\iota$ has 8 singular points of type $A_1$, which are the image of the 8 points fixed by $\iota_S$ on $S$. The minimal resolution contains 8 exceptional divisors $E_i$, which are rational curves. Therefore $E_i^2=-2$ and $E_iE_j=0$ if $i\neq j$.

Observe that, by blowing-up $S$ in the 8 points fixed by $\iota_S$, one obtains a non-minimal smooth surface $\widetilde{S}$ and $\iota_S$ lifts to an involution  $\widetilde{\iota_S}$ of $S$ which fixes the 8 exceptional divisors. The smooth surface  $\bar{S}=\widetilde{S}/\widetilde{\iota_S}$ is isomorphic to the minimal resolution of $S/\iota$ and the $2:1$ cover $\widetilde{S}\ra \widetilde{S}/\widetilde{\iota_S}$ is branched on the eight curves $E_i$. This implies that $(\sum_iE_i)/2\in \NS(\bar{S})$.
\begin{defi}\label{def:Nikulin surface and lattice}
A Nikulin surface $W$ is the minimal resolution of the quotient of a $K3$ surface $S$ by a symplectic involution $\iota_S$.
The \emph{Nikulin lattice} $N$ is the minimal primitive sublattice of $\NS(W)$ which contains all the exceptional curves arising from the desingularization of $S/\iota_S$.
Equivalently the Nikulin lattice $N$ is the lattice $\langle E_i,i=1,\ldots, 8, \sum_iE_i/2\rangle$ where $E_i^2=-2$, $E_iE_j=0$.
\end{defi}
The Nikulin lattice is a negative definite rank $8$ even lattice whose discriminant group is $(\Z/2\Z)^{\oplus 6}$ and whose discriminant form is $u(2)^{\oplus 3}$.

We recall that a $K3$ surface admits a symplectic involution if and only if the lattice $E_8(-2)$ is primitively embedded in its N\'eron--Severi lattice and in this case the symplectic involution acts as $-1$ on this copy of $E_8(-2)$ and as $+1$ on its orthogonal complement in the second cohomology group. Moreover, a projective $K3$ surface is a Nikulin surface if and only if its N\'eron--Severi lattice contains primitively the lattice $N$. 
In \cite[Corollary 2.2]{GS} the explicit relation between the N\'eron--Severi lattices (and the transcendental lattice) of projective $K3$ surfaces $S$ admitting a symplectic involution $\iota$ and its Nikulin surface $W$ are explicitly given.

We observe that the Nikulin $K3$ surfaces are dense in the moduli space of $K3$ surfaces, indeed all the Kummer surfaces are also Nikulin surfaces (see \cite{NikKumm}), and the Kummer surfaces are dense in the moduli space of $K3$ surfaces (see \cite{PS}). A different way to state the same result is that the $K3$ surfaces which are terminalization of the quotient of $K3$ surfaces by a symplectic involution are dense in the family of their deformations (deformations as $K3$ surfaces, and not just as terminalization of quotients). Even if the previous formulation is less natural, it is the one that we can generalize in higher dimension, indeed  we will prove that the Nikulin orbifolds (terminalizations of certain quotients) are dense in the family of orbifolds of Nikulin type (i.e.~ in the family of the deformations of Nikulin orbifolds), see Theorem \ref{Thm1.1} for the precise statement.
\subsection{Symplectic involutions on  manifolds of $K3^{[2]}$-type, Nikulin orbifolds and orbifolds of Nikulin type}\label{subsec: preliminarie on K32-type and Nikulin orbifold}
An irreducible holomorphic symplectic fourfold is said to be of $K3^{[2]}$-type if it is deformation equivalent to the Hilbert scheme of $2$ points on a $K3$ surface. If $S$ is a $K3$ surface, then $S^{[2]}= \Hilb^2(S)$ is isomorphic to the blow-up of the quotient of $S\times S$ by the involution $(12)$ switching the two copies of $S$. This quotient is singular along the image of the diagonal, and by blowing-up this singularity once, one obtains a smooth symplectic fourfold, isomorphic to $\Hilb^2(S)$. We will denote by $\Delta$ the exceptional divisor of the map $\Hilb^2(S)\to (S\times S)/(12)$.
Let $X$ be a  manifold of $K3^{[2]}$-type admitting a symplectic involution $\iota$. Then there exists a deformation of the pair $(X,\iota)$ to a pair $(S^{[2]},\iota_S^{[2]})$, where $S$ is a $K3$ surface admitting a symplectic involution $\iota_S$ and $\iota_S^{[2]}$ is the induced automorphism on $S^{[2]}$, see \cite{Mon}.

By \cite{Mon}, the fixed locus of $\iota$ on $X$ consists of 28 isolated points and a $K3$ surface, which we denote $F$ and call generalized Nikulin surface in the following. 
If one specializes $(X,\iota)$ to $(S^{[2]},\iota_S^{[2]})$ as explained above, one is able to explicitly describe the fixed locus of $\iota_{S}^{[2]}$ by relating it with the geometry of $S$ and $\iota_S$, and in particular one shows that the $K3$ surface in the fixed locus of $\iota_S^{[2]}$ on $S^{[2]}$ is the minimal resolution of $S/\iota_S$ and in particular is a Nikulin surface (see \cite{Mon}).

In analogy with what happens for $K3$ surfaces, a  manifold of $K3^{[2]}$-type admits a symplectic involution if and only if $E_8(-2)$ is primitively embedded in its N\'eron--Severi lattice and this depends on the fact that the action induced by a symplectic involution on the second cohomology group is essentially unique, see \cite{Mon}. 

More precisely, the second cohomology group of any manifold of $K3^{[2]}$-type endowed with the Beauville--Bogomolov-- form is isometric to $U^{\oplus 3}\oplus E_8(-1)^{\oplus 2}\oplus \langle -2\rangle$ and this lattice will be denoted by $\Lambda_{K3^{[2]}}$. 
Let $I\in O(\Lambda_{K3^{[2]}})$ be the involution which switches the two $E_8(-1)$ summands and let $\Lambda_{K3^{[2]}}^I$ be its invariant sublattice inside $\Lambda_{K3^{[2]}}$.
We denote  $\mathcal{M}_{K3^{[2]},I}$ the coarse moduli space of  marked pairs $(X,\phi)$ such that $X$ is a fourfold of $K3^{[2]}$-type and $\phi:H^2(X,\Z)\rightarrow \Lambda_{K3^{[2]}}$ is a marking such that there exists a symplectic involution $\iota\in\Aut(X)$ satisfying $\phi\circ\iota^*\circ\phi ^{-1}=I$. The period map is the holomorphic map \[\mathcal{P}_{K3^{[2]},I}:\mathcal{M}_{K3^{[2]},I}\rightarrow D_{\Lambda_{K3^{[2]}}^I}:=\{x\in\mathbb{P}(\Lambda_{K3^{[2]}}^I\otimes \mathbb{C})\mid q_{K3^{[2]}}(x)=0,\ B_{K3^{[2]}}(x,\overline{x})>0\}
\]
such that $(X,\phi)\mapsto [\phi_{\mathbb{C}}(H^{2,0}(X))]$; it is the restriction of the period map $\mathcal{P}_{K3^{[2]}}$ of fourfolds of $K3^{[2]}$-type. We observe that the image of $\mathcal{P}_{K3^{[2]},I}$ contains the image of very general marked pairs $(X,\phi)$ such that $\NS(X)\simeq E_8(-2)$, which is  $D_{\Lambda_{K3^{[2]}}^I}\setminus \coprod_{w\in \Lambda_{K3^{[2]}}^I\setminus \{0\}}w^\perp$. Work of Mongardi (see \cite{Mon,Mon2}) implies the following:
\begin{thm}\label{thm:surjectivity-K3[2]-period-map}
The period map $\mathcal{P}_{K3^{[2]},I}:\mathcal{M}_{K3^{[2]},I}\rightarrow D_{\Lambda_{K3^{[2]}}^I}$ is surjective.
\end{thm}
\begin{proof}
We give a sketch of the proof for lack of an explicit reference, in analogy with Nikulin's proof  for K3 surfaces (see \cite[Theorem 4.15]{Nikulin-FiniteGroups}).

Let $p\in D_{\Lambda_{K3^{[2]}}^I}$ and let $(X,\phi)\in \mathcal{P}_{K3^{[2]}}^{-1}(p)$. As before, $\phi^{-1}(E_8(-2))\subset\NS(X)$ is primitive and $g=\phi^{-1}I\phi$ is a Hodge isometry of $H^2(X,\mathbb{Z})$ acting trivially on $H^{2,0}(X)$. It fixes the positive $3$-space contained in $\phi^{-1}(\Lambda_{K3^{[2]}}^I)_\mathbb{R}$, so $g$ is a Hodge monodromy.

As already observed in \cite{Mon}, $E_8(-2)$ contains no wall-divisors, i.e. classes of square $-2$ or of square $-10$ and divisibility $2$. As a consequence, the set of wall divisors $\Delta(X)\subset \NS(X)$ is the union of the set $\Delta(X)\cap \phi^{-1}(\Lambda_{K3^{[2]}}^I)$ of invariant wall-divisors  and of the set $\Delta(X)'$ of wall-divisors $D$ which can be written $D=D_++D_-$ with $D_+\in \Lambda_{K3^{[2]}}^I\otimes \Q$, $D_-\in (\Lambda_{K3^{[2]}}^I)^\perp \otimes \Q$ and $q_{K3^{[2]}}(D_-)<0$, such that $g(D)=D_+-D_-\in\Delta(X)'$. Hence, there is a $g$-invariant class $\alpha$ in the positive cone of $X$ lying on no wall, and its K\"ahler-type chamber $K$ is $g$-invariant.

By global Torelli theorem \cite{MarkmanTorelli}, there exists $(X',\phi')\in \mathcal{P}_{K3^{[2]},I}^{-1}(p)$ such that its K\"ahler cone $\mathcal{K}_{X'}$ satisfies $\phi'(\mathcal{K}_{X'})=K$. Then $\phi'^{-1}g\phi'$ is a Hodge monodromy of $X'$ preserving a K\"ahler class, so by the Hodge theoretic Torelli theorem \cite{MarkmanTorelli} it equals $\sigma^{*}$ for a biregular symplectic involution $\sigma$ of $X'$.     
\end{proof}

By the description of the fixed locus of $\iota$ on $X$, it follows that the quotient $X/\iota$ is singular in 28 points and a surface. There exists a symplectic resolution of the singular surface, obtained by blowing it up once. We denote $Y$ such a partial resolution, which is a terminalization of $X/\iota$, so that $Y$ has 28 singular points and is an irreducible symplectic orbifold (see \cite{Fuj}). 

By \cite{Menet}, the second cohomology group of $Y$ is endowed with a Beauville--Bogomolov quadratic form  and it is isometric (as abstract lattice) to  $U(2)^{\oplus 3}\oplus E_8(-1)\oplus \langle -2\rangle^{\oplus 2}$ and the class of the exceptional divisor of the partial resolution $Y\ra X/\iota$ is a divisor $\Sigma$ such that $\Sigma^2=-4$, $\divi(\Sigma)=2$ and $\Sigma^\perp\simeq \mathbb{T}$. We will call $\Lambda_N$ the lattice $U(2)^{\oplus 3}\oplus E_8(-1)\oplus \langle -2\rangle^{\oplus 2}$ and we will denote by $q_{\Lambda_N}$ (respectively $B_{\Lambda_N}$) its quadratic form (respectively its bilinear form).
In \cite[Table 3]{CGKK}, the relation between the N\'eron--Severi lattices (and the transcendental lattices) of a projective  manifold of $K3^{[2]}$-type $X$ with a symplectic involution $\iota$, and of $Y$, the Nikulin orbifold terminalization of $X/\iota$, is explicitly given.

Observe that not all the deformations of $Y$ are obtained as quotient of a  manifold of $K3^{[2]}$-type by a symplectic involution,  indeed not all the deformations of $Y$ have an algebraic class of self intersection $-4$, divisibility 2 and orthogonal lattice isometric to $\mathbb{T}$. The orbifolds obtained as deformation of $Y$ are called orbifold of Nikulin type and we will identify the ones which are Nikulin orbifold (i.e.~ terminalizations of quotients of manifolds of $K3^{[2]}$-type by a symplectic involution) in Theorem \ref{Thm1.1}.

\subsection{Moduli spaces of irreducible symplectic orbifolds}\label{subsec:moduli-orbifolds}
Since orbifolds of Nikulin type are irreducible symplectic orbifolds (and also primitive symplectic varieties), there exists a coarse moduli space $\mathcal{M}_{\Lambda_N}$ of marked orbifolds of Nikulin type $(Y,\eta)$, which maps holomorphically to the period domain associated to the Beauville--Bogomolov  lattice $\Lambda_N$ \cite{BL, Menet}.

Indeed, the set $$\mathcal{M}_{\Lambda_N}:=\lbrace (Z,\eta)\mid Z\ \mathrm{is\  of\ Nikulin\ type,}\ \eta:H^2(Z,\mathbb{Z})\rightarrow \Lambda_N\ \mathrm{is\ an\ isometry}\rbrace/\simeq$$
has a natural structure of a complex space given by the local Torelli theorem as described in \cite{Menet}.

 The period map is the holomorphic map
\[
\mathcal{P}:\mathcal{M}_{\Lambda_N}\longrightarrow D_{\Lambda_N}:=\{x\in\mathbb{P}(\Lambda_N\otimes \mathbb{C})\mid q_{\Lambda_N}(x)=0,\ B_{\Lambda_N}(x,\overline{x})>0\},\ 
(Z,\eta)\mapsto [\eta_{\mathbb{C}}(H^{2,0}(Z))]
\]
where $\eta_{\mathbb{C}}$ is the $\mathbb{C}$-linear extension of $\eta$.
By work of Menet \cite[Theorem 1.1]{MenetTorelli}, the global Torelli theorem is known to hold for irreducible symplectic orbifolds; in particular, $\mathcal{P}$ is surjective and a local isomorphism on each connected component (see \cite[Proposition 5.8, Theorem 5.9]{MenetTorelli}). Moreover, a Hodge-theoretic version of the Torelli theorem is known by \cite[Theorem 1.1]{MenetRiess}.

Given $(Z,\eta)\in\mathcal{M}_{\Lambda_N}$, we denote by $\mathrm{Mo}^2(Z)$ the group of monodromy operators acting on $H^2(Z,\mathbb{Z})$; this group acts on $\mathcal{M}_{\Lambda_N}$ via $(Z,\eta)\mapsto (Z,\eta\circ\gamma)$. We also denote by $\mathrm{Mo}^2_{Hdg}(Z)$ the subgroup of monodromy operators which are also Hodge isometries. For any irreducible symplectic orbifold $Z$ we denote respectively $\mathcal{K}_Z$, $\mathcal{BK}_Z$ and $\mathcal{FE}_Z$ the K\"ahler cone, the birational K\"ahler cone and the fundamental exceptional chamber of $Z$.

Let $p\in D_N$ be a point and let $(Z,\eta)\in\mathcal{P}^{-1}(p)$. Consider $\mathcal{KT}(Z)$ the set of K\"ahler-type chambers in $\mathcal{C}_Z$, i.e.~ subsets $g(f^*\mathcal{K}_W)$, where $g\in\mathrm{Mo}^2_{Hdg}(Z)$ and $f:Z\dashrightarrow W$ is a bimeromorphic map to another orbifold of Nikulin type $W$. For any connected component $\mathcal{M}^0_{\Lambda_N}$, the map
\[
\rho_p: \mathcal{P}^{-1}(p)\cap \mathcal{M}^0_{\Lambda_N}\rightarrow \mathcal{KT}(Z),\ (W,\varphi)\mapsto \eta^{-1}(\varphi(\mathcal{K}_W))
\]
is well-defined because if there exists a parallel transport operator $g:H^2(W_1,\mathbb{Z})\rightarrow H^2(W_2,\mathbb{Z})$ which is a Hodge isometry then there exists a bimeromorphic map $h:W_1\dashrightarrow W_2$ (see \cite[Corollary 5.11]{MenetTorelli}).

\begin{prop}\label{prop:equivariant-bijection}
The map $\rho_p$ is a $\mathrm{Mo}^2_{Hdg}(Z)$-equivariant bijection.
\end{prop}
\begin{proof}
The proof is analogous to the one of \cite[Proposition 5.14]{MarkmanTorelli}. Indeed, if $\rho_p(W_1,\varphi_1)=\rho_p(W_2,\varphi_2)$, then $\varphi_2^{-1}\varphi_1(\mathcal{K}_{W_1})=\mathcal{K}_{W_2}$. As a consequence, $\varphi_2^{-1}\circ\varphi_1$ is a parallel transport operator which is a Hodge isometry and which sends a K\"ahler class into a K\"ahler class. Hence by \cite[Theorem 1.1]{MenetRiess} $ \varphi_2^{-1}\circ\varphi_1=f_*$ for $f:(W_1,\varphi_1)\rightarrow (W_2,\varphi_2)$ an isomorphism, and $\rho_p$ is injective.

By definition, each K\"ahler-type chamber $K$ is contained in an exceptional chamber $g(\mathcal{FE}_Z)$ for some $g\in\mathrm{Mo}^2_{Hdg}(Z)$; since the action of $\mathrm{Mo}_{Hdg}^2(Z)$ on the set of exceptional chambers is transitive, given a very general $\alpha\in K$ there exists $g\in\mathrm{Mo}^2_{Hdg}(Z)$ such that $g(\alpha)\in \mathcal{FE}_Z$. It follows from \cite[Proposition 4.18]{MenetRiess} that $g(\alpha)\in\mathcal{BK}_Z$ and there exists a bimeromorphic map $h:Z\dashrightarrow W$ such that $h_*(g(\alpha))\in\mathcal{K}_W$. This implies that $h_*\circ g(K)=\mathcal{K}_W$ and $\rho_p(W, \eta\circ g^{-1}\circ h^*)=K$, hence $\rho_p$ is surjective. 

Equivariance is obvious from the definition. \end{proof}

Recent work by Brandhorst, Menet and Muller \cite[Theorem 1.1]{BMM} shows that the monodromy group of orbifolds of Nikulin type is $\mathrm{Mo}^2(\Lambda_N)=O^+(\Lambda_N)$, the arithmetic subgroup of orientation preserving isometries. Since it is a normal subgroup, it does not depend on the choice of the connected component of $\mathcal{M}_{\Lambda_N}$.

As a corollary we obtain the following version of the Torelli theorem for Nikulin orbifolds.
\begin{thm}\label{prop:connected components of moduli}
The moduli space $\mathcal{M}_{\Lambda_N}$ has two connected components, $\mathcal{M}_{\Lambda_N}^+$ and $\mathcal{M}_{\Lambda_N}^-$, exchanged by $-\id$. The period map is generically $1:1$ restricted to each of these components.
\end{thm}
\begin{proof}
The proof is essentially the same as in the case of $K3$ surfaces (see \cite[Proposition 7.5.4]{HuybrechtsK3}) and it is the same for all deformation families in which $\mathrm{Mo}^2(\Lambda)=O^+(\Lambda)$ (cf. \cite[note added in proof]{Onorati}).
\end{proof}
We will denote respectively by $\mathcal{P}_+$ and $\mathcal{P}_-$ the restrictions of the period map to $\mathcal{M}_{\Lambda_N}^+$ and $\mathcal{M}_{\Lambda_N}^-$. It follows from \cite{MenetTorelli} and Proposition \ref{prop:equivariant-bijection} that $\mathcal{P}_+$ and $\mathcal{P}_-$ are local isomorphisms with $D_{\Lambda_N}$.


\section{The main results}\label{section: The main results}
The aim of this section is to present more precisely our three main theorems, which will be proved in the next sections.

We recalled in Section \ref{subsec: symplectic involution on K3s and Nikuli surfaces} that Nikulin surfaces are characterized by a lattice-theoretic condition, i.e.~ the presence of the Nikulin lattice in the N\'eron--Severi lattice. We state the analogous result for  Nikulin orbifolds inside the family of orbifolds of Nikulin type. Also the density result is the analogue in higher dimension of the density of Nikulin surfaces in the family of $K3$ surfaces, see Section \ref{subsec: symplectic involution on K3s and Nikuli surfaces}.

\begin{thm}\label{Thm1.1}
    A marked orbifold $(Y,\eta)$ of Nikulin type is bimeromorphic to a Nikulin orbifold if and only if its N\'eron--Severi lattice $\NS(Y)$ contains primitively a class of square $(-4)$, divisibility 2 in $\Lambda_N$ and orthogonal lattice isometric to $\mathbb{T}$.
    
    Moreover, the locus $\mathcal{M}_{O\cdot\mathbb{E}}$ of marked orbifolds of Nikulin type which are bimeromorphic to a Nikulin orbifold is dense inside any connected component of $\mathcal{M}_{\Lambda_N}$.
\end{thm}

\begin{defi}
Let $X$ be a  manifold of  $K3^{[2]}$-type admitting a symplectic involution $\iota$. We call generalized Nikulin surface of $(X,\iota)$ the $2$-dimensional component of $$\Fix_{\iota}(X)=\{x\in X\mbox{ such that }\iota(x)=x\}.$$ A generalized Nikulin surface is a surface that is a generalized Nikulin surface of $(X,\iota)$, for some  manifold  of  $K3^{[2]}$  type $X$ with symplectic involution $\iota$, i.e.~ which appears as a fixed surface of the action of a symplectic involution on a  manifold  of  $K3^{[2]}$-type.
\end{defi}

Given a manifold of $K3^{[2]}$-type $X$ with a symplectic involution $\iota$, there are two symplectic varieties naturally associated to $(X,\iota)$, which are the generalized Nikulin surface of $(X,\iota)$ and the Nikulin orbifold of $(X,\iota)$, i.e.~ the terminalization of $X/\iota$. The first is smooth and $2$-dimensional, the latter singular and $4$-dimensional, but both of them are endowed with a symplectic structure induced by the one of $X$.
We are now able to prove the conjecture that we presented in \cite[Conjecture 3.12]{CGKK} and which relates the Hodge structures of these two symplectic varieties. 
\begin{thm}\label{Thm1.2} Let $X$ be a  manifold of $K3^{[2]}$-type with symplectic involution $\iota$, $F$ the generalized Nikulin surface of $(X,\iota)$ and $Y$ the Nikulin orbifold of $(X,\iota)$. Then the transcendental lattice $\T_F$ of the K3 surface $F$ is Hodge isometric to the transcendental lattice $\T_Y$ of the Nikulin orbifold $Y$.
\end{thm}

In order to prove Theorem \ref{Thm1.2} and to study the moduli space of generalized Nikulin surfaces,  we are interested in finding a lattice-theoretic characterization of these $K3$ surfaces. For this reason we introduce an appropriate sublattice of the Nikulin lattice introduced in Definition \ref{def:Nikulin surface and lattice}.

\begin{defi}\label{defi: generalizad Nikulin orbifold}
The generalized 
Nikulin lattice is the sublattice $G\subset N$ that is the orthogonal complement to the class $\frac{1}{2}(\sum_i E_i)$ in the Nikulin lattice $N$.\end{defi}

We show in Proposition \ref{lemma: properties of G} that $G$ is isometric to $E_7(-2)$.
We  prove Theorem \ref{Thm1.2} in Section \ref{section: proofs of 2nd theorem and definition of Generalized Nikulin lattice}, and the proof is  based on the fact that the generalized Nikulin lattice is primitively embedded in the N\'eron--Severi lattice of every generalized Nikulin surface. Also the vice versa is true so that the generalized Nikulin lattice identifies the $K3$ surfaces which are generalized Nikulin surfaces, by the following result, proved in Section \ref{section: proofs of 2nd theorem and definition of Generalized Nikulin lattice}.

\begin{thm}\label{thm: projective $K3$ G polarized are generalized Nikulin}
A $K3$ surface is a generalized Nikulin surface if and only if the generalized Nikulin lattice can be primitively embedded in its N\'eron--Severi lattice.
\end{thm}

In Section \ref{Families of projective generalized Nikulin surfaces} we  describe the family of projective generalized Nikulin surfaces, which has countably many components, each of dimension 12. Recall that there are two varieties, $X$ and $Y$ attached to $F$: $X$ is the  manifold of $K3^{[2]}$-type admitting a symplectic involution $\iota$ such that $\Fix_{\iota}(X)\supset F$ and $Y$ is the Nikulin orbifold terminalization of $X/\iota$. The three families of generalized Nikulin surfaces,  manifolds of $K3^{[2]}$-type admitting a symplectic involution and  Nikulin orbifolds are described as families of lattice-polarized varieties. Therefore, in order to explicitly relate $F$, $X$ and $Y$, we prove that the choice of the N\'eron--Severi group of a generalized Nikulin surface $F$ with Picard number 8 determines uniquely $(\NS(X), \T_X)$ and $(\NS(Y), \T_Y)$ and viceversa, see Proposition \ref{prop:relation NS(F), NS(X), NS(W)}.

In Section \ref{section:projective models} we describe projective models of projective generalized Nikulin surfaces. This generalizes the results proved in \cite{GS} for the projective Nikulin surfaces.

\section{Proof of Theorem \ref{Thm1.1}}\label{section: proofs of 1st theorem}\label{2}

The aim of this section is to give a proof of Theorem \ref{Thm1.1}. We use the construction of the moduli space of marked orbifolds of Nikulin type recalled in Section \ref{subsec:moduli-orbifolds}.
Before we recall some lattice-theoretic results needed in the following.
\begin{defi}\label{def:T}
We will call $\mathbb{T}$ the abstract lattice $\mathbb{T}:=U(2)^{\oplus 3}\oplus E_8(-1)\oplus\langle -4\rangle$ and $\mathbb{E}$ a vector in $\Lambda_N$ such that $\mathbb{E}^2=-4$, $\rm{div}(\mathbb{E})=2$ and $\mathbb{E}^\perp\simeq \mathbb{T}$. 
\end{defi}

\begin{lemma}\label{lemma: embeddings of <-4> and T in LambdaN}
\begin{enumerate}
\item The map $O(\langle -4\rangle)\ra O(A_{\langle -4\rangle})$ is surjective.
\item The lattice $\mathbb{T}$ admits a unique, up to isometries, primitive embedding in $\Lambda_N$.
\item The lattice $\mathbb{T}$ is unique in its genus.
\item The map $O(\mathbb{T})\ra O(A_{\mathbb{T}})$ is surjective.
\item
There exists a unique, up to isometries, primitive embedding $\psi_{\langle-4\rangle}$ of the lattice $\langle -4\rangle=\Z h$ such that $\psi_{\langle -4\rangle}(h)=\mathbb{E}$ has divisibility 2 in $\Lambda_N$ and $\psi_{\langle-4\rangle}(h)^\perp\simeq \mathbb{T}$.

\end{enumerate}
\end{lemma}
\begin{proof}
(1) The group  $O(A_{\langle -4\rangle})$ is $\{\pm \mathrm{Id}\}$ as well as $O(\langle -4\rangle)$.

(2) To determine the primitive embeddings, up to isometries, of $\mathbb{T}$ in $\Lambda_N$, one has to apply \cite[Proposition 1.15.1]{N}: an embedding of $\mathbb{T}$ in $\Lambda_N$ is determined by the choice of $(H_S,H_q,\gamma,K,\gamma_K)$  as in \cite[Proposition 1.15.1]{N}, where $S=\mathbb{T}$ and $q=q_{A_{\Lambda_N}}$. The discriminant quadratic forms are $q_{A_{\mathbb{T}}}= u(2)^{\oplus 3}\oplus \left(-\frac{1}{4}\right)$ and $q_{A_{\Lambda_N}}=u(2)^{\oplus 3}\oplus \left(-\frac{1}{2}\right)^{\oplus 2}$; call $\delta_1$ and $\delta_2$ the two last generators of the previous form.
Since $\rk \mathbb{T}^\perp=1$, the length $\ell(A_{T^\perp})\leq 1$ and this yields $H_S\simeq (\mathbb{Z}/2\mathbb{Z})^{\oplus 7}$: indeed, the order of $H_S$ cannot exceed $2^7$ since otherwise it would contain an element of order $4$ and there are no elements of order $4$ inside $A_{\Lambda_N}$; moreover,  if the order $|H_S|\leq 2^6$, then $|A_{T^\perp}|\geq 2^4$ and this would imply $\ell(A_{T^\perp})\geq 2$.

As a consequence, $H_S=A_{U(2)^{\oplus 3}}\oplus \langle 2h\rangle$, where $h$ is the generator of $\mathbb{Z}/4\mathbb{Z}\left(-\frac{1}{4}\right)$, and its $2$-adic form has parity $\delta=0$. The only two possible subgroups $H_q\subset A_{\Lambda_N}$ of order $2^7$ and with forms of parity $\delta=0$ are $A_{U(2)^{\oplus 3}}\oplus \langle\delta_1\pm\delta_2\rangle$, and these two choices are isometric.

Finally, one computes that $q_{A_K}=\left(-\frac{1}{4}\right)$; it follows from uniqueness of $\langle -4\rangle$ in its genus and from (1) that $(K,\gamma_K)$ is unique up to isometry.

(3) The lattice $\mathbb{T}$ is even of signature $(3,12)$ and discriminant form $q_{A_T}=u(2)^{\oplus 3}\oplus \left(-\frac{1}{4}\right)$. This determines a unique lattice up to isometries, which is isometric to $\mathbb{T}$.

(4)  The map $O(\mathbb{T})\ra O(A_{\mathbb{T}})$ is surjective by \cite[Theorem 1.14.2]{N}. 

(5) Observe $A_{\Z\mathbb{E}}=\langle \mathbb{E}/4\rangle$ and $q_{A_{\Lambda_N}}=u(2)^{\oplus 3}\oplus \left(-\frac{1}{2}\right)^{\oplus 2}$.  Then the embedding of $\mathbb{E}$ in $\Lambda_N$ is determined by the choice of $(H_S,H_q,\gamma,K,\gamma_K)$  as in \cite[Proposition 1.15.1]{N}, where $S=\mathbb{E}$. There is a unique choice for the subgroup $H_S$, since it has to be cyclic of order 2. Therefore we put $H_S=\langle 2\frac{\mathbb{E}}{4}\rangle$. This shows the existence of the primitive embedding $\psi_{\langle -4\rangle}$, it suffices to choose $H_q=\langle\delta_1+\delta_2\rangle$, isomorphic to $H_S$. Indeed, one directly computes that $q_{A_K}$ is $u(2)^{\oplus 3}\oplus \mathbb{Z}/4\mathbb{Z}\left(-\frac{1}{4}\right)$. The orthogonal of $\Z\mathbb{E}$ is a lattice of signature $(3,12)$ and discriminant form $q_{A_K}$. This determines a unique lattice up to isometries, which is isometric to $\mathbb{T}$ by (3). The map $O(\mathbb{T})\ra O(A_{\mathbb{T}})$ is surjective by (4), hence $(K,\gamma_K)$ are unique up to isometries. In order to show that $\psi_{\langle -4\rangle}$ is unique up to isometries, since $H_S$ is the unique subgroup of order 2 in $A_{\mathbb{Z}E}$, we need to prove that any other choice of $H'_q\subset A_{\Lambda_N}$ such that $H'_q=\langle x\rangle$ with $q_{A_{\Lambda_N}}(x)=1$ and such that $q_{A_K}=q_{A_{\mathbb{T}}}$ is isometric to the choice we have made above. Indeed, 
we consider two cases: either $x\in\mathbb{Z}\delta_1\oplus\mathbb{Z}\delta_2$ or not. In the former case, since $q_{A_{\Lambda_N}}(x)=1$, it must be $x=\delta_1\pm\delta_2$, and we computed $q_{A_K}=q_{A_{\mathbb{T}}}$; these two choices are clearly isometric.

In the latter case, we have $x=y+a_1\delta_1+a_2\delta_2$ with $a_i\in\mathbb{Z}/2\mathbb{Z}$ for $i=1,2$ and $y\in A_{U(2)^{\oplus 3}}$ non trivial, hence either $q_{A_{\Lambda_N}}(y)=0$ or $q_{A_{\Lambda_N}}(y)=1$. In the former case, $q_{A_{\Lambda_N}}(x)=1$ implies $a_1=\pm a_2=1$ and one computes that $A_K$ contains the class of $w+\delta_1$, which is a class of square $- \frac{1}{2}$, where $w\in A_{U(2)^{\oplus 3}}$ such that $q_{A_{\Lambda_N}}(w)=0$ and $b_{A_{\Lambda_N}}(y,w)=\frac{1}{2}$. In the latter case, $q_{A_{\Lambda_N}}(x)=1$ implies $x=y$ and $A_K$ contains the class of $\delta_1$. As a consequence, in all these cases the $2$-adic component of $A_{K}$ has discriminant form with parity $\delta=1$, while the $2$-adic component of $A_{\mathbb{T}}$ has discriminant form with parity $\delta=0$. We proved in this way that there is a unique orbit, up to isometry, of primitive embeddings as in (5).

\end{proof}

\begin{rem}
    It is shown in \cite[Theorem 6.15]{MenetRiess} that there are at most three monodromy orbits of classes of square $-4$ and divisibility $2$ inside $\Lambda_N$. The direct computation of orthogonal sublattices in the three cases gives the same result as in Lemma \ref{lemma: embeddings of <-4> and T in LambdaN} (5).
\end{rem}

\begin{rem}\label{rem: explicitl psi}{\rm Let us denote $\lambda_i$, $i=1,\ldots, 16$ the element of a basis of $\Lambda_N$ on which the bilinear form is $U(2)^{\oplus 3}\oplus E_8(-1)\oplus\langle -2\rangle\oplus \langle -2\rangle$. Since the embeddings of $\mathbb{Z}\mathbb{E}$ and $\mathbb{T}$ in $\Lambda_N$ are essentially unique, we can fix them as follows:
$\psi_{\langle-4\rangle}(\mathbb{E})=\lambda_{15}-\lambda_{16}$  and $\psi_{\mathbb{T}}(\mathbb{T})$ is generated by $\lambda_{i}$, $i=1,\ldots 14$ and $\lambda_{15}+\lambda_{16}$.}
\end{rem}
We  consider the codimension 1 locus
\[
D_{O\cdot\mathbb{E}}:=(\bigcup_{\begin{array}{c}
 q_{\Lambda_N}(\delta)=-4, \divi \delta=2,      \\
  \delta^\perp\simeq \mathbb{T}
\end{array}}\delta^\perp)\cap D_{\Lambda_N}.
\]
given by hyperplanes in $D_{\Lambda_N}$ which are orthogonal to a $\delta$ in the isometry orbit $O\cdot \mathbb{E}$ of $\mathbb{E}$.
\begin{lem}\label{lemma:density of periods of quotients}

The hyperplane arrangement $D_{O\cdot\mathbb{E}}$ is dense inside $D_{\Lambda_N}$ in the Euclidean topology. 
\end{lem}
\begin{proof}
The density is shown in \cite[Proposition 3.2 and Remark 3.12]{AnaninVerbitsky} (see also \cite{MarkmanMehrotra}).
\end{proof}

We denote $\mathcal{M}_{O\cdot\mathbb{E}}:=\mathcal{P}^{-1}(D_{O\cdot\mathbb{E}})=\mathcal{P}_+^{-1}(D_{O\cdot\mathbb{E}})\cup\mathcal{P}_-^{-1}(D_{O\cdot\mathbb{E}})\subset \mathcal{M}_{\Lambda_N}$. 

\begin{prop}\label{prop:connected components of (-4,2) locus}
Given a very general $\omega\in D_{O\cdot\mathbb{E}}$ such that if $(Y,\eta)\in\mathcal{P}^{-1}(\omega)$ then $\eta(\NS(Y))=\mathbb{Z}\delta\simeq \langle -4\rangle$, there is a $1:1$ correspondence between $\mathcal{P}^{-1}(\omega)$ and the set $\{\mathcal{K}_Y, R_\delta(\mathcal{K}_Y)\}\times \{\pm 1\}$, which is equivariant for the action of Hodge monodromies.

\end{prop}
\begin{proof} 
For a very general period $\omega\in D_{O\cdot\mathbb{E}}$, if $(Y,\eta)\in \mathcal{P}^{-1}(\omega)$, then $\eta(\NS(Y))=\mathbb{Z}\delta\simeq \langle -4\rangle$. On the other hand, by \cite{MenetRiess, BMM} $\delta$ is a prime exceptional divisor, hence the positive cone of $Y$ splits into two chambers, $\mathcal{K}_Y$ and $R_\delta(\mathcal{K}_Y)$, where $\mathcal{K}_Y$ denotes the K\"ahler cone of $Y$ and $R_\delta\in O(\Lambda_N)$ is the reflection in $\delta$. By \cite[Thm 5.16]{MarkmanTorelli} and its generalization to irreducible symplectic orbifolds given in Proposition \ref{prop:equivariant-bijection}, there is a $1:1$ correspondence between $\mathcal{P}^{-1}(\omega)$ and the set $\{\mathcal{K}_Y, R_\delta(\mathcal{K}_Y)\}\times \{\pm 1\}$, which is equivariant for the action of Hodge monodromies. 
\end{proof}

\begin{cor}\label{cor:density}
The moduli space $\mathcal{M}_{O\cdot\mathbb{E}}$ is dense inside $\mathcal{M}_{\Lambda_N}$.
\end{cor}

Let us fix some notation before showing Theorem \ref{Thm1.1}. Given a $K3^{[2]}$-type manifold $X$ with a symplectic involution $\iota$, we denote $\pi \colon X\ra X/\iota$ the quotient map and $\rho:Y\ra X/\iota$ the terminalization which is the blow up of the singular surface and which introduces the exceptional divisor $\Sigma$.

\begin{proof}[Proof of Theorem \ref{Thm1.1}] Any marked Nikulin orbifold $(Y,\eta)$ satisfies $\mathcal{P}(Y,\eta)\in D_{O\cdot\mathbb{E}}$ because of what we observed in Section \ref{subsec: preliminarie on K32-type and Nikulin orbifold}, and the very general one satisfies $\NS(Y)=\mathbb{Z}\Sigma\simeq\langle -4\rangle$.
We will show that, for 
any marked pair $(Z,\eta)\in \mathcal{M}_{O\cdot\mathbb{E}}$, the orbifold $Z$ is bimeromorphic to a Nikulin orbifold $Y$. 

Under the assumption that $(Z,\eta)\in \mathcal{M}_{O\cdot\mathbb{E}}$, the period $\mathcal{P}(Z,\eta)\in \delta^\perp\subset\mathcal{D}_{O\cdot\mathbb{E}}$ for some $\delta\in\Lambda_N\otimes \mathbb{C}$ such that $\delta^\perp$ is isometric to the lattice $\mathbb{T}$.

Let $I$ be the isometry  of order two of $\Lambda_{K3^{[2]}}$ recalled in Section \ref{subsec: preliminarie on K32-type and Nikulin orbifold}. 
In \cite[Section 3.1]{CGKK} we considered the map of $\mathbb{Z}$-modules $Q:\Lambda_{K3^{[2]}}\rightarrow \Lambda_N$ defined  as
\[
(\underline{u},\underline{w},\underline{v},\underline{x},\underline{y},k)\mapsto(\underline{u},\underline{w},\underline{v},\underline{x}+\underline{y},k,k),
\]
such that the image of its restriction to the invariant sublattice $\Lambda_{K3^{[2]}}^I$ is of finite index inside $\mathbb{T}\simeq \delta^\perp$;
 in fact,$$Q_{|\Lambda_{K3^{[2]}}^I}(\underline{u},\underline{w},\underline{v},\underline{x},\underline{x},k)=(\underline{u},\underline{w},\underline{v},2\underline{x},k)\in \mathbb{T},$$ hence \[q_{\Lambda_N}(Q_{|\Lambda_{K3^{[2]}}^I}(\underline{u},\underline{w},\underline{v},\underline{x},\underline{x},k))=2q_{K3^{[2]}}((\underline{u},\underline{w},\underline{v},\underline{x},\underline{x},k)).\]
Observe that the map $Q_{|\Lambda_{{K3}^{[2]}}^I}$ is injective and that its $\C$-linear extension $(Q_\C)_{|\Lambda_{{K3}^{[2]}}^I}$ is surjective onto $\mathbb{T}\otimes \C$, so that $\mathbb{P}(\Lambda_{K3^{[2]}}^I\otimes \mathbb{C})\cap Q_{\mathbb{C}}^{-1}(\mathcal{P}(Z,\eta))$ is a unique point $p$. Even though $Q$ is not an isometry, $p$ still satisfies $q_{K3^{[2]}}(p)=0$, $B_{K3^{[2]}}(p,\overline{p})>0$, hence by 
Theorem \ref{thm:surjectivity-K3[2]-period-map} there exists a marked $(X,\phi)\in\mathcal{P}_{K3^{[2]},I}^{-1}(p)$  with $X$ of $K3^{[2]}$-type such that $\NS(X)\supset E_8(-2)$ and which is endowed with a symplectic involution $\sigma$ such that  $\phi\circ\sigma^*=I\circ \phi$, as defined in Section \ref{subsec: preliminarie on K32-type and Nikulin orbifold}. 

We want to construct a marking $\zeta:H^2(Y,\mathbb{Z})\rightarrow \Lambda_N$ such that $\zeta\circ\rho^*\circ\pi_*=Q\circ\phi$ and $\mathcal{P}(Y,\zeta)=\mathcal{P}(Z,\eta)$. Recall that $H^2(Y,\mathbb{Z})$ is an overlattice of $\rho^*(H^2(X/\sigma,\mathbb{Z}))\oplus \mathbb{Z}\Sigma$, obtained by adding the so-called gluing vector  $(\Delta_Y+\Sigma)/2$, where $\Delta_Y=\rho^*\pi_*\delta_X$ and where $\delta_X$ is the primitive class which corresponds to $\Delta$ in $H^2(X,\Z)$.
 We want to define $\zeta$ as the linear extension to $H^2(Y,\mathbb{Z})$ of \begin{equation}\label{eq: big map}\left(\frac{1}{2}Q_{|\Lambda_{K3^{[2]}}^I}\circ \phi\circ \pi^*\circ (\rho^*)^{-1}_{|\Sigma^\perp}\right)\oplus \alpha,\end{equation} where $\alpha(\Sigma):=\delta_1-\delta_2$  ($\delta_i$ are the two generators of the $\langle -2\rangle$ summands in $\Lambda_N$). For this we observe that $\alpha$ is obviously an isometry and we need to show that $\frac{1}{2}Q_{|\Lambda_{K3^{[2]}}^I}\circ \phi\circ \pi^*\circ (\rho^*)^{-1}_{|\Sigma^\perp}$ is an isometry which glues with $\alpha$, i.e.~ the $\Q$-linear extension of the map \eqref{eq: big map} sends the glue vector $(\Delta_Y+\Sigma)/2$ to the glue vector of the overlattice $\Lambda_N$ of $\langle -4\rangle\oplus \mathbb{T}$.

Since $(\rho^*)^{-1}_{|\Sigma^\perp}$ is an isometry, we focus on $\frac{1}{2}Q_{|\Lambda_{K3^{[2]}}^I}\circ \phi\circ \pi^*$. The map $\pi^*$ is injective since $X$ is simply connected (see \cite{Menet}) and $Q_{|\Lambda_{K3^{[2]}}^I}$ is injective, 
 hence $\frac{1}{2}Q_{|\Lambda_{K3^{[2]}}^I}\circ \phi\circ \pi^*$ is injective and its image is the sublattice $\mathbb{T}$.  

Moreover, we have $$\pi^*a\pi^*b\pi^*c\pi^*d=2abcd,$$ hence $$q_X(\pi^*a)^2=\frac{1}{3}(\pi^*a)^4=\frac{2}{3}a^4=4q_Y(\rho^*a)^2,$$ by Fujiki relations for $X$ and $Y$.
It follows that $\frac{1}{2}Q_{|\Lambda_{K3^{[2]}}^I}\circ \phi\circ \pi^*\circ (\rho_*)_{|\Sigma^\perp}$ is an isometry, since \[q_{\Lambda_N}(\frac{1}{2}Q(\phi(\pi^*(\rho_*(y)))))=\frac{2}{4}q_{\Lambda_{K3^{[2]}}}(\phi(\pi^*(\rho_*(y))))=\frac{1}{2}q_{X}(\pi^*(\rho_*(y)))=q_Y(\rho^*\rho_*(y))=q_Y(y).\]  Finally, by definition of $Q$, $Q(\delta_X)=\delta_1+\delta_2$, hence the glue vector of the overlattice $H^2(Y,\mathbb{Z})$ is sent to the glue vector of the overlattice $\Lambda_N$ of $\mathbb{T}\oplus\langle -4\rangle$. Thus, $\zeta$ is a marking and $$\mathcal{P}(Y,\zeta)=[\zeta_\mathbb{C}(H^{2,0}(Y))]=[(Q\circ\phi\circ\frac{1}{2}\pi^*\circ\rho_*)_{\mathbb{C}}(H^{2,0}(Y))]=[(Q\circ\phi)_{\mathbb{C}}(H^{2,0}(X))]=\mathcal{P}(Z,\eta).$$

We thus obtain a Hodge isometry $\eta^{-1}\circ \zeta:H^2(Y,\Z)\rightarrow H^2(Z,\Z)$ which is a parallel transport operator: indeed, we can always assume that it is orientation preserving, up to a change of sign from $\zeta$ to $-\zeta$, and that $(Y,\zeta)$ and $(Z,\eta)$ are in the same connected component of $\mathcal{M}_{\Lambda_N}$, by Theorem  \ref{prop:connected components of moduli}.
\end{proof}

Let $\psi_{\langle -4\rangle}$ be the primitive embedding fixed in Remark \ref{rem: explicitl psi}. One can also consider the coarse moduli space $\mathcal{M}_{\langle -4\rangle,\psi}$ of (isomorphism classes of) marked $\psi_{\langle -4\rangle}$-polarized pairs $(Z,\eta)$, where $Z$ is an orbifold of Nikulin type and $\eta:H^2(Z,\Z)\rightarrow \Lambda_N$ is a marking such that there exists a primitive embedding $j:\langle -4\rangle\hookrightarrow \NS(Z)$ satisfying $\eta\circ j=\psi_{\langle -4\rangle}$. The period map is a holomorphic surjection $$\mathcal{P}_{\langle -4\rangle,\psi}:\mathcal{M}_{\langle -4\rangle,\psi}\rightarrow D_{\mathbb{T}}:=D_{\Lambda_{N}}\cap\mathbb{P}(\mathbb{T}\otimes\mathbb{C}).$$

We observe that for any $\delta\in\Lambda_N$ such that $q_{\Lambda_N}(\delta)=-4$, $\divi\delta=2$ and $\delta^\perp\simeq \mathbb{T}$, there is an isometry $\gamma\in O(\Lambda_N)$ such that $\gamma(\delta)=\mathbb{E}$ and $\gamma(\delta^\perp)=\psi_{\mathbb{T}}(\mathbb{T})$, by Lemma \ref{lemma: embeddings of <-4> and T in LambdaN} (5). As a consequence, the hyperplane $\delta^\perp$ inside $\mathcal{D}_{O\cdot\mathbb{E}}$ is isomorphic to $D_{\mathbb{T}}$ and   $(Z,\gamma\circ\eta)\in \mathcal{M}_{\langle -4\rangle,\psi}$ for any pair $(Z,\eta)\in\mathcal{P}^{-1}(\delta^\perp)\subset\mathcal{M}_{O\cdot\mathbb{E}}$. Indeed, one can take $j:=\eta^{-1}\circ\gamma^{-1}\circ \psi_{\langle -4\rangle}$ to obtain a primitive embedding of $\langle -4\rangle$ inside $\NS(Z)$.
\begin{corollary} 
    There is a holomorphic map $\Phi$ with dense image $\mathcal{M}_{K3^{[2]},I}\rightarrow \mathcal{M}_{\langle -4\rangle,\psi}$ given by 
    \(    (X,\eta)\mapsto (Y,\zeta)    \)
    with $(Y,\zeta)$ constructed as in the proof of Theorem \ref{Thm1.1} such that the following diagram commutes
\[\xymatrix{
    \mathcal{M}_{K3^{[2]},I}\ar[r]^\Phi\ar[d]_{\mathcal{P}_{K3^{[2]},I}}  &
\mathcal{M}_{\langle -4\rangle,\psi}\ar[d]^{\mathcal{P}_{\langle -4\rangle,\psi}}\\
    D_{\Lambda_{K3^{[2]}}^I}\ar[r]^{\mathbb{P}(Q_\C)} & D_{\mathbb{T}}.    
    }
    \]    
\end{corollary}

\section{The generalized Nikulin lattice and the proof of Theorems \ref{Thm1.2} and \ref{thm: projective $K3$ G polarized are generalized Nikulin}}\label{section: proofs of 2nd theorem and definition of Generalized Nikulin lattice}
In this section we prove Theorem \ref{Thm1.2} and discuss the properties of the generalized Nikulin lattice, introduced in Definition \ref{defi: generalizad Nikulin orbifold}. We use the same notation as in Section \ref{subsec: preliminarie on K32-type and Nikulin orbifold}.

\subsection{The generalized Nikulin lattice}
In Definition \ref{defi: generalizad Nikulin orbifold} we introduced the generalized Nikulin lattice $G$, defined as the orthogonal to the class $\sum_iE_i$ in the Nikulin lattice $N$. We are now interested in its lattice-theoretic properties, which will be useful in the following. 

The lattice $G$ is generated by the vectors $$w_i=E_i-E_{i+1},\ i=1,\ldots, 7\mbox{ and }\hat{w}:=(w_1+w_3+w_5+w_7)/2,$$ where the latter class corresponds to $\frac{1}{2}\sum_i{(-1)}^{i+1}E_i\in N$  and is orthogonal to $\sum_{i=1}^8E_i$. With this notation a generating set of the discriminant group of $G$ is given by 
\begin{equation}\label{disc G}
w_i/2,\ i=1,2,3,4,5,6,7,\ \ \frac{w_1-w_2-w_5+w_6+\hat{w}}{4}.
\end{equation}

\begin{lem}\label{lemma: properties of G}
The generalized Nikulin lattice $G$ has the following properties: 
\begin{itemize}
\item it is isometric to $E_7(-2)$ and its discriminant quadratic form is $u(2)^{\oplus 3}\oplus \left(\frac{1}{4}\right)$;
\item the natural map $O(G)\ra O(A_G)$ is surjective; 
\item  it admits a unique primitive embedding in $\Lambda_{K3}$ (up to isometries); \item its orthogonal complement in $\Lambda_{K3}$ is isometric to $\mathbb{T}$;
\item $\mathbb{T}$ admits a unique primitive embedding in $\Lambda_{K3}$.
\end{itemize}
\end{lem}
\begin{proof}
We observe that $w_i^2=-4$ and $w_{i}w_{i+1}=2$, i.e.~ the lattice $\langle w_i\rangle$ is $A_7(-2)$ and its discriminant group is $(\Z/2\Z)^6\times \Z/16\Z$. When one adds the class $\hat{w}:=(w_1+w_3+w_5+w_7)/2$ one obtains a lattice which is isometric to $E_7(-2)$, whose discriminant group is $(\Z/2\Z)^6\times \Z/4\Z$ and whose discriminant form is $u(2)^{\oplus 3}\oplus (\frac{1}{4})$. This is the generalized Nikulin lattice $G$ defined before. By a direct computation with \cite{oscar} one shows that both $O(A_{E_7(-2)})$ and the image of map $O(G)\ra O(A_G)$ are groups of order 2903040, therefore the map is surjective.

By \cite[Theorem 1.14.4]{N}, $G\simeq E_7(-2)$ admits a unique primitive embedding in $\Lambda_{K3}$, up to isometries. This determines the orthogonal complement in $\Lambda_{K3}$, which is the unique lattice with signature $(3,12)$ and with discriminant form  $u(2)^{\oplus 3}\oplus (-\frac{1}{4})$, i.e.~ the opposite of the discriminant form of $G$. In particular, the orthogonal complement to $G$ in $\Lambda_{K3}$ is isometric to $\mathbb{T}$, which admits a unique embedding in $\Lambda_{K3}$ by \cite[Theorem 1.14.4]{N}.
\end{proof}
Observe that Lemma \ref{lemma: properties of G} is analogous to Lemma \ref{lemma: embeddings of <-4> and T in LambdaN} where it is also proved that $O(\mathbb{T})\ra O(A_{\mathbb{T}})$ is surjective.

\begin{rem}{\rm By construction the Nikulin lattice $N$ is an overlattice of index 4 of $G\oplus\langle -4\rangle\simeq E_7(-2)\oplus \langle -4\rangle$. In particular, to reconstruct $N$ as an overlattice of $G\oplus\langle -4\rangle$ one adds  $(3w_1  +3w_2+ 2w_3  +2w_4+  w_5 + w_6  +\hat{w}  +z )/4$, where $z$ is the generator of the direct summand $\langle -4\rangle$. We observe that the added class corresponds to $E_1$ if $z$ corresponds to $\frac{1}{2}\sum_iE_i$. 

Also the lattice $E_8(-2)$ is an overlattice of $G\oplus\langle -4\rangle\simeq E_7(-2)\oplus \langle -4\rangle$. The index of the inclusion is now 2, and the gluing vector between $E_7(-2)$ and $\langle-4\rangle$ is $(\varepsilon_4+\varepsilon_6+\varepsilon_7+z)/2$, where $z$ is the generator of the direct summand $\langle -4\rangle$ and $\varepsilon_i$, $i=1,\ldots, 7$ is a basis of $E_7(-2)$ such that $\varepsilon_i\varepsilon_j=2$ for $i,j\in\{1,\ldots, 6\}$ and $|i-j|=1$, as well as  $\varepsilon_3\varepsilon_7=2$.
With the same notation one can even observe that $N$ is isometric to the overlattice of index 4 of $E_7(-2)\oplus\langle -4\rangle$ obtained by adding the vector $(2\varepsilon_1+ \varepsilon_4+\varepsilon_6+\varepsilon_7+z)/4$.}  
\end{rem}

\begin{rem}\label{rem: embedding E7(-2) in LambdaK3}
{\rm By Lemma \ref{lemma: properties of G} there exists a unique, up to isometries, embedding of $E_7(-2)$ in $\Lambda_{K3}$. A very natural way to describe such an embedding is the following: let $f_i^{(j)}$, $i=1,\ldots,8$, $j=1,2$ a basis of the $j$-th copy of $E_8(-1)$ in $\Lambda_{K3}$, such that $f_i^{(j)}f_{i+1}^{(j)}=1$ $i=1,\ldots ,6$ and $f_3^{(j)}f_{8}^{(j)}=1$. Denoted, as before, $\varepsilon_i$ the standard basis of $E_7(-2)$, we fix the embedding $\varphi_{E_7(-2)}(\varepsilon_i)=f_i^{(1)}-f_i^{(2)}$, $i=1,\ldots, 6$ and   $\varphi_{E_7(-2)}(\varepsilon_7)=f_8^{(1)}-f_8^{(2)}$. This determines uniquely the embedding $\varphi_{\mathbb{T}}$ of $\mathbb{T}$ in $\Lambda_{K3}$.

Alternative embeddings (more related with our construction, but less explicit in the standard basis of $\Lambda_{K3}$), are the following: $\Lambda_{K3}$ is an obvious overlattice of $U^{\oplus 3}\oplus N\oplus N$ (where one glues the discriminant of one copy of $N$ with the other) and also of $U(2)^{\oplus 3}\oplus E_8(-1)\oplus N$ (see e.g. \cite{vGS}). The orthogonal complement to $\sum_i E_i$ in $N\subset U(2)^{\oplus 3}\oplus E_8(-1)\oplus N$ 
is a primitive sublattice of $\Lambda_{K3}$, isometric to $G$ and its orthogonal complement is $\mathbb{T}$, whose last summand is $\frac{1}{2}\sum_i E_i$. The orthogonal complement to $\sum_i E_i$ in $N\subset U^{\oplus 3}\oplus N\oplus N$ is isometric to $G$ and its orthogonal complement is $U^{\oplus 3}\oplus N\oplus \langle -4\rangle$, which is isometric to $\mathbb{T}$.

We underline that all these embeddings are equivalent up to isometries of $\Lambda_{K3}$.}
\end{rem}

\begin{rem}{\rm Since $G\simeq E_7(-2)$ admits a unique primitive embedding in $\Lambda_{K3}$, its orthogonal is uniquely determined by the discriminant form and $O(G)\rightarrow O(A_G)$, the $K3$ surfaces $S$ with $\NS(S)\simeq E_7(-2)$ do not admit Fourier--Mukai partners.}\end{rem}

\subsection{The restriction map $\nu^*$ and the very general generalized Nikulin surfaces}
Given a  manifold of $K3^{[2]}$ type $X$ with a symplectic involution $\iota$, let $F$ be the surface fixed by $\iota$ on $X$. Then the inclusion map $\nu:F\hookrightarrow X$ induces the restriction map $\nu^*:H^2(X,\Q)\ra H^2(F,\Q)$. We recall some of its properties
\begin{lemma}\label{lemma: properties of nu*}Let $\nu^*$ be as above, $\mu_{[F]}:H^2(X,\Q)\ra H^6(X,\Q)$ be the cup product with $[F]$ and $\langle - ,- \rangle_{\bullet}$ are the cup products in $H^*(\bullet,\Q)$. \begin{enumerate}
\item
$\nu^*\colon H^2(X,\mathbb Q)\to H^2(F,\mathbb Q)$ is a morphism of Hodge structures;
\item $\langle \nu^*x,\nu^*y\rangle_F= \langle \mu_{[F]}(x),y\rangle_X$  for all $x,y\in H^2(X,\mathbb{Z})$;
\item $\ker(\mu_{[F]}) = \ker \nu^*$;
\item $\nu^*((H^2(X,\Q)^{\iota})^{\perp})=0$;
\item the restrictions $\nu^*_{{|H^2(X,\Q)}^{\iota^*}}$ and $\nu^*_{{|H^2(X,\Z)}^{\iota^*}}$ are injective.
\end{enumerate}
\end{lemma}
\begin{proof}
Point (1) is proved in \cite[Proposition 3.16]{CGKK}, and points (2) and (3) are proved in \cite[Proposition B.2 and Remark B.3]{Vo}. For (4), notice that the classes $x\in (H^2(X,\Z)^{\iota^*})^{\perp}$ are such that $\iota^*(x)=-x$, so $\mu_{[F]}(\iota^*(x))=-\mu_{\iota^*[F]}(x)$ which implies $\mu_{[F]}(x)=-\mu_{[F]}(x)=0$. The result follows by (3). If $(X,\iota)$ is a very general pair with $\iota$ a symplectic involution on $X$, then $\T_X\simeq H^2(X,\Z)^{\iota^*}$ so that $\nu^*(H^2(X,\Q)^{\iota^*})$ is a non trivial sub Hodge structure of $\T_F\otimes \Q$. As in the proof of \cite[Proposition 3.16]{CGKK}, by the irreducibility of $\T_X\otimes \Q$, one concludes that  $\nu^*_{{|H^2(X,\Q)}^{\iota^*}}$ is injective, which clearly implies also the injectivity of $\nu^*_{{|H^2(X,\Z)}^{\iota^*}}$.  Since the injectivity is a topological property, hence flat on the deformation family of $(X,\iota)$, it holds for every pair $(X,\iota)$
\end{proof}
\begin{rem}{\rm Point (4) implies that $\nu^*((H^2(X,\Q)^{\iota})=\nu^*((H^2(X,\Q))$ in the cohomology with rational coefficients and in particular the lattices $\nu^*((H^2(X,\Z)^{\iota})$ and $\nu^*(H^2(X,\Z))$ have the same saturation in $H^2(F,\Z)$, which is a lattice of rank $15=\rk((H^2(X,\Z)^{\iota}))$ by (5).}
\end{rem}

If $(X,\iota)$ is very general, then $\T_X\simeq H^2(X,\Z)^{\iota^*}$ and therefore the saturation of $\nu^*\left(H^2(X,\mathbb Z)^{\iota^*}\right)$ is Hodge isomorphic to $\T_F$ by Lemma \ref{lemma: properties of nu*}. 
\begin{defi}
Denote $T_F^{G}$
the saturation of $\nu^*(H^2(X,\mathbb Z)^{\iota^*})\subset H^2(F,\mathbb{Z})$. 
\end{defi}

\begin{prop}\label{main-prop} \label{cor: very general F has NS=G}
    Let $F$ be the generalized Nikulin surface of a very general manifold of $K3^{[2]}$-type $X$ with a symplectic involution $\iota$. Then,$$\NS(F)\simeq E_7(-2)\ \ \mbox{ and }\ \ \T_F=T_F^G\simeq \mathbb{T}.$$
\end{prop}
\begin{proof}
Since $\nu^*\colon H^2(X,\mathbb Q)\to H^2(F,\mathbb Q)$ is a morphism of Hodge structures, $\T_F\otimes \Q=\nu^*(\T_X)\otimes \Q$ for every $X$ with a symplectic involution $\iota$. Here we are interested in describing $\T_F$ and $\NS(F)$ (as $\Z$ lattices, and not just their $\Q$-extensions).

Let us  specialize $(X,\iota)$ to $(S^{[2]},\iota_S^{[2]})$ where $S^{[2]}$ is the Hilbert scheme of two points on a $K3$ surface $S$ admitting a symplectic involution $\iota_S$ and $\iota_S^{[2]}$ is the induced symplectic involution on $S^{[2]}$. This specialization exists by \cite{Mon}, and is associated to a specialization of the related fixed surfaces. Indeed, by \cite[proof of Proposition 4.6]{Mon}, the surface $F$ (fixed by the action of $\iota$ on $X$) specializes to the surface fixed by $\iota_S^{[2]}$ on $S^{[2]}$. The surface fixed by $\iota_S^{[2]}$ on $S^{[2]}$ is isomorphic to the Nikulin surface of $(S, \iota_S)$, i.e.~ it is isomorphic to $\widetilde{S/\iota_S}$. Therefore, we have a specialization of the triple $(X,\iota,F)$ to the triple $(S^{[2]},\iota_S^{[2]}, \widetilde{S/\iota_S})$. If $S$ is a very general K3 surface with a symplectic involution, i.e.~ $\NS(S)\simeq E_8(-2)$, then $\NS(\widetilde{S/\iota_S})\simeq N$ and $\NS(S^{[2]})$ primitively contains the anti-invariant lattice $\NS(S^{[2]})_-=(H^2(S^{[2]},\Z)^{\iota_S^{[2]}})^{\perp}$ and the class $\Delta$, which is invariant for the action of $\iota_S^{[2]}$.

We consider a local universal family $\mathcal{X}\to U$ for a small disc $U$ of deformations of $S^{[2]}$ together with the involution $\iota_S^{[2]}$(cf.~\cite[p.1480]{Mon}).
After choosing a trivialization of the local system as in \cite[p 1481]{Mon}, the deformation space is a subset of $\PP(H^2(S^{[2]},\mathbb C))$. The considered deformations are parameterized by periods $H^{2,0} \in H^2(X,\mathbb C)$ that are orthogonal to the anti-invariant $E_8(-2)$ (the embedding is unique up to isomorphism by \cite[Corollary 5.3]{Mon}) but not necessarily to $\Delta$. In this local deformation family, all the fibers admit a symplectic involution (but not all of them are Hilbert scheme of points on a $K3$ surface).

The deformation from $S^{[2]}$ to $X$ consists of requiring that the class $\Delta$ becomes a transcendental class (since it is invariant and for the very general $X$ admitting a symplectic involution $\iota$ we have $H^2(X,\Z)^{\iota}=\T_X$). Therefore, the corresponding deformation for the fixed surface from $\widetilde{S/\iota_S}$ to $F$ consists in requiring that $\nu^*(\Delta)$ becomes transcendental.

As a consequence, in order to determine the N\'eron--Severi group of $F$, it suffices to consider the orthogonal complement of  
$\nu^*(\Delta)$ in $\NS(\widetilde{S/\iota_S})$.
Observe that $$\sum_iE_i=\nu^* (\Delta)$$ which is clear set theoretically and can be proved scheme theoretically using the local coordinates as in \cite[Example 4.5 a)]{F}. Recalling that $\delta_X$ is the primitive class which corresponds to $\Delta$ in $H^2(X,\Z)$, we obtain $\frac{1}{2}\sum_iE_i=\nu^* (\delta_X)$.
The orthogonal complement to $\frac{1}{2}\sum_iE_i$ in the Nikulin lattice $N$ is isometric to $E_7(-2)$, by Lemma \ref{lemma: properties of G}. So $\NS(F)\simeq E_7(-2)$ and by Lemma \ref{lemma: properties of G}, this implies that $\T_F\simeq \mathbb{T}$. Being the transcendental lattice primitively embedded in the second cohomology group, $\T_F$ is also the saturation of $\nu^*(\T_X)=\nu^*(H^2(X,\Z)^{\iota^*})$ and in particular $\T_F=T_F^G$.
\end{proof}

\subsection{The transcendental lattices of $Y$ and $F$ and the proof of Theorem \ref{Thm1.2}}

The choice of the pair $(X,\iota)$ determines a pair $(Y,\Sigma)$ where $Y$ is the terminalization of $X/\iota$ with exceptional divisor $\Sigma$. The  image of the map $\pi_*:H^2(X,\Z)\ra H^2(X/\iota,\Z)$ is embedded in $H^2(Y,\Z)$ via $\rho^*$ where $\rho:Y\ra X/\iota$ is the terminalization and thus it is orthogonal to $\Z\Sigma$. The same holds true for the image through $\pi_*$ of $H^2(X,\Z)^{\iota^*}$. So the saturation of $\rho^*\pi_*(H^2(X,\Z))$ and of $\rho^*\pi_*(H^2(X,\Z)^{\iota^*})$ coincide and both coincide with the orthogonal complement of $\Z\Sigma$ in $H^2(Y,\Z)$. By the explicit expression of $\pi_*$ one sees that $\rho^*\pi_*(H^2(X,\Z))$ is saturated.

\begin{defi} Denote 
$T_Y^{-4}$ the saturation of the lattice $\rho^*\pi_*(H^2(X,\mathbb Z)^{\iota^*})$ in $H^2(Y,\mathbb{Z})$.\end{defi}

\begin{corr}\label{5.4}\label{cor: very general F are G polarized and TF=TY}
If $X$ is a manifold of $K3^{[2]}$-type with $\NS(X)=E_8(-2)$ then:\begin{itemize}\item it admits a symplectic involution $\iota$; \item the fixed $K3$ surface $F$ 
satisfies $\NS(F)\simeq E_7(-2)$, and $\T_F\simeq \mathbb{T}\simeq T^G_F$; \item  the terminalization $Y$ of the quotient $X/\iota$ is a  
Nikulin orbifold 
such that $\NS(Y)\simeq \Z \mathbb{E}$ and $\T_Y\simeq \mathbb{T}\simeq T_Y^{-4}$.\end{itemize}
\end{corr}
\proof If $X$ is a  manifold of $K3^{[2]}$-type with $\NS(X)\simeq E_8(-2)$, then it  
admits a symplectic involution $\iota$, see e.g. \cite{CGKK}. Then its fixed surface $F$ is a  
generalized Nikulin surface and by Proposition \ref{cor: very general F has NS=G} $\NS(F)\simeq E_7(-2)$ and $\T_F\simeq \mathbb{T}$. The N\'eron--Severi of $Y$ is computed in Theorem \ref{Thm1.1} and then its transcendental lattice is isometric to $\mathbb{T}$ by Lemma \ref{lemma: embeddings of <-4> and T in LambdaN}. By assumption $\NS(X)\simeq E_8(-2)$ and $\T_X\simeq H^2(X,\Z)^{\iota}$. So the saturation of $\pi_*(H^2(X,\Z)^{\iota})$ in $\Lambda_N$ is the transcendental lattice $\T_Y$ and so $T_Y^{-4}\simeq \T_Y$.
In particular, $T_F^G\simeq \T_F\simeq \mathbb{T}\simeq \T_Y\simeq T_Y^{-4}$ as lattices. \endproof

To prove the conjecture \cite[Conjecture 3.12]{CGKK} we now consider a   manifold of $K3^{[2]}$-type $X$ with a symplectic involution $\iota$, without requiring that $X$ is very general. Let $F$ be the fixed locus of a symplectic involution $\iota$ on $X$ and $Y$ the Nikulin orbifold of $(X,\iota)$. We need to relate $\T_F$ with $\T_Y$ and to do that we will consider them as primitive sublattices of the lattices $T_F^G$ and $T_Y^{-4}$ respectively.

\begin{remark} Observe that $T_F^G$ (resp. $T_Y^{-4}$) is primitively embedded in $H^2(F,\Z)$ (resp. $H^2(Y,\Z)$) and this determines also an explicit embedding of its orthogonal complement $E_7(-2)$ (resp. $\Z \Sigma$)  in $H^2(F,\Z)$ (resp. $H^2(Y,\Z)$).\end{remark}  

For any choice of the pair $(X,\iota)$ as above, the sublattice $T_F^G$ (resp. $T_Y^{-4}$) contains primitively $\T_F$ (resp. $\T_Y$); as a consequence, these two sublattices are also two polarized Hodge substructures of $H^2(F,\Z)$ and $H^2(Y,\Z)$ respectively.
\begin{thm}\label{them: proof of the conj} 
Let $X$ be a  manifold of $K3^{[2]}$-type with a symplectic involution $\iota$ and let $Y$ be the terminalization of $X/\iota$, i.e.~ the associated Nikulin orbifold. Then $T_F^{G}$ is Hodge isometric with $T_Y^{-4}$.  
In particular, the transcendental lattices $\T_Y$ and $\T_F$ are isometric.
\end{thm}
\begin{proof}
 From \cite[Proposition 3.16]{CGKK}
 we deduce that there is a rational Hodge isomorphism $f$ from $T_Y^{-4}\otimes \Q$ to $T_F^{G}\otimes \Q$ (we are putting $f:=\nu^*\tilde{\rho}_*\tilde{\pi}^*$ with the notation of the proof of \cite[Proposition 3.16]{CGKK}, which coincides with $\nu^*\pi^*\rho_*$ by the commutativity of the diagram in the proof of \cite[Proposition 3.16]{CGKK}), since $$f(T^{-4}_Y\otimes \Q)=\nu^*\pi^*\pi_*(H^2(X,\Z)^{\iota^*})\otimes \Q=\nu^*(H^2(X,\Z)^{\iota^*}(2))\otimes \Q=T_F^G\otimes \Q.$$ 
 By Corollary \ref{5.4} both $T_Y^{-4}$ and $T_F^{G}$ are abstractly isometric to $\mathbb{T}$. 

 The domain $T_Y^{-4}$ and the codomain $T_F^{G}$ of the map $f$ are both isometric to $\mathbb{T}$, and we want to construct an integral Hodge isometry between these lattices. By \cite[Proposition 3.16]{CGKK}  $$\forall\alpha, \beta\in T_Y^{-4}\otimes \Q,\ \ B_Y(\alpha,\beta)=\frac{1}{4}\langle f(\alpha), f(\beta)\rangle_F,$$ where $B_Y(\cdot,\cdot)$ is the Beauville-Bogomolov form on $Y$ and $\langle\cdot ,\cdot\rangle_F$ is the cup product on $H^2(F,\Q)$.
Then we have to show that integral classes on $T_Y^{-4}$ map to $2T_F^G$(=$\{2w\mid w\in T_F^G\}$). Indeed, $f$ is the composition of three maps among which there is $\pi^*$.  
From \cite[page 19]{Menet} the classes $x$ in the $E_8(-1)$ summand map under $\pi^*$ to $x+\iota^*(x)$ and since $F$ is invariant $\nu^*(x)=\nu^*(\iota^*(x))$, therefore $x$ is mapped to $2T_F^G$ (this is the class $s^{\ast}_1(x)$ in loc. cite). The remaining classes from $U(2)^{\oplus 3}\oplus \langle-4\rangle$ already map to classes divisible by $2$ on $X$, indeed $\pi^*x=2x$ for every invariant class and $U(2)^{\oplus 3}\oplus \langle-4\rangle$ is primitively contained in the saturation of $\rho^*\pi_*(H^2(X,\Z)^{\iota^*}$.
So $$\forall\gamma\in T_Y^{-4},\ \text{we have}\ f\left(\frac{1}{2}\gamma\right)\in H^2(F,\Z),\mbox{ and  we define }\tilde{f}(\gamma):=f\left(\frac{1}{2}\gamma\right)$$ so that $\tilde{f}$ is a map between integral cohomology groups.
Then $B_Y(\alpha,\beta)=\frac{1}{4}\langle f(\alpha),f(\beta)\rangle_F=\langle f(\frac{1}{2}\alpha), f(\frac{1}{2}\beta)\rangle_F=\langle\tilde{f}(\alpha), \tilde{f}(\beta)\rangle_F$.

 It follows that the map $\tilde f$ is an isometry and it preserves the Hodge structure. 
 Then by standard Hodge theory $\tilde f$ induces an isomorphism of polarized integral Hodge structures $T^{-4}_Y\rightarrow T^G_F$ which restricts to an isomorphism of polarized Hodge structures $\T_Y\rightarrow \T_F$ .
\end{proof}
\begin{remark}\label{rem: conjecture and marikings} The lattices $T_Y^{-4}$ and $T_F^G$ are both isometric to $\mathbb{T}$ and we fixed explicitly the embeddings $\psi_{\mathbb{T}}$ and $\varphi_{\mathbb{T}}$ of $\mathbb{T}$ in $\Lambda_N$ and $\Lambda_{K3}$ respectively (see Remarks \ref{rem: explicitl psi} and \ref{rem: embedding E7(-2) in LambdaK3}). Moreover, there exist markings $\alpha$ and $\eta$ of $F$ and $Y$ respectively such that $T_F^G=\alpha^{-1}(\varphi_{\mathbb{T}}(\mathbb{T}))$ and $T_Y^{-4}=\eta^{-1}(\psi_{\mathbb{T}}(\mathbb{T}))$. Therefore, the Hodge isometry $\tilde{f}\colon T_Y^{-4}\rightarrow T^G_F$ induces an isometry \begin{equation}\label{eq: Hodge f}\tilde{f}:\eta^{-1}(\psi_{\mathbb{T}}(\mathbb{T}))\ra \alpha^{-1}(\varphi_{\mathbb{T}}(\mathbb{T})),\end{equation} which restricts to a Hodge isometry $T_Y^{-4}\supset \T_Y\ra \T_F\subset T_F^G$. \end{remark}

\subsection{The lattice-theoretic description of generalized Nikulin surfaces and the proof of Theorem \ref{thm: projective $K3$ G polarized are generalized Nikulin}}

We want to characterize the generalized Nikulin surfaces from a lattice-theoretic point of view, by identifying them with $K3$ surfaces whose N\'eron--Severi group primitively contains $E_7(-2)$. To do that we will consider a $K3$ surface with this lattice-theoretic property and we associate to it a Nikulin orbifold, by using the properties of their transcendental lattices. Then, we are able to obtain a generalized Nikulin surface from the Nikulin orbifold and to show that it is isomorphic to the original $K3$ surface, which clearly shows that even the original one was a generalized Nikulin surface. 

\begin{thm}\label{thm: generalizes Nikulin iff generalized lattice in NS}
A $K3$ surface $S$ is isomorphic to a generalized Nikulin surface if and only if the lattice $E_7(-2)$ is primitively embedded in $\NS(S)$.\end{thm}
\proof By definition, if $S$ is a generalized Nikulin surface, then there exists a pair $(X,\iota)$ such that $S\subset \Fix_{\iota}(X)$ and there exists a Hodge isometry between $\T_F$ and $\T_Y$, where $Y$ is the terminalization of the quotient $X/\iota$, by Theorem \ref{them: proof of the conj}. This implies that $T_F^G$ is well defined and it contains $\T_F$ as already observed. Therefore $$\NS(F)=\left(\T_F\right)^{\perp}\supseteq (T_F^G)^{\perp}\simeq E_7(-2).$$ 

Assume now that $S$ is a $K3$ surface whose N\'eron--Severi group contains primitively $E_7(-2)$ and let $j:E_7(-2)\hookrightarrow \NS(S)$ be the primitive embedding. To prove that $S$ is a generalized Nikulin surface we apply the following strategy: 

1) we show that, given a marking $\mu$ such that $\mu\circ j=\varphi_{E_7(-2)} $ (as defined in Remark \ref{rem: embedding E7(-2) in LambdaK3}), the pair $(S,\mu)$ determines a marked Nikulin orbifold $(Y,\eta)$; 

2) we associate to $(Y,\eta)$ a marked generalized Nikulin surface $(F,\alpha)$; 

3) we show that $S\simeq F$.

{\bf Step 1): $(S,\mu)$ determines $(Y,\eta)$.}
The period map for marked $K3$ surfaces defines a point $p:=\mathcal{P}_{K3}(S,\mu)\in D_{\mathbb{T}}$, where $D_{\mathbb{T}}:=D_{\Lambda_{K3}}\cap\mathbb{P}(\mathbb{T}\otimes\mathbb{C})$. 

Let $\mathcal{M}_{\langle -4\rangle,\psi}$ be the moduli space of marked orbifolds of Nikulin type $(Y,\eta)$ such that $\eta(\NS(Y))\supset \psi_{\langle -4\rangle}(\langle -4\rangle)$ and $\eta(\T_Y)\subset \psi_{\mathbb{T}}(\mathbb{T})$.

Let now $(Y,\eta)\in\mathcal{P}_{\langle -4\rangle,\psi}^{-1}(p)$ be a pair with $Y$ a Nikulin orbifold: it exists by the proof of Theorem \ref{Thm1.1} and it is such that $\eta(T_Y^{-4})=\psi_{\mathbb{T}}(\mathbb{T})=\psi_{\mathbb{T}}(\varphi_{\mathbb{T}}^{-1}(\mu(j(E_7(-2))^\perp)$ (see Remark \ref{rem: conjecture and marikings}). Similarly,  
the generator $\Sigma\in\NS(Y)$ of $\langle -4\rangle$ satisfies $\eta(\mathbb{Z}\Sigma)=\psi_{\langle -4\rangle}(\langle -4\rangle)$.

{\bf Step 2): $(Y,\eta)$ determines $(F,\alpha)$.} Given the marked Nikulin orbifold $(Y,\eta)$, there exists a $K3$ surface $F$ such that $F$ is the fixed surface of the symplectic involution $\iota$ on the  manifold of $K3^{[2]}$-type $X$ such that $Y$ is the terminalization of $X/\iota$. By Theorem \ref{them: proof of the conj} and Remark \ref{rem: conjecture and marikings}, the map $\widetilde{f}$ in \eqref{eq: Hodge f} is a Hodge isometry which restricts to a Hodge isometry between the transcendental lattices $\T_Y$ and $\T_F$. Therefore there exists a marking $\alpha$ of $F$ such that 
$\alpha^{-1}(\varphi_{\mathbb{T}}(\mathbb{T}))=T_F^G=\widetilde{f}(T_Y^{ -4})$ and it is such that  $\alpha((T_F^G)^{\perp})=\varphi_G(E_7(-2))$. 

{\bf Step 3) $S$ and $F$ are isomorphic.} The map \begin{equation}\label{eq:big HOdge isom}\widetilde{f}\circ \eta^{-1}\circ \psi_{\mathbb{T}}\circ \varphi^{-1}_{\mathbb{T}}\circ \mu:\ \ \left(\mu^{-1}\circ \varphi_{\mathbb{T}}\right)(\mathbb{T})\ \longrightarrow\ \left(\alpha^{-1}\circ\varphi_{\mathbb{T}}\right)(\mathbb{T})\end{equation}
is a Hodge isometry. Note that $\widetilde{f}\circ \eta^{-1}\circ \psi_{\mathbb{T}}\circ \varphi^{-1}_{\mathbb{T}}\circ \mu=\alpha^{-1}\circ \mu$ by by \eqref{eq: Hodge f}.

It remains to prove that each Hodge isometry between these two spaces extends to a Hodge isometry between $H^2(S,\Z)$ and $H^2(F,\Z)$ to conclude that $S\simeq F$ .  
This is proved once one shows that each isometry of $\varphi_{\mathbb{T}}(\mathbb{T})$ extends to an isometry of $\Lambda_{K3}$. By \cite[Lemma 7.3(2)]{GHS}, an isometry $\xi\in O(\mathbb{T})$ induces an isometry $\xi'\in O(A_{\mathbb{T}})\simeq O(A_G)$. By the surjectivity of $O(G)\rightarrow O(A_G)$, see Lemma \ref{lemma: properties of G}, there exists $\beta\in O(G)$ whose image is $\xi'$. So $\xi\oplus \beta$ is an isometry of $\varphi_{\mathbb{T}}(\mathbb{T})\oplus\varphi_G(E_7(-2))$ which extends to an isometry $\gamma$ of $\Lambda_{K3}$, which is in particular an Hodge isometry if so is $\xi$. Therefore, choosing $\xi:=\alpha\circ \widetilde{f}\circ \eta^{-1}\circ \psi_{\mathbb{T}}\circ \varphi^{-1}_{\mathbb{T}}$ and $\gamma$ its extension to $\Lambda_{K3}$,  $\alpha^{-1}\circ\gamma\circ \mu:H^2(S,\mathbb{Z})\rightarrow H^2(F,\mathbb{Z})$ is a Hodge isometry. 
Since there exists a Hodge isometry between $H^2(S,\Z)$ and $H^2(F,\Z)$, we conclude that $S\simeq F$, by Torelli theorem (see \cite{PS}). \endproof
\begin{rem}\label{rem: nu* algebraic}
Let us consider the lattice $L:=\{\ell\in \NS(F)\mid \ell g=0\ \forall\ g\in \left(\alpha^{-1}\circ\varphi_G\right)(G)\}$, which is the sublattice of $\NS(F)$ orthogonal to the embedded copy of $G\simeq E_7(-2)$ in $\NS(F)$.
The restriction of $\nu^*$  to $\NS(X)^{\iota^*}\subset H^2(X,\Z)^{\iota^*}$ is the sublattice $\nu^*(\NS(X)^{\iota^*})\subset \NS(F)$, and  $\nu^*(\NS(X)^{\iota^*})$ is orthogonal to the copy of $G$ embedded in $\NS(F)$, i.e. $\nu^*(\NS(X))^{\iota}\subset L$. 
Since they have the same rank, the inclusion has a finite index.
Observe that $\alpha(L)$ is contained in $\varphi_{\mathbb{T}}(\mathbb{T})$ (since it is orthogonal to $\varphi_G(G)$). So $\psi_{\mathbb{T}}\circ\varphi_{\mathbb{T}}^{-1}\circ\alpha$ embeds $L$ in $\Lambda_{N}$ (as a sublattice of $\psi_{\mathbb{T}}(\mathbb{T})$). So $(\mu^{-1}\circ\psi_{\mathbb{T}}\circ\varphi_{\mathbb{T}}^{-1}\circ\alpha)(L)=\widetilde{f}^{-1}(L)$, with the notation of the proof of Theorem \ref{them: proof of the conj}.  

On the other hand $\pi_*(\NS(X)^{\iota^*})$ is contained with finite index in $\left(\mu^{-1}\circ \psi_{\mathbb{T}}\right(\mathbb{T})$ and then it is contained with finite index in $\left(\mu^{-1}\circ \psi_{\mathbb{T}}\right)\left(\left(\varphi_{\mathbb{T}}^{-1}\circ\alpha\right)(L)\right)=\widetilde{f}^{-1}(L)$. 
We conclude that the saturation of $\nu^*(\NS(X)^{\iota^*})$ coincides with 
the saturation of $\pi_*(\NS(X)^{\iota^*})$. Indeed both of them are the saturation of a lattice which corresponds (in $H^2(F,\Z)$ and $H^2(Y,\Z)$ respectively) to the subattice $\alpha^{-1}(L)\subset\mathbb{T}$.\end{rem}

\section{Families of projective generalized Nikulin surfaces}\label{Families of projective generalized Nikulin surfaces}\label{section: Families of projective generalized Nikulin surfaces}
The N\'eron--Severi group of a very general generalized Nikulin surface is the negative definite lattice $E_7(-2)$, by Corollary \ref{cor: very general F are G polarized and TF=TY}, and in particular the very general generalized Nikulin surface is non projective. In this section we consider the projective generalized Nikulin surfaces in order to determine lattice-theoretically  their families and to relate them with the families of manifolds of $K3^{[2]}$-type admitting a symplectic involution of which they are generalized Nikulin surfaces. 
Therefore, we look for results similar to the ones presented in \cite[Section 2]{GS} where the families of projective Nikulin surfaces are described lattice-theoretically and the relation between these families and the ones of  $K3$ surfaces with a symplectic involution of which they are Nikulin surfaces are given. 

In particular, we classify all the lattices which appear as N\'eron--Severi group of projective generalized Nikulin surfaces of minimal Picard number in Corollary \ref{prop: Ns e T of projective F}, we describe the family of projective generalized Nikulin surfaces in Corollary \ref{cor: families of projective K3} and we relate the families of generalized Nikulin surfaces with the associated manifolds of $K3^{[2]}$-type and Nikulin orbifolds in Proposition \ref{prop:relation NS(F), NS(X), NS(W)}. 

\begin{lemma}\label{lemma: X proj iff F proj and rho(F)} Let $X$ be a manifold of $K3^{[2]}$-type admitting a symplectic involution $\iota$ and $F$ its fixed surface. Then, $X$ is projective if and only if $F$ is projective. In this case, $\rho(F)\geq 8$ and $\rho(X)=\rho(F)+1$.\end{lemma}
\begin{proof} Denoted, as above, $Y$ the Nikulin orbifold terminalization of $X/\iota$, we proved in Theorem \ref{them: proof of the conj} that $\T_F\simeq \T_Y$, which in particular implies that the signature of $\T_F$ and $\T_Y$ are the same. 
If $X$ is projective, the quotient variety $X/\iota$ is projective and the same holds true for its terminalization $Y$. Now, $X$ is projective if and only if the signature of $\T_X$ is $(2,20-(\rho(X)-1))=(2,21-\rho(X))$. We recall that $\rho(Y)=\rho(X)-7$ (since the anti-invariant sublattice of $\iota$  in $H^2(X,\Z)$, which is isometric to $E_8(-2)$, gets killed by the quotient map, and there is an extra algebraic class in $Y$, arising from the terminalization of the quotient $X/\iota$, see e.g.\cite{CGKK}). Therefore, the signature of $\T_X$ is $(2,21-\rho(X))$ if and only if the signature of $\T_Y$ (which coincides with the one of $\T_F$) is $$(2, 14-(\rho(X)-7))=(2,21-\rho(X))).$$ On the other hand, $F$ is projective if and only if the signature of $\T_F$ is $(2,20-\rho(F))$, and we conclude that $X$ is projective if and only if $F$ is projective and that  $\rho(X)=\rho(F)+1$.

In Theorem \ref{them: proof of the conj} we proved that if $F$ is a generalized Nikulin surface, then there is a primitive embedding  of $E_7(-2)$ in $\NS(F)$. Since $E_7(-2)$ is a negative definite lattice of rank 7, $\rho(F)\geq 7$ and if there exists a positive class in $\NS(F)$, then it is not contained in $E_7(-2)\subset NS(F)$. So, $\rho(F)>7$ and there is a class with a positive self intersection contained in the orthogonal complement to the copy of  $E_7(-2)$ in $\NS(F)$.

Equivalently, one can observe that if $X$ is a projective manifold of $K3^{[2]}$-type with a symplectic involution, then $\rho(X)\geq 9$, which implies $\rho(F)\geq 8$.\end{proof}

We proved that the minimal possible Picard number of a projective generalized Nikulin surface $F$ is $8$ and by the arguments in the proof, it is clear that in this case $\NS(F)$ is an overlattice of finite index of $\langle 2d\rangle \oplus E_7(-2)$ for a positive integer $d$, in which $E_7(-2)$ is primitively embedded and $\langle 2d\rangle$ is spanned by a class $L$ which is a primitive generator of the orthogonal complement of $E_7(-2)$ in $\NS(F)$. In particular, both $\langle 2d\rangle$ and $E_7(-2)$ are primitively embedded in $\NS(F)$. Hence, we are now interested in identifying the lattices which are overlattices of finite index of $\langle 2d\rangle \oplus E_7(-2)$ in which both $\langle 2d\rangle$ and $E_7(-2)$ are primitively embedded. 

\begin{lem}\label{lem}\label{lem: overlattices of <2d>+G}
Let $d\in\mathbb{N}_{>0}$ and 
$\Gamma_d$  be an overlattice of finite index $r$ (possibly 1) of $\langle 2d\rangle\oplus E_7(-2)$ in which both $\langle 2d\rangle$ and $E_7(-2)$ are primitively embedded. Then $\Gamma_d$ is one of the following (and the possibilities are mutually excluding):
\begin{enumerate}
\item $\Gamma_d=\langle 2d\rangle\oplus E_7(-2)$;
\item if $d\equiv 0\mod 4$ then $\Gamma_d=(\langle 2d\rangle\oplus E_7(-2))'$, where $(\langle 2d\rangle\oplus E_7(-2))'$ is the unique (up to isometries) overlattice of index 2 of $\langle 2d\rangle\oplus E_7(-2)$ in which both $\langle 2d\rangle$ and $ E_7(-2)$ are embedded primitively;
\item if $d\equiv 2\mod 4$, then $\Gamma_d$ is either $(\langle 2d\rangle\oplus E_7(-2))'$ or $(\langle 2d\rangle\oplus E_7(-2))^*$ where  $(\langle 2d\rangle\oplus E_7(-2))'$ and $(\langle 2d\rangle\oplus E_7(-2))^*$  are the only two non isometric overlattices of index 2 of $\langle 2d\rangle\oplus E_7(-2)$ in which both and $\langle 2d\rangle$ and $E_7(-2)$ are embedded primitively;
\item if $d\equiv 6\mod 8$, then $\Gamma_d=(\langle 2d\rangle\oplus E_7(-2))^{\bullet}$ where $(\langle 2d\rangle\oplus E_7(-2))^{\bullet}$ is the unique (up to isometries) overlattice of index 4 of $\langle 2d\rangle\oplus E_7(-2)$ in which both and $\langle 2d\rangle\oplus E_7(-2)$  are embedded primitively.
\end{enumerate}
\end{lem}
\begin{proof}
The discriminant form of $E_7(-2)$ is $u(2)^{\oplus 3}\oplus \Z_4(\frac{1}{4})$, generated by the elements $u_i^{(j)}$, $i=1,2$ and $j=1,2,3$ and $v$ (an order 4 element with self intersection $\frac{1}{4}$). The discriminant form of $\langle 2d\rangle$ is generated by $\ell$, an element of order $2d$ and self intersection $\frac{1}{2d}$.

An overlattice of $\langle 2d\rangle\oplus E_7(-2)$ is associated to an 
isotropic subgroup of $\langle u_i^{(j)},v,\ell\rangle$ and the fact that both $\langle 2d\rangle$ and $E_7(-2)$ are primitively embedded in the overlattice implies that any non zero element of the subgroup has  non trivial components both in $\langle u_i^{(j)},v,\rangle$ and in $\langle \ell\rangle$;  as a consequence, since $\langle \ell\rangle$ is cyclic and the elements in $\langle u_i^{(j)},v,\ell\rangle$ have order 2 or 4, this forces the isotropic subgroup to be of order 2 or 4, and cyclic as well. 
The isotropic vector which determines the overlattice 
has necessarily
either $d\ell$ or $\frac{d}{2}\ell$ if $d\equiv 0\mod 2$ as component in  $\langle \ell\rangle$. We observe that $(d\ell)^2=\frac{d}{2}$ and $\left(\frac{d}{2}\ell\right)^2=\frac{d}{8}$.

In $u(2)^{\oplus 3}\oplus \Z_4(\frac{1}{4})$ there are: one orbit of non trivial elements of order 2 with self intersection 0 (represented by $u_1^{(1)}$); two orbits of elements of order 2 and self intersection 1 (represented by $u_1^{(1)}+u_2^{(1)}$ and $2v$); two orbits of elements of order 4, for one of them the self intersection is $\frac{1}{4}$ (and the representative is $v$), for the other the self intersection is $\frac{5}{4}$ (and the representative is $v+u_1^{(1)}+u_2^{(1)}$). So
\begin{itemize}\item if $d\equiv 0 \mod 4$, $d\ell+u_1^{(1)}$ is the unique isotropic vector of order 2 (up to isometries);
\item if $d\equiv 2 \mod 4$, there are exactly two non equivalent isotropic vectors of order 2 up to isometries, which are $d\ell+u_1^{(1)}+u_2^{(1)}$ and $d\ell+2v$
\item if $d\equiv 6 \mod 16$, $\frac{d}{2}\ell+v+u_1^{(1)}+u_2^{(1)}$ is the unique isotropic vector of order 4 (up to isometries);
\item if $d\equiv 14 \mod 16$, $\frac{d}{2}\ell+v$ is the unique isotropic vector of order 4 (up to isometries).\end{itemize}
We see that: if $d\equiv 0\mod 4$ there exists a unique overlattice of index 2, if $d\equiv 6,14\mod 16$ (i.e.~ $d\equiv 6\mod 8$) there are 3 possible overlattices, one of index 4 and two of index 2, if $d\equiv 2 \mod 8$ there are exactly two overlattices both of index 2. No other overlattices satisfying the assumptions can be constructed, up to isometries.

We use \cite[Proposition 1.4.1]{N} to compute the discriminant quadratic form $q_{A_{\Gamma_d}}$ of these overlattices $\Gamma_d$, and we obtain:
\begin{itemize}
    \item if $d\equiv 0 \mod 4$ and $H=\langle d\ell+u_1^{(1)}\rangle$, $H^\perp/H =\langle\ell+u_2^{(1)},u_i^{(j)},\ j=2,3,v\rangle$, hence $q_{A_{\Gamma_d}}:=\Z/2d\Z\left(\frac{1}{2d}\right)\oplus u(2)^{\oplus 2}\oplus \Z/4\Z\left(\frac{1}{4}\right) $; 
\item if $d\equiv 2 \mod 4$ and $H=\langle d\ell+u_1^{(1)}+u_2^{(1)}\rangle$, $H^\perp/H =\langle\ell+u_1^{(1)},
u_i^{(j)},\ j=2,3,v\rangle$, hence $q_{A_{\Gamma_d}}:=\Z/2d\Z\left(\frac{1}{2d}\right)\oplus u(2)^{\oplus 2}\oplus \Z/4\Z\left(\frac{1}{4}\right) $;
\item  if $d\equiv 2 \mod 4$ and $H=\langle d\ell+2v\rangle$, $H^\perp/H =\langle\overline{\ell+v},2v,u_i^{(j)},\ j=1,2,3\rangle$ where $\overline{\ell+v}$ is the class of $\ell+v$ in the quotient by $H$; observe that $d(\ell+v)\in H$ so that $\overline{\ell+v}$ has order $d$. Hence $q_{A_{\Gamma_d}}:=\Z/\frac{d}{2}\Z\left(\frac{d+2}{d}\right)\oplus u(2)^{\oplus 3}\oplus q_k$ where $d=4k+2$ and $q_k$ is as in \ref{eq: form qk}.
\item if $d\equiv 6 \mod 16$ and $H=\langle \frac{d}{2}\ell+v+u_1^{(1)}+u_2^{(1)}\rangle$, $H^\perp/H =\langle\ell+v+u_1^{(1)},\frac{d}{2}\ell+3v,u_i^{(j)},\ j=2,3\rangle$, hence $q_{A_{\Gamma_d}}:=\Z/\frac{d}{2}\Z\left(\frac{8}{d}\right)\oplus v(2)\oplus u(2)^{\oplus 2}$;
\item if $d\equiv 14 \mod 16$ and $H=\langle\frac{d}{2}\ell+v\rangle$, $H^\perp/H =\langle\overline{\ell+3v},u_i^{(j)},\ j=1,2,3\rangle$ where $\overline{\ell+3v}$ is the class of $\ell+3v$ in the quotient by $H$; observe that $\frac{d}{2}(\ell+3v)\in H$ so that $\overline{\ell+3v}$ has order $\frac{d}{2}$. Hence $q_{A_{\Gamma_d}}:=\Z/\frac{d}{2}\Z\left(\frac{d+2}{4d}\right)\oplus u(2)^{\oplus 3}$.
\end{itemize}
This shows that the enumerated overlattices belong indeed to different genera.\end{proof}

In \eqref{disc G} we considered a generating set for the discriminant group of $E_7(-2)$, hence, by using that set of generators, we can now exhibit elements in the discriminant group which satisfy the properties of $u_i^{(j)}$ and $v$ appearing in the previous proof. They can be chosen in different ways, which a posteriori coincide up to isometries, and here we fix a choice:
\begin{equation}\label{eq: choice of elements in discriminant of G}u_1^{(1)}=\frac{w_1+w_4}{2},\ \ \ u_2^{(1)}=\frac{w_2+w_4}{2} \mbox{ (observe }u_1^{(1)}+u_2^{(1)}\equiv\frac{w_1+w_2}{2}\mbox{)},\ \ \ v=\frac{w_1-w_2-w_5+w_6+\hat{w}}{4}.\end{equation}

\begin{corollary}\label{prop: Ns e T of projective F}
Let $F$ be a projective generalized Nikulin surface and $\rho(F)=8$, then one of the following holds
\begin{enumerate}
\item $\NS(F)\simeq \langle 2d\rangle\oplus E_7(-2)$, where $d\in\mathbb{N}_{>0}$ and $\T_F\simeq U^{\oplus 2}\oplus\langle -2d\rangle\oplus N\oplus \langle -4\rangle$;
\item $d\equiv 0\mod 2$ and $\NS(F)\simeq (\langle 2d\rangle\oplus E_7(-2))'$ and $\T_F\simeq  U(2)^{\oplus 2}\oplus\langle -2d\rangle\oplus E_8(-1)\oplus \langle -4\rangle$;
\item $d\equiv 2\mod 4$ and $\NS(F)\simeq (\langle 2d\rangle\oplus E_7(-2))^*$ and  
$\T_F\simeq U^{\oplus 2}\oplus K_{\frac{d}{2}}(2)\oplus N$;
\item $d\equiv 6\mod 8$ and $\NS(F)\simeq (\langle 2d\rangle\oplus E_7(-2))^{\bullet}$ and $\T_F\simeq U(2)^{\oplus 2}\oplus K_{\frac{d}{2}}(2)\oplus E_8(-1)$.
\end{enumerate}
\end{corollary}
\begin{proof}
As observed in Lemma \ref{lemma: X proj iff F proj and rho(F)}, if $F$ is a generalized Nikulin surface which is moreover projective, its N\'eron--Severi group contains primitively the lattice $E_7(-2)$ and a positive class, which gives the polarization of the surface. So $\NS(F)$ is an overlattice of $\langle 2d\rangle\oplus E_7(-2)$ in which both $E_7(-2)$ and $\langle 2d\rangle$ are primitively embedded for a certain $d$. If $\rho(F)=8$, the inclusion $\langle 2d\rangle\oplus E_7(-2)\hookrightarrow \NS(F)$ has finite index $r$. This implies that $\NS(F)$ is one of the lattices classified in Lemma \ref{lem: overlattices of <2d>+G}.
Each lattice in Lemma \ref{lem: overlattices of <2d>+G} satisfies the assumptions in \cite[Theorem 1.14.4]{N} and hence admits a unique (up to isometries) primitive embedding in $\Lambda_{K3}$. The discriminant form of $\NS(F)$ is the one computed in Lemma \ref{lem: overlattices of <2d>+G} and the one of $\T_F$ is its opposite. Moreover, by \cite[Corollary 1.13.3]{N}, the lattices $\T_F$ are uniquely determined by their signature and discriminant form. By comparing this discriminant form with the ones given in Section \ref{sec: preliminaries} (and observing that 
$\Z/{\frac{d}{2}}\Z\left(-\frac{d+2}{d}\right)=\Z/{\frac{d}{2}}\Z\left(-\frac{d+2}{4d}\right)$ if $d\equiv 14\mod 16$) one identifies the transcendental lattice as listed above. \end{proof}

\begin{rem}
 Lemma \ref{lem: overlattices of <2d>+G} and Corollary \ref{prop: Ns e T of projective F} correct a mistake contained in \cite[Table 3 \& Proposition 3.6]{CGKK}, where we erroneously stated that the orthogonal of \[\NS(Y_2)=\left[\begin{matrix}
     t-1&2\\
     2&-4
 \end{matrix}\right]\]
 inside $\Lambda_N$ is always isometric to $U(2)^{\oplus 2}\oplus K_{t}(2)\oplus E_7(-1)\oplus \langle -2\rangle $, while this is correct only for  $t\equiv 1 \mod 4$. The correct transcendental of  $Y_2$ in \cite[Proposition 3.6]{CGKK} for any value of $t\equiv 1 \mod 2$ is $ U^{\oplus 2}\oplus K_t(2)\oplus N $.
\end{rem}

\begin{corollary}\label{cor: families of projective K3}
The family of projective generalized Nikulin surfaces up to isomorphisms is the following union of countably many loci, each of dimension 12, of lattice-polarized $K3$ surfaces:
$$\bigcup_{d\in \mathbb{N}} \{(\langle 2d\rangle\oplus G)-\mbox{pol. K3s}\}\cup \bigcup_{d\in 2\mathbb{N}} \{(\langle 2d\rangle\oplus G)'-\mbox{pol. K3s}\}\cup\bigcup_{d\equiv_42} \{(\langle 2d\rangle\oplus G)^\star-\mbox{pol. K3s}\}\cup \bigcup_{d\equiv_86} \{(\langle 2d\rangle\oplus G)^{\bullet}-\mbox{pol. K3s}\}.$$\end{corollary}
\proof The lattice polarization which appear in the statement are exactly the lattices $\Gamma_d$ in Lemma \ref{lem: overlattices of <2d>+G}. If a $K3$ surface is $\Gamma_d$-polarized, then it is projective and its N\'eron--Severi group primitively contains $E_7(-2)$, therefore it is a generalized Nikulin surface by Theorem \ref{thm: generalizes Nikulin iff generalized lattice in NS}. Viceversa, if $F$ is a projective Nikulin surface, then $\NS(F)$ contains primitively $E_7(-2)$ (by Theorem \ref{thm: generalizes Nikulin iff generalized lattice in NS}) and a class with positive self intersection, orthogonal to $E_7(-2)$, therefore it is $\Gamma_d$-polarized.\endproof

Let $W$ be a $K3$ surface which is $\Gamma_d$-polarized, where $\Gamma_d$ is one of the lattices in Lemma \ref{lem: overlattices of <2d>+G}. Then $W$ is a generalized Nikulin surface, i.e.~ it is a surface contained in a manifold of $K3^{[2]}$-type $X_W$ and fixed by a symplectic involution $\iota\in \Aut(X_W)$. Let us denote $Y_W$ the terminalization of the quotient $X_W/\iota$. So the surface $W$ is associated to two 4-folds $X_W$ and $Y_W$. If $W$ is projective and $\rho(W)=8$, then $X_W$ and $Y_W$ are projective and $\rho(X_W)=9$, $\rho(Y_W)=2$. In particular,  $\NS(W)\simeq \Gamma_d$ for a certain $d\in \mathbb{N}$ and $\Gamma_d$ as in Corollary \ref{prop: Ns e T of projective F} and the pairs  $(\NS(X_W), \T_{X_W})$ and $(\NS(Y_W), \T_{Y_W})$ are classified in \cite{CGKK}. In the following proposition we relate $\NS(W)$ with $(\NS(X_W), \T_{X_W})$ and $(\NS(Y_W), \T_{Y_W})$.

Observe that in view of Corollary \ref{cor: families of projective K3}, we are not just associating to a $K3$ surface $W$ the 4-folds $X_W$ and $Y_W$, but to a family of $K3$ surfaces (the ones $\Gamma_d$-polarized) a family of manifolds of $K3^{[2]}$-type and a family of Nikulin orbifolds (the ones determined by the pairs $(\NS(X_W), \T_{X_W})$ and $(\NS(Y_W), \T_{Y_W})$).

\begin{prop}\label{prop:relation NS(F), NS(X), NS(W)}
Let $W$ be a $K3$ surface such that $\NS(W)=\Gamma_d$, where $\Gamma_d$ is one of the lattices in Lemma \ref{lem: overlattices of <2d>+G}, $X_W$ and $Y_W$ as above. The Picard lattices and the transcendental lattices of $X_W$ and $Y_W$ are uniquely determined by $\Gamma_d$ as in Table \ref{eq: table conrrespondence between NS(F) and family of X}, where we refer to \cite[Section 2]{CGKK} for the notation of the lattices $\Lambda_{t}$, $\widetilde{\Lambda_t}$ and embeddings $j$ inside $\Lambda_{K3^{[2]}}$. 
\begin{align}\label{eq: table conrrespondence between NS(F) and family of X}\begin{array}{|c|c|c|c|c|c|}
\hline
d&\NS(W)&\T_W=\T_{Y_W}& \NS(Y_W)&\NS(X_W)&j\\
\hline  
d\in\mathbb{N}_{>0}&\langle 2d\rangle\oplus G&U^{\oplus 2}\oplus\langle -2d\rangle\oplus N\oplus \langle -4\rangle&\langle 2d\rangle\oplus \langle -4\rangle&\widetilde{\Lambda_{4d}}:=\left(\langle 4d\rangle\oplus E_8(-2)\right)'&\widetilde{j} \\
\hline
d\equiv 0\mbox{ mod }2&(\langle 2d\rangle\oplus G)'&U(2)^{\oplus 2}\oplus\langle -2d\rangle\oplus E_8(-1)\oplus \langle -4\rangle&\langle 2d\rangle\oplus \langle -4\rangle&\langle d\rangle\oplus E_8(-2)&j_1\\
\hline
d\equiv 2\mbox{ mod } 4&(\langle 2d\rangle\oplus G)^*&U^{\oplus 2}\oplus K_{\frac{d}{2}}(2)\oplus N &\left[\begin{array}{cc}\frac{d}{2}-1&2\\2&-4\end{array}\right]&\langle d\rangle\oplus E_8(-2)&j_2\\
\hline
d\equiv 6\mbox{ mod } 8&(\langle 2d\rangle\oplus G)^{\bullet}&U(2)^{\oplus 2}\oplus K_{\frac{d}{2}}(2)\oplus E_8(-1)&\left[\begin{array}{cc}\frac{d}{2}-1&2\\2&-4\end{array}\right]&\langle d\rangle\oplus E_8(-2)&j_3\\
\hline\end{array}\end{align}
\end{prop}
\begin{proof} Since $E_7(-2)$ is primitively embedded in $\Gamma_d$, it follows that $W$ is a generalized K3 surface and there exists a Nikulin orbifold $Y_W$ such that $\T_W$ and $\T_{Y_W}$ are Hodge isometric, by Theorem \ref{them: proof of the conj}. Since the choice of $\Gamma_d=\NS(W)$ uniquely determines $\T_W$, it also uniquely determines $\T_{Y_W}$ and therefore $\NS(Y_W)$, by \cite[Table 3]{CGKK}. Moreover, in \cite[Table 3]{CGKK} the data $(\T_{Y_W}, \NS(Y_W))$ is associated to the choice of $\NS(X_W)$ and of its embedding in $U^{\oplus 3}\oplus E_8(-1)^{\oplus 2}\oplus \langle -2\rangle$ (and once one fixes this embedding, one also determines $\T_{X_W}$).\end{proof}

Let $F$ be a generalized Nikulin surface and let $X$ (resp. $Y$) be the associated  4-fold of $K3^{[2]}$-type with a symplectic involution $\iota$ (resp. the terminalization of $X/\iota$). Assume $\rho(F)=8$ and denote $h_F$ the effective generator of the orthogonal complement of $E_7(-2)$ in $\NS(F)$; then $\NS(X)$ (resp. $\NS(Y)$) is an overlattice of finite  index (possibly one) of $\Z h_X\oplus  E_8(-2)$ (resp. $\Z h_Y\oplus \langle -4\rangle$), see e.g. \cite{CGKK}. We are now interested in comparing $h_F$, $h_X$ and $h_Y$. In the proof of Theorem \ref{thm: generalizes Nikulin iff generalized lattice in NS}, we attach to the $K3$ surface $F$ two sublattices $M$ and $L$ of $\mathbb{T}$, which are respectively isometric to the transcendental lattice $\T_F$ and to the lattice $\langle h_F\rangle$. By the proof of Theorem \ref{thm: generalizes Nikulin iff generalized lattice in NS} the sublattice $\lambda_L(L)\subset \mathbb{T}\simeq U(2)^{\oplus 3}\oplus E_8(-1)\oplus \langle -4\rangle$ has to coincide with  $\psi_{\mathbb{T}}^{-1}(h_F)$ where $\psi_{\mathbb{T}}$ is the isometry between $T_Y^{ -4}$ and $\mathbb{T}$ given in Remark \ref{rem: explicitl psi}. Moreover, $h_Y$ is the primitive generator of $\langle \rho^*\pi_*(h_X)\rangle$ and the class $\rho^*\pi_*(h_X)$ is computed in \cite[Proposition 3.5, 3.6, 3.7, 3.8]{CGKK}. Therefore, we obtain $\langle h_F\rangle=\langle\nu^*(h_X)\rangle \simeq \langle\rho^*\pi_*(h_X)\rangle=\langle h_Y\rangle$ for the embedding $j_r$, $r=1,2,3$ where $j_r$ are as in the Table \ref{eq: table conrrespondence between NS(F) and family of X}, and $\langle h_F\rangle \simeq \langle (\rho^*\pi_*(h_X))/2\rangle=\langle h_Y\rangle$ 
for the embedding $\widetilde{j}$. Moreover, by \cite[Proposition 3.5, 3.6, 3.7, 3.8]{CGKK}, one can explicitly describe $\varphi(h_F)=\psi(h_Y)$ as vectors in $\mathbb{T}\simeq U(2)^{\oplus 3}\oplus E_8(-1)\oplus \langle -4\rangle$, obtaining:
\begin{align}\label{eq:psihY}\psi(h_Y)=\left\{\begin{array}{ll}\left(\left(\begin{array}{c}1\\\frac{d}{2}\end{array}\right),\underline{0}, \underline{0}, \underline{0},0\right)&\mbox{ for }j_1\mbox{ and }h_Y^2=2d\mbox{ and }\NS(F)=\left(\langle 2d\rangle\oplus E_7(-2)\right)';\\
\left(\left(\begin{array} {c}2\\2k+2\end{array}\right),\underline{0}, \underline{0}, 2\underline{x},1\right)&\mbox{ for }j_2\mbox{ and }h_Y^2=2d\mbox{ and }\NS(F)=\left(\langle 2d\rangle\oplus E_7(-2)\right)^{\star};\\
\left(\left(\begin{array}{c}2\\(d+2)/4\end{array}\right),\underline{0}, \underline{0}, \underline{0},1\right)&\mbox{ for }j_3\mbox{ and }h_Y^2=2d\mbox{ and }\NS(F)=\left(\langle 2d\rangle\oplus E_7(-2)\right)^{\bullet};\\
\left(\left(\begin{array} {c}1\\k+1\end{array}\right),\underline{0}, \underline{0}, \underline{x},0\right)&\mbox{ for }\widetilde{j}\mbox{ and }h_Y^2=2d\mbox{ and }\NS(F)=\langle 2d\rangle\oplus E_7(-2)\end{array}\right.\end{align}
where the expressions of the vectors are given in $U(2)^{\oplus 3}\oplus E_8(-1)\oplus \langle -4 \rangle$ and in the case $j_2$ (resp. $\widetilde{j}$) $\underline{x}=e_1$ if $d=8k+2$ (resp. $d=2k+1$) and  $\underline{x}=e_1+e_3$ if $d=8k-2$ (resp. $d=2k$).

\begin{rem}\label{rem: the class HF restriction of HX}{\rm Observe that both the transcendental lattices $\T_Y$ and $\T_F$ can be computed as the orthogonal sublattice of $\psi(h_Y)$ in $\mathbb{T}$, where $\psi(h_Y)$ is given in \eqref{eq:psihY}. Moreover, one sees that $h_F$, i.e.~ the primitive generator of $\left(E_7(-2)\right)^{\perp}$ in $\NS(F)$, is the restriction to $F$ of divisor $h_X$ (possibly divided by two), which is the ample generator of $\NS(X)^{\iota^*}$, since it is the class $\varphi_{\mathbb{R}}^{-1}(\psi_{\mathbb{T}}(h_Y))\in H^2(F,\Z)$.}
\end{rem}

In \cite{CGKK} we associated to each manifold of $K3^{[2]}$-type $X$ with a symplectic involution $\iota$ a $K3$ surface, such that $X$ is a moduli space of stable (twisted) sheaves on that $K3$ surface. For this purpose we introduced two different kinds of $K3$ surfaces $S_d$ and $Z_d$, which are identified by their N\'eron--Severi in the following way: $\NS(S_d)= \langle 2d\rangle\oplus \langle -2\rangle^{\oplus 7}$ and $\NS(Z_d)=(\langle 2d\rangle\oplus \langle -2\rangle^{\oplus 7})'$, where the latter lattice is the overlattice of index 2 of $\NS(S_d)$ constructed by adding $\sum b_i/2$, with $\{b_i\}_i$ a basis of $\NS(S_d)=\langle 2d\rangle\oplus \langle -2\rangle^{\oplus 7}$. We observe that in certain cases these surfaces are also either generalized Nikulin surfaces, or geometrically associated to generalized Nikulin surfaces. For this reason we recall the relation proved in \cite{CGKK} between the families of  manifolds of $K3^{[2]}$-type (described as in Table \ref{eq: table conrrespondence between NS(F) and family of X}, by giving the N\'eron--Severi group of $X$ and its embedding in $\Lambda_{K3^{[2]}}$)
and these surfaces. 
\begin{align}\label{eq: X as twisted moduli space of sheaf}\ \ \ \ \ \ \ \ \ \ \ \ \begin{array}{|c|c|c|c|c|c|}
\hline
\NS(X)&j&\mbox{ description of }X&(\NS(K3),\T_{K3}) (K3=S_d,\ Z_d)\\
\hline  
\widetilde{\langle 4d\rangle\oplus E_8(-2)}&\widetilde{j}&M_v(S_{4d})&(\langle 4d\rangle\oplus \langle -2\rangle^{\oplus 7}, U^{\oplus 2}\oplus D_4(-1)\oplus \langle -4d\rangle\oplus \langle -2\rangle^{\oplus 5})\\
\langle d\rangle\oplus E_8(-2)&j_1&M_v(S_{d},\beta)&(\langle d\rangle\oplus \langle -2\rangle^{\oplus 7}, U^{\oplus 2}\oplus D_4(-1)\oplus \langle -d\rangle\oplus \langle -2\rangle^{\oplus 5})\\
\langle d  \rangle\oplus E_8(-2)&j_2&S_{d}^{[2]}&(\langle d\rangle\oplus \langle -2\rangle^{\oplus 7}, U^{\oplus 2}\oplus D_4(-1)\oplus \langle - d\rangle\oplus \langle -2\rangle^{\oplus 5})\\
\langle d\rangle\oplus E_8(-2)&j_3&M_v(Z_{d},\beta)&((\langle d\rangle\oplus \langle -2\rangle^{\oplus 7})', U^{\oplus 2}\oplus N\oplus K_{d/2})\\
\hline
\end{array}\end{align}
The symbol $M_v(V)$ means that $X$ is a moduli space of stable sheaves on $V$ for an appropriate Mukai vector, $M_v(V,\beta)$ means that $X$ is a moduli space of stable twisted sheaves on $V$ for an appropriate Mukai vector $v$ and a Brauer class $\beta$ (the precise vectors and Brauer classes can be found in \cite{CGKK}).

\begin{cor}\label{cor: isom between NS(F) and NS(S),NS(Z)}  
The are the following isometries:
\begin{enumerate}
\item if $d\equiv \pm 1\mod 8$ and $2$ is a square mod $d$, then $\langle 2d\rangle\oplus E_7(-2)\simeq \langle 4d\rangle \oplus \langle -2\rangle^{\oplus 7}\simeq \NS(S_{2d})$.
\item if $d\equiv 2\mod 16$ and $2$ is a square mod $\frac{d}{2}$, $(\langle 2d\rangle\oplus E_7(-2))^*\simeq \langle d\rangle \oplus \langle -2\rangle^7\simeq \NS(S_{d/2})$.
\end{enumerate}\end{cor}
\begin{proof} It suffices to observe that all the lattices listed above are uniquely determined by their discriminant form and that the lattices in the pairs which are stated to be isometric have the same discriminant form.

Indeed, the discriminant quadratic form of $\langle 2d\rangle\oplus E_7(-2)$ is $\left(\frac{1}{2d}\right)\oplus u(2)^{\oplus 3}\oplus \left(\frac{1}{4}\right)$; if $d\equiv\pm 1\mod 8$ its $2$-adic component equals the $2$-adic component of $A_{\langle 4d\rangle \oplus \langle -2\rangle^{\oplus 7}}$ by \cite[Proposition 1.8.2 (g)]{N}. Moreover, the two cyclic summands of order $d$, which are respectively $\left(\frac{2}{d}\right)$ and $\left(\frac{4}{d}\right)$, have isometric forms only if  $2$ is a square mod $d$ (since $d$ is odd, this implies that any square root of $2$ is coprime with $d$).

An analogous proof holds for (2).\end{proof}
Observe that in $(1)$ the condition is satisfied if $d\equiv \pm1\mod 8$ and is a prime number. 

\begin{remark}{\rm To the same very general projective manifold of $K3^{[2]}$-type $X$ admitting a symplectic involution $\iota$, we associate two different $K3$ surfaces: $F_X$, which is the fixed locus of $\iota$ on $X$ and the $K3$ surface $V_X$, which is either of type $S_d$ or of type $Z_d$ as in Table \eqref{eq: X as twisted moduli space of sheaf} and it is such that $X=M_v(V_X,\beta)$ (possibly $\beta=0$), i.e.~ $X$ is a moduli space of stable (possibly twisted) sheaves on $V_X$. By Corollary \ref{cor: isom between NS(F) and NS(S),NS(Z)}, we have $\NS(F_X)\simeq \NS(V_X)$ if $X$ is in the families associated to the embeddings $j_2$ or $\tilde{j}$ and $d$ satisfies some numerical conditions.
A natural question is whether $F_X$ is isomorphic to $V_X$.  
We show in Section \ref{subsec: F degree 4 star} that they are when the embedding is $j_2$ and the degree is  $d=2$.
The situation becomes more complicated in the case of $\tilde{j}$ with $d=4$ as described in Problem \ref{prob14.1}. }\end{remark}

We also observe that there are some relations among the projective generalized Nikulin surfaces with different polarizations, in particular some of them are moduli space of (twisted) sheaves on the others. 
\begin{remark}
The following relations between moduli spaces hold:
\begin{itemize}
\item Let $F$ be such that $\NS(F)=\langle 8d\rangle\oplus E_7(-2)$ and let $H$ be the polarization of degree $8d$ orthogonal to $E_7(-2)$, then the surface $Q=M_v(F)$ where $v=(2,H,2d)$ is a $K3$ surface with $\NS(Q)\simeq \langle 2d\rangle\oplus E_7(-2)$.
\item Let $F$ be such that $\NS(F)=\langle 2d\rangle\oplus E_7(-2)$ and let $A$ be the polarization of degree $2d$ orthogonal to $E_7(-2)$. We can assume that $E_7(-2)$ is embedded in $E_8(-1)^{\oplus 2}\subset H^2(F,\Z)$ and  that $A=du_1^{(1)}+u_2^{(1)}\in U\subset H^2(F,\Z)$. Consider the $B$-field $B:=\frac{1}{2}(u_1^{(1)}+u_2^{(2)})$ 
and the Mukai vector $v_B:=(2d,2dB+A,1)$. Let $R$ be  the surface $R=M_{v_B}(F,\beta)$, then $\NS(R)=\langle 8d\rangle \oplus E_7(-2)$.
\end{itemize}
The previous results follows by the similar results on $\langle 8d\rangle$ and $\langle 2d\rangle$-polarized $K3$ surfaces.

In particular, if $d\equiv 1\mod 2$, then \begin{itemize}\item$(\langle 2d\rangle\oplus E_7(-2))$-polarized $K3$ surfaces are moduli spaces of stable sheaves over a $(\langle 2(4d)\rangle\oplus E_7(-2))$-polarized $K3$ surface, i.e.~  $S_{2d}$ is a moduli space of stable sheaves over $F$ with $\NS(F)=\langle 2(4d)\rangle\oplus E_7(-2)$.

\item $(\langle 8d\rangle\oplus E_7(-2))$-polarized $K3$ surfaces are  moduli spaces of stable twisted sheaves over $S_{2d}$.

In Proposition \ref{prop: 8+G and 2+G related by twisted moduli space} we give a geometric interpretation of the previous 
result if $d=1$.
\end{itemize}
\end{remark}

\subsection{Remarks on projective generalized Nikulin surfaces and projective Nikulin surfaces}
The construction of the N\'eron--Severi lattice of a very general generalized Nikulin surface is obtained by considering the orthogonal to $\sum_iE_i$ in the Nikulin lattice $N$, as explained in Section \ref{section: proofs of 2nd theorem and definition of Generalized Nikulin lattice}. This construction depends on the fact that the very general generalized Nikulin surface $F$ appears as a natural deformation of a very general Nikulin surface $\widetilde{S/\iota_S}$ and that  $\sum_iE_i=\nu^*(\Delta)$. The class $\Delta$ is the class which is algebraic for the  manifold of $K3^{[2]}$-type whose fixed surface is $\widetilde{S/\iota_S}$, and it is transcendental for the  manifold of $K3^{[2]}$-type whose fixed surface is $F$, so that the deformation under consideration consists in moving $\Delta$ from the algebraic part to the transcendental part of the second cohomology group.

Now, let $F$ be a projective generalized Nikulin surfaces (in particular it is not very general); one expects that its N\'eron--Severi group can be obtained as the orthogonal complement to $\nu^*(t)$ for a given $t\in \Lambda_{K3^{[2]}}$ which moves from the algebraic to the transcendental part of the second cohomology group under a specific deformation, but this does not implies that $\nu^*(t)=\sum_iE_i$ and indeed $t$ is not necessarily $\Delta$, as we will see. We already observed that $\NS(F)$ is determined by the choice of the ample divisor $h_X$ in $\NS(X)$ and in particular by the choice of the embedding of such a divisor in $\Lambda_{K3^{[2]}}$. We discuss explicitly two cases: in both of them $\NS(X)=\Z h_X\oplus E_8(-2)$, but we consider the two different embeddings $j_1$ and $j_2$ of $h_X$ in $\Lambda_{K3^{[2]}}$. In the first case $t$ is $\Delta$, and in the second one it is a different class. 

Let us fix the following embedding of $h_X$ in $U^{\oplus 3}\oplus E_8(-1)^{\oplus 2}\oplus\langle -2\rangle$:
\[
j_1(h_X):=\left(\left(\begin{array}{c}1\\d\end{array}\right),
\left(\begin{array}{c}0\\0\end{array}\right),
\left(\begin{array}{c}0\\0\end{array}\right),
\underline{0}, \underline{0},0\right),
\]
(see \cite[Proposition 2.5]{CGKK} for details) and observe that the class $\Delta$ is contained in $\T_X$.
Now we specialize $X$ is such a way that it becomes $S^{[2]}$ for a certain $K3$ surface $S$ endowed with a symplectic involution $\iota_S$, since this corresponds to specializing the generalized Nikulin surface $F$ to the Nikulin surface $\widetilde{S/\iota_S}$. This specialization corresponds to requiring that $\Delta$ becomes algebraic: under our assumptions, $\NS(X)\simeq \langle 2d\rangle\oplus E_8(-2)\oplus \langle -2\rangle$, so $\NS(S)\simeq \langle 2d\rangle\oplus E_8(-2)$ and $\NS(\widetilde{S/\iota})=(\langle 4d\rangle\oplus N)'$, which is the unique even overlattice of index 2 of $\langle 4d\rangle\oplus N$ in which both the summands are primitively embedded. 
To deduce from this the N\'eron--Severi group of the surface $F$ fixed by $\iota$ on $X$, one requires that $\nu^*\Delta=\sum_i E_i\subset \NS(\widetilde{S/\iota})$ becomes transcendental.
So $\NS(F)$ is the orthogonal to $\sum_iE_i=\nu^*(\Delta)$ in $\NS(\widetilde{S/\iota_S})$, which is isometric to $\NS(F)=(\langle 4d\rangle\oplus E_7(-2))'$. 

Let us now consider the embedding $j_2$ with $\NS(X)=\langle 8k+2\rangle\oplus E_8(-2)$:
$$j_2(h_X)=\left(\left(\begin{array}{c}2\\2k+2\end{array}\right),\underline{0},\underline{0},\underline{e_1},\underline{e_1},1\right).$$ The class $\Delta$ is neither transcendental (indeed its intersection with $j_2(h_X)$ is non zero) nor algebraic. In order to specialize $X$ to $S^{[2]}$ for a certain $K3$ surface $S$ with a symplectic involution $\iota_S$, one needs to require that a class $t\in \T_X$ becomes algebraic and that this implies that also $\Delta$ becomes algebraic for $X$. 
The required specialization is obtained for example by choosing $t:=\left(\left(\begin{array}{c}2\\2k+2\end{array}\right),\underline{0},\underline{0},\underline{e_1},\underline{e_1},4k+2\right)\in \T_X$: indeed $t\cdot j_2(h_X)=0$ and therefore $t\in \T_X$; moreover, if $X'$ is a  manifold of $K3^{[2]}$-type such that $\NS(X')$ is an overlattice of finite index of $\NS(X)\oplus \Z t$, then $\Delta=2(t-j_2(h_X))/(4k+1)\in \NS(X')$, and in particular it is algebraic. Recalling that we denote $\delta_X=\frac{1}{2}\Delta$ the primitive class in $\Lambda_{K3^{[2]}}$ corresponding to the exceptional divisor $\Delta$, we have $\delta_X=(t-j_2(h_X))/(4k+1)$. This allows to determine a $\Z$ basis for $\NS(X')$, given by $$j_2(h_X)-\delta_X=\left(\left(\begin{array}{c}2\\2k+2\end{array}\right),\underline{0},\underline{0},\underline{e_1},\underline{e_1},0\right),\ 
\left(\underline{0},\underline{0},\underline{0},\underline{e_i},-\underline{e_i},0\right)\ {i=1,\ldots, 8},\ \delta_X=\left(\underline{0},\underline{0},\underline{0},\underline{0},\underline{0},1\right)$$ Computing the intersection form on this basis one obtains $\NS(X')\simeq (\langle 8k+4\rangle\oplus E_8(-2))'\oplus \langle -2\rangle$, which implies that $X'$ is the Hilbert $S^{[2]}$ for a K3 surface $S$ such that $\NS(S)=\left(\langle 8k+4\rangle\oplus E_8(-2)\right)'$. Moreover, one sees that the symplectic involution on $X'$ is indeed the natural involution induced by the symplectic involution on $S$ acting as minus the identity on the copy of $E_8(-2)$ primitively embedded in $\NS(S)$.
Since $S$ is a $K3$ surface with a symplectic involution such that $\NS(S)=\left(\langle 8k+4\rangle\oplus E_8(-2)\right)'$, $\NS(\widetilde{S/\iota_S})\simeq \langle 4k+2\rangle\oplus N$.

To compute the N\'eron--Severi group of the surface $F$ fixed by $\iota_X$ on $X$, we consider the $K3$ surface $\widetilde{S/\iota_S}$ and we require that the class $\nu^*(t)$ becomes transcendental. Observe that $t=j_2(h_X)+(4k+1)\delta_X$ so we have to compute the orthogonal complement to $\nu^*(j_2(h_X)+(4k+1)\delta_X)$ in $\NS(\widetilde{S/\iota})$. This class corresponds to $$2L+\frac{4k+2}{2}\sum_{i=1}^8E_i\in \NS(\widetilde{S/\iota_S}),$$ where $L$ is the generator of the summand $\langle 4k+2\rangle$ in $\NS(\widetilde{S/\iota_S})\simeq \langle 4k+2\rangle\oplus N$.

\section{Projective models of generalized Nikulin surfaces}\label{section:projective models}
The aim of this section is to describe the projective models of some generalized Nikulin surfaces $F$. 
We first state some general results and then we shall consider a case by case analysis for small degrees. This generalizes the results proved in \cite[Sections 3 and 4]{GS} for Nikulin surfaces; in a specific case (i.e.~ $\Gamma_d\simeq \langle 8\rangle\oplus E_7(-2)$), the geometric description of $F$ allows us to determine the unirationality of the space which parametrizes the surfaces $F$, in the same spirit of the results proved  in \cite{V} for the Nikulin surfaces of degree 8.\\ 

Let $F$ be a generalized Nikulin surface whose N\'eron--Severi group is isometric to $\Gamma_d$ as in Lemma \ref{lem: overlattices of <2d>+G}. Then $E_7(-2)$ is primitively embedded in $\NS(F)$ and its orthogonal complement is generated by a class with positive self-intersection. Up to a change of sign, we can take it to be effective and we denote $h$ the effective generator  of the orthgonoal to $E_7(-2)$. By Remark \ref{rem: the class HF restriction of HX}, the class $h$ is either the restriction of the primitive $\iota$-invariant polarization on $X$, or  its saturation (if the restriction is not already primitive in $\NS(F)$), where $(X,\iota)$ are such that $F$ is the surface in $X$ fixed by $\iota$.
\begin{lemma}\label{lemma: h ample} If $\NS(F)\simeq \Gamma_d$ (see Lemma \ref{lem: overlattices of <2d>+G}) and $h$ is the effective generator of the orthogonal complement of $E_7(-2)$ in $\NS(F)$, then $h$ is ample.\end{lemma}
\proof 
Recall that 
$\Gamma_d$ is an overlattice of index $r=1,2,4$ of $\Z h\oplus E_7(-2)$. If $D$ is a $(-2)$-class
in $\NS(F)$, then $D=\alpha h+\beta g$, $\alpha,\beta\in\frac{1}{4}\mathbb{Z}$ and $g\in E_7(-2)$.
Since $E_7(-2)$ does not contain vectors of self-intersection $-2$, it follows that
$\alpha\neq 0$, i.e.~ there are no $(-2)$-classes orthogonal to $h$, which is then
strictly contained in a chamber of the positive cone. One can assume that it is the
chamber of the ample cone, i.e.~ we can assume $h$ to be ample.
\endproof

\begin{prop}\label{prop: proj models direct sum}
Let $F$ be a generalized Nikulin surface such that $\NS(F)\simeq \langle 2d\rangle\oplus E_7(-2)$. 

If $d\geq 2$, $h$ is very ample and $\varphi_{|h|}(F)\subset\mathbb{P}^{d+1}$ is a smooth surface which admits 7 quadric sections which split in pairs of degree $2d$ curves of genus $d-1$ whose intersection consists of $2d+4$ points.

If $d=1$, $h$ is ample but not very ample and $\varphi_{|h|}(F)\subset\mathbb{P}^{2}$ is a double cover branched on a sextic which admits 7 conics which intersect the branch curve with an even multiplicity in each intersection points, and then they split in the double cover.
\end{prop}
\begin{proof}
We already proved in Lemma \ref{lemma: h ample} that $h$ is ample. Let $E$ be a class with self-intersection $0$. As above, $E=\alpha' h+\beta g$ with
$g\in G$, and then $E\cdot h=2d\alpha'$ for an integer $\alpha'$. This implies that
if $d>1$, then there are no genus $1$ curves on $F$ whose intersection with $h$ is
$2$. Moreover, by construction $h$ is primitive, so that $h$ is an ample class with
self-intersection bigger than $2$, which cannot be $2A$ for an ample divisor $A$ with
self-intersection $2$. By \cite[Theorem 5.2]{SD}, this implies that $h$ is very ample.
On the other hand, if $d=1$, then $h$ is an ample divisor with $h^2=2$, which implies
that $\varphi_{|h|}$ is $2:1$ onto $\mathbb{P}^2$ and the branch locus is a smooth
sextic.

To conclude the proof it suffices to consider the following basis for $\NS(F)$:
\[
\left\{h,\, h-w_1,\, h-\sum_{i=1}^2w_i,\, h-\sum_{i=1}^3w_i,\,
h-\sum_{i=1}^4w_i,\, h-\sum_{i=1}^5w_i,\, h-\sum_{i=1}^6w_i,\,
h-\hat{w}\right\}.
\]
Denoting by $b_j$, for $j=1,\ldots,8$, the elements of this basis, one observes that $b_j^2=2d-4$ and
$h\cdot b_j=2d$ for $j=2,\ldots, 8$. Hence the classes $b_j$ are effective divisors,
i.e.~ they represent (possibly reducible) curves on $F$. Moreover, for each $b_j$
with $j=2,\ldots, 8$, the divisor $d_j:=2h-b_j$ is also effective; indeed
$d_j^2=b_j^2$ and $d_j\cdot h=b_j\cdot h$. Since $2h=b_j+d_j$ for $j=2,\ldots, 8$,
there are $7$ quadric sections (contained in the linear system $|2h|$) which split
into the union of two curves $b_j$ and $d_j$. Generically $b_j$ and $d_j$ are
irreducible; their intersection numbers and genera can be computed explicitly from
their expressions in terms of the basis of $\NS(F)$.

If $d=1$, the inverse image of $\varphi_{|h|}(b_j)$ is the union of $b_j$ and $d_j$,
i.e.~ $\varphi_{|h|}(b_j)$ splits in the double cover. Moreover, since
$b_j+d_j=2h$, the curve $\varphi_{|h|}(b_j)$ has degree $2$ in $\mathbb{P}^2$.
To conclude, recall that a rational curve meets the branch locus with even multiplicity
at each point if and only if it splits in the double cover.
\end{proof}
\begin{rem}\label{rem: vice versa proj models direct sum} In the previous proof we give a different basis of $\NS(F)$, in terms of quadric sections of a certain model of $F$. Then one can reverse the previous statement as follows: let $F$ be a surface with an embedding in $\mathbb{P}^{d+1}$ which admits 7 quadric sections which split each in two curves $B_j$ and $D_j$, $j=1,\ldots, 7$, such that: $B_j$ and $D_j$ have genus $d-1$, and $B_iB_j=2d-2$ $i\neq j$, and $i,j\neq 7$, $B_7B_i=2d-2$ if $i\equiv 1 \mod 2$ and  $B_7B_i=2d$ if $i\equiv 0\mod 2$, $i\neq 7$. Then, $F$ is a generalized Nikulin surface and if it is general among the ones with the described properties, then $\NS(F)\simeq \langle 2d\rangle\oplus E_7(-2)$. Indeed the intersection properties between $h$, $B_j$ and $D_j$ suffice to show that the lattice $\{h,B_1,\ldots B_7\}$ is isometric to $\langle 2d\rangle\oplus E_7(-2)$ and that generically it coincides with $\NS(F)$. Similar considerations hold for the case $d=1$.
\end{rem}

\begin{prop}\label{prop: proj models '}
Let $d$ be an even positive integer $\NS(F)\simeq (\langle 2d\rangle\oplus E_7(-2))'$. 

If $d\neq  2$, $h$ is very ample, $\varphi_{|h|}(F)\subset\mathbb{P}^{d+1}$ is a smooth surface which admits 7 hyperplane sections which split in two curves $C_1$ and $C_2$; denoted $\deg(C_i)$ and $g(C_i)$ the degree and the genus respectively of a curve, it holds:
$$\begin{array}{l}
\deg(C_1)=\deg(C_2)=d,\ \ g(C_1)=g(C_2)=\frac{d}{4}-\frac{1}{2}\mbox{ and }C_1C_2=\frac{d}{2}+3\mbox{ if }d\equiv 2 \mod 4\\
\deg(C_1)=\deg(C_2)=d,\ \ g(C_1)=g(C_2)=\frac{d}{4}-1\mbox{ and }C_1C_2=\frac{d}{2}+4\mbox{ if }d\equiv 0 \mod 4.
\end{array}
$$
If $d=2$, $h$ is ample but not very ample and $\varphi_{|h|}:F\ra\mathbb{P}^3$ is a $2:1$ map to a quadric $Q\simeq \mathbb{P}^1\times \mathbb{P}^1$ and there are 7  curves of bidegree $(1,1)$ on $Q$ which split in the double cover.
\end{prop}
\begin{proof}
The arguments are similar to the ones of Proposition \ref{prop: proj models direct sum}, by recalling that now a basis for $\NS(F)$ is given by $$\left\{\frac{h-w_1-w_2}{2}, w_i,\ i=1,\ldots, 6, \hat{w}\right\}\mbox{ if }d\equiv 2\mod 4\mbox{ and }\left\{\frac{h-w_1-w_4}{2}, w_i,\ i=1,\ldots, 6, \hat{w}\right\}\mbox{ if }d\equiv 0\mod 4.$$

The presence of certain hyperplane sections which split can be proven considering a different basis for $\NS(F)$. In particular, denoted $z=\frac{h-w_1-w_2}{2}$ and $u=\frac{h-w_1-w_4}{2}$, we consider $$\left\{z, z+w_1,z+w_2, z-w_3, z-\sum_{i=3}^4w_i, z-\sum_{i=3}^5w_i,  z-\sum_{i=3}^6w_i, z-\sum_{i=3}^7w_i\right\}\mbox { 
if }d\equiv 2\mbox{ mod }4\mbox{  and }$$
$$\left\{u, u+\sum_{i=1}^2w_i,u-w_2,u-w_3,u+\sum_{i=3}^4w_i,u-w_5,u-\sum_{i=5}^6w_i,u-\sum_{i=5}^7w_i\right\}\mbox{ 
if }d\equiv 0\mbox{ mod }4.$$

As before, denoted $b_j$ the elements of the previous set of generators, one obtains that $\varphi_{|h|}(b_j)$ are curves on the surface as well as $\varphi_{|h|}(h-b_j)$, and since $h=b_j+(h-b_j)$, this guarantees the splitting of 7 hyperplane sections. In particular, the last 7 vectors of the set of generators above produce the curves with the stated properties. The argument is analogous if $d=2$, i.e.~ $\varphi_{|h|}$ is a double cover: if $d=2$, then the class $\frac{h-w_1-w_2}{2}\in \NS(F)$ has self-intersection 0 and $h(\frac{h-w_1-w_2}{2})=2$, which implies that $h$ in not very ample and $$\varphi_{|h|}:F\ra Q\simeq\mathbb{P}^1\times \mathbb{P}^1\subset\mathbb{P}^3$$ where $\frac{h-w_1-w_2}{2}$ and $\frac{h+w_1+w_2}{2}$ are the pullback of the classes of the two rulings of $Q$.

In this case, the inverse images of the classes $\varphi_{|h|}(b_j)$ consist of the union of the two curves $b_j$ and $d_j$, i.e.~ the curves split. The curves are of bidegree $(1,1)$ since $b_j(\frac{h\pm(w_1+w_2)}{2})=1$ for $j=2,\ldots,8$.
\end{proof}

\begin{prop}\label{prop: proj models star}
Let $d\equiv 2\mod 4$ and  $\NS(F)\simeq (\langle 2d\rangle\oplus E_7(-2))^{\star}$. Then $h$ is very ample, so $\varphi_{|h|}(F)\subset\mathbb{P}^{d+1}$ is a smooth surface. If $d\geq 10$,  $\varphi_{|h|}(F)\subset\mathbb{P}^{d+1}$ admits 7 hyperplane sections which split in two curves $C_1$ and $C_2$ and, denoted $\deg(C_i)$ and $g(C_i)$ the degree and the genus respectively of each curve, it holds:
$$\deg(C_1)=\deg(C_2)=d,\ \ g(C_1)=g(C_2)=\frac{d-10}{4}\mbox{ and }C_1C_2=\frac{d}{2}+7.$$
\end{prop}
\begin{proof}
It is analogous to the proof of Proposition \ref{prop: proj models direct sum} and \ref{prop: proj models '}, by considering the set of generators 
$$\left\{x, x-\sum_{i=1}^3w_i,x-w_3, x-\sum_{i=3}^4w_i, x+w_5, x-w_6, x-\sum_{i=6}^7w_i, x-\hat{\omega}\right\}$$
where $x=\frac{h}{2}+2v=\frac{h}{2}+\frac{3w_1-2w_2+w_3-w_5+2w_6+w_7}{4}$ (and $v$ is as in \eqref{eq: choice of elements in discriminant of G}). Observe that if $d=2$, i.e.~ $h^2=4$, then there exist classes whose intersection with $h$ is 2, but they cannot be classes with self-intersection 0: indeed if $a$ is a class such that $ha=2$, then $a=x+\beta h+\gamma g$, with $g\in E_7(-2)$. Since $h^2\equiv g^2\equiv 0\mod 4$ and $hx\equiv gx\equiv hg\equiv 0\mod 2$, one obtains $a^2\equiv x^2\mod 4$ and $x^2\equiv 2\mod 4$. This shows that $h$ is very ample.
\end{proof}

\begin{prop}\label{prop: proj models bullet}
Let $d\equiv 6\mod 8$ and $\NS(F)\simeq (\langle 2d\rangle\oplus E_7(-2))^{\bullet}$.
Then $h$ is very ample and if $d>6$, $\varphi_{|h|}(F)\subset\mathbb{P}^{d+1}$ admits 7 hyperplane sections which split in the union of 2 (not necessarily irreducible) curves $C_1$ and $C_2$; denoted $\deg(C_i)$ and $g(C_i)$ the degree and the genus respectively of a curve, it holds:
$$\begin{array}{l}
\deg(C_1)=\frac{d}{2},\deg(C_2)=\frac{3d}{2}\ \ g(C_1)=\frac{d-22}{16}\ g(C_2)=\frac{9d-22}{16}\mbox{ and }C_1C_2=\frac{3d+38}{8}\mbox{ if }d\equiv 6 \mod 16,\ d\neq 6,\\
\deg(C_1)=\frac{d}{2},\deg(C_2)=\frac{3d}{2}\ \ g(C_1)=\frac{d-14}{16}\ g(C_2)=\frac{9d-14}{16}\mbox{ and }C_1C_2=\frac{3d+30}{8}\mbox{ if }d\equiv 14 \mod 16.
\end{array}$$
\end{prop}\begin{proof}
It is analogous to the proofs of Proposition \ref{prop: proj models direct sum}, \ref{prop: proj models '} and \ref{prop: proj models star}, by considering the last seven elements of the set of generators as classes of the curves in the splitting hyperplane sections
$$\left\{y, y+w_2,y+\sum_{i=2}^3w_i,y+\sum_{i=2}^4w_i,y+w_5,y-w_6,y-\sum_{i=6}^7w_i,y-\hat{w}\right\}\mbox{ if }d\equiv 6\mod 16,\ d\neq 6$$
and $$\left\{c, c-\sum_{i=1}^3w_i,c-w_3, c-\sum_{i=3}^4w_i, c+w_5, c-w_6, c-\sum_{i=6}^7w_i, c-\hat{w}\right\}
\mbox{ if }d\equiv 14\mod 16$$
where $c=\frac{h}{4}+v=\frac{h}{4}+\frac{3w_1-2w_2+w_3-w_5+2w_6+w_7}{8}$ and $y=\frac{h}{4}+v+u_1^{(1)}+u_2^{(1)}-w_4=\frac{h}{4}+\frac{7w_1+2w_2+w_3-w_5+2w_6+w_7}{8}$ 
(and $v$ is an in \eqref{eq: choice of elements in discriminant of G}). Observe that $d\geq 6$, implies that $hk\geq 3$ for every nonzero effective class $k\in \NS(F)$, which suffices to prove $h$ is very ample.
\end{proof}

\begin{rem}{\rm 
In the previous propositions one can be more precise for specific choices of values of $d$. For example, in Proposition \ref{prop: proj models bullet}, when $d$ is big enough the mentioned hyperplane sections split in the union of 4 curves and in Proposition \ref{prop: proj models star} for all the values of $d$ (even $d\leq 10$) there are 2 hyperplanes sections which split in the union of two curves, but if $d\geq 10$, then one needs also the conditions that certain hypersurface sections split in order to reconstruct the geometric properties of a full set of generators for the N\'eron--Severi group.

To obtain the analogue of the Remark \ref{rem: vice versa proj models direct sum} if $\NS(F)$ is a proper overlattice of $\langle 2d\rangle\oplus E_7(-2)$ (i.e.~ in the cases considered in Propositions \ref{prop: proj models '}, \ref{prop: proj models star}, \ref{prop: proj models bullet}), one has first to consider a basis over $\Z$ of the N\'eron--Severi group instead of the set of generators we considered in the proofs of the previous propositions and then to give explicitly the intersections among all the curves which are elements of the chosen basis. To obtain a $\Z$ basis from the set of generators given in the proof of Propositions \ref{prop: proj models '}, \ref{prop: proj models star}, \ref{prop: proj models bullet} it suffices to substitute a class of the given set of generators with an effective class which contains $w$, then one can compute the intersection properties.
}\end{rem}

\subsection{The case $\NS(F)=\langle 2\rangle\oplus E_7(-2)$ and $\NS(X)=(\langle 4\rangle\oplus E_8(-2))'$, embedding $\widetilde{j}$.}

By Proposition \ref{prop: proj models direct sum}, $\varphi_{|h|}:F\ra \mathbb{P}^2$ is a double cover branched along a sextic which admits 7 conics which intersect the branch sextic with even multiplicity in each of their intersection points.
By Corollary \ref{cor: isom between NS(F) and NS(S),NS(Z)}, $\NS(F)\simeq \langle 4\rangle\oplus \langle -2\rangle^{\oplus 7}$, which exhibits $F$ as a seven nodal quartic.
To relate these two models, one denotes $A$ the nef and big divisor of degree 4 such that $\varphi_{|A|}(X)$ is a seven nodal quartic and $n_i$, $i=1,\ldots, 7$ the contracted rational curves, and one observes that the linear system $h:=2A-\sum_in_i$ has degree 2, so that $\varphi_{|h|}:F\ra\mathbb{P}^2$ is a double cover and $\varphi_{|h|}(n_i)$ are the 7 conics in $\mathbb{P}^2$ which split in the double cover $\varphi_{|h|}$.
\subsubsection{ Description of a degree 2 model using equations}\label{Verra}
The aim of this section is to describe  a method to find explicit equations of the sextics above. 
By Proposition \ref{prop:relation NS(F), NS(X), NS(W)}, we see that the generalized Nikulin surface $F$ with $\NS(F)\simeq \langle 2\rangle\oplus E_7(-2)$, is the fixed surface of a symplectic involution on a manifold of $K3^{[2]}$-type $X$ with $\NS(X)\simeq \left(\langle 4\rangle\oplus E_8(-2)\right)'$. These manifolds are symmetric double EPW quartics and are considered in \cite[Section 5]{CGKK} and now we deduce a more precise description of $F$ by this relation. We shall restrict the Lagrangian description of the EPW quartic to the fixed $K3$ surface. In fact to the fixed plane, that is the image of the fixed $K3$ surface through the involution.

Let $\eta: 3\mathcal O(-1)\to 6 \mathcal O_{\mathbb P^2}$ be given by the matrix:
\[
W=\left(\begin{array}{ccc}
2x_3 & 0 &0\\
 0& 2x_2 &0\\
0 & 0 &2x_1\\
0 & x_1 & x_2\\
-x_1 & 0 &x_3\\
-x_2 & -x_3 &0
\end{array}
\right)
\]
where $x_1,x_2,x_3$ are the coordinates of $\mathbb P^2$ and let $E:=\operatorname{coker}  \eta$. Let $\psi: E^{\vee}\to E$ be a general symmetric map. 
\begin{prop}\label{degree2}
    The natural polarization of degree $4$ on $X$ gives a map to $\PP^9$ whose restriction 
    to $F$ is a $2:1$ cover of a Veronese embedding of $\PP^2$ in $\PP^5\subset \PP^9$.
    In such a way $F$ is isomorphic to the double cover of $\mathbb P^2$ branched along the discriminant locus of the conic bundle associated to $\psi$. This gives a model of a general generalized Nikulin surface $F$ with $\NS(F)\simeq \langle 2\rangle\oplus E_7(-2)$. 
\end{prop}

\begin{proof}
    We know that $F$ appears as the fixed locus of the symplectic involution on the symmetric double EPW quartic described in \cite[Section 5A]{CGKK}. The latter is a double cover of a Lagrangian degeneracy locus obtained as follows. We consider the cone over the Segre product $\mathbb P(V_3)\times \mathbb P(\bigwedge^2 V_3^\vee)$ for $V_3$ a $3$-dimensional vector space.  After fixing a basis of $V_3$ this is interpreted as the cone over the locus of rank 1 matrices corresponding to maps $V_3\to V_3^{\vee}$. 
    
    On the cone we have an involution given by transposing the matrices and acting trivially on the additional coordinate corresponding to the vertex. This involution is naturally associated to an involution $\iota$ of the space $$V_3\otimes \bigwedge^2 V_3^{\vee}\oplus V_3^{\vee}\otimes \bigwedge^2 V_3\subset \bigwedge^3 (V_3\oplus V_3^{\vee})/\bigwedge^3 V_3$$ which acts as transposition on both components; the space is equipped with a natural skew symmetric form coming from wedge product. 
    The symmetric EPW quartic section on the cone is induced by a Lagrangian degeneracy locus associated to a Lagrangian space invariant with respect to $\iota$. Note that such a Lagrangian space $A$ is given as a direct sum of Lagrangian spaces $A_1$ and $A_2$ in the two eigenspaces for $\iota$: $H^+_{12}$ of dimension 12 representing pairs of symmetric matrices and  $H^-_6$ of dimension 6 representing pairs of skew-symetric matrices. Note furthermore that the fixed locus of the involution on the cone  admits two components. 
    One of them is a cone over the diagonal in $\mathbb P(V_3)\times \mathbb P(V_3)$. As pointed out in \cite[Proposition 5.3]{CGKK} the generalized Nikulin surface is a double cover of the intersection of the cone over the diagonal with a linear space branched in its intersection with a cubic. The linear space and the cubic are in fact two components of the intersection of the cone over the diagonal with the EPW quartic. 
    Note that for $U$ in the cone over the diagonal the spaces $\bar T_U$ are also invariant and hence also direct sums of Lagrangian subspaces $\bar T_{U,1}$, $\bar T_{U,2}$  of the two eigenspaces. We hence have two Lagrangian degeneracy loci on the cone over the diagonal. One of them is a linear section while the other is a cubic hypersurface.

    More precisely, an affine open subset of the cone over the diagonal can be interpreted as the locus of symmetric rank 1 matrices $B$ with entries $b_{ij}$.  Let furthermore $V_3\otimes \bigwedge^2 V_3^{\vee}$ have coordinates given by entries $m_{ij}$ of a matrix $M$.  
   In these coordinates we get a family parametrized by the space of rank 1 matrices $B$ of quadric forms on $V_3\otimes \bigwedge^2 V_3^{\vee}$:
   $$Q_B(M)=\sum b_{ij} M^{ij},$$ where $M^{ij}$ is the entry of the adjoint matrix to $M$. Now we have an  involution on the space of quadrics in the coordinates $m_{ij}$ induced by transposing the matrix $M$. If we decompose the matrix $M$ as $M=M_++M_-$ with $M_+=\frac{1}{2}(M+M^T)$ and $M_-=\frac{1}{2}(M-M^T)$
   we have 
   $$Q_B(M)=\sum b_{ij} M_+^{ij} + \sum b_{ij} M_-^{ij}.$$
   
 In this setup the considered affine piece of the EPW quartic section is given by the discriminant of the family of quadrics in $m_{ij}$ given by $Q_A +Q_B(M)$, 
 where $Q_A$ is a general quadric in the coordinates $m_{ij}$ invariant with respect to the involution acting as transposition of $M$.
Taking the latter symmetry into account we end up with two degeneracy loci corresponding to the families of quadrics:
\begin{enumerate}
\item $b_0 Q_{A-} +Q_B(M_-)$ in the space of quadrics in $M_-$ i.e.~ in the space of skewsymmetric $3\times 3$ matrices.
\item $b_0 Q_{A+} +Q_B(M_+)$ in the space of quadrics in $M_+$ i.e.~ in the space of symmetric $3\times 3$ matrices.
\end{enumerate}
Now, observe that $Q_B(M_-)$ is a family of quadrics which is represented by the matrix $B$ hence all these quadrics are of rank $\leq 1$. It follows that its determinant is of the form $b_0^2l(B,b_0)$ with $l(B,b_0)$ a linear form. This leads to a hyperplane section of the cone over the Veronese surface. 

Similarly, a direct computation shows that $Q_B(M_+)$ is a family of quadrics of rank $\leq 3$. In particular,  
$$\det (b_0 Q_{A+} +Q_B(M_+))=b_0^3 f(B,b_0),$$
where $f(B,b_0)$ is a cubic. Now restricting $Q_B(M_+)$ to the Veronese surface defined by the hyperplane equation we get a family of quadrics in $\mathbb P^5$ of rank equal to 3. Translating back to Lagrangian degeneracy loci: in the space
$\Sym^2 V_3\oplus \Sym^2 V_3^{\vee}$ we have a natural skew-symmetric form, a Lagrangian $A_+$ and a family of Lagrangians $T_{B,+}$ parametrized by a Veronese surface. We know that 
$A\cap \Sym^2 V_3=0$ while $E_B=T_{B,+}\cap \Sym^2 V_3$ form a bundle of rank 3, that by syzygy computation using Macaulay 2, is isomorphic to $3\mathcal O(-1)$, while its embedding in $\Sym^2 V_3$ is given by the matrix $W$. It follows that the Lagrangian degeneracy locus given by $A$ and the family $\{T_{B,+}\}$ is equal to the Lagrangian degeneracy locus given by the family $\{\bar A_B\}$, the image of $A\cap E_B^{\perp}$ by the natural quotient map $E_B^{\perp}\to E_B^{\perp}/E_B$, and the family $\{\bar T_{B,+}=T_{B,+}/E_B\}$. Moreover, we have $\Sym^2 V_3/E_B \subset E_B^{\perp}\to E_B^{\perp}/E_B$ is a Lagrangian subbundle disjoint from $\bar A_B$ as well as $T_{B,+}$. It follows that the degeneracy locus  is in fact symmetric degeneracy locus between $$(\Sym^2 V_3/E_B)^{\vee} \to \Sym^2 V_3/E_B$$ which is as in the assertion.
It remains to observe that since $A$ is general the induced symmetric degeneracy locus is general of this form, which concludes the proof.\end{proof}

{\begin{rem}{\rm In both  Propositions \ref{degree2}  and \ref{prop: proj models direct sum} we described a model of $F$ as double cover of $\mathbb{P}^2$ where the map to $\mathbb{P}^2$ is given by the ample polarization $h$ (which is the saturation of the restriction of the primitive $\iota$-invariant polarization on $X$). So the two descriptions coincide, which allows us to conclude that the discriminant of the conic bundle described in Proposition \ref{degree2} is a smooth sextic  which admits 7 conics which intersect the branch sextic with even multiplicity in each of their intersection points.}\end{rem}}

\subsection{The case $\NS(F)=(\langle 4\rangle\oplus E_7(-2))'$ and $\NS(X)=\langle 2\rangle\oplus E_8(-2)$, embedding $j_1$.}
By Proposition \ref{prop: proj models '}, if $\NS(F)\simeq (\langle 4\rangle\oplus E_7(-2))'$, then $\varphi_{|h|}:F\ra \mathbb{P}^3$ is a double cover of a quadric such that there are seven $(1,1)$ curves on $Q$ which split in the cover. Here we give extra information on these curves and we completely characterize the generalized Nikulin surfaces which admit a model as double cover of $\mathbb{P}^1\times \mathbb{P}^1$ and which has a minimal Picard number.
Recall, that if a $K3$ surface $S$ admits a model as double cover of $\mathbb{P}^1\times \mathbb{P}^1$, the two projections $\mathbb{P}^1\times \mathbb{P}^1\ra\mathbb{P}^1$ endow $S$ with two genus 1-fibrations such that $E_1E_2=2$ if $E_j$ $j=1,2$ is the class of the fiber of these genus 1 fibrations. Call $\mathcal{E}_i$ the elliptic fibration $\mathcal{E}_i:=\varphi_{|E_i|}$.
\begin{lemma}\label{lem: decrip F (4+G)'}
Let $F$ be a $K3$ surface such that $\NS(F)\simeq (\langle 4\rangle\oplus E_7(-2))'$. Then $\varphi_{|h|}:F\ra \mathbb{P}^1\times \mathbb{P}^1$ is a double cover such that the equivalent following conditions hold: \begin{enumerate}\item there are six rational curves $R_i$ on $F$ such that $\varphi_{|h|}(R_i)$ is a $(1,1)$ curve in $\mathbb{P}^1\times \mathbb{P}^1$ which splits in the double cover and whose intersection properties are \begin{eqnarray}\label{eq: int Ri's}R_iR_j=0\mbox{ if }1\leq i<j\leq 5\mbox{ or }(i,j)=(1,6),(3,6),(5,6)\mbox{ and }R_2R_6=R_4R_6=2\end{eqnarray} and the set $\{E_j,R_i,\ j=1,2\ i=1,\ldots, 6\}$ is a $\Z$-basis of $\NS(F)$; \item the genus 1 fibrations $\mathcal{E}_1$ and $\mathcal{E}_2$ admit a common zero section $\mathcal{O}$ and both their Mordell--Weil groups are isomorphic to  $\Z^6$ and are generated by the same set of 6 sections $\mathcal{S_i}$ such that \begin{eqnarray}\label{eq: int sections 4'}\mathcal{S}_i\mathcal{S}_j=0\mbox{ if }1\leq i<j\leq 5,\ \mathcal{O}\mathcal{S}_i=\left\{\begin{array}{l}0\mbox{ if }i=1,2,3,\\2\mbox{ if }i=4,5,6,\end{array}\right.\  \mathcal{S}_6\mathcal{S}_i=\left\{\begin{array}{l}4\mbox{ if }i=1\\2\mbox{ if }i=2,\ldots,5,\end{array}\right. \end{eqnarray}\end{enumerate}
\end{lemma}
\begin{proof}
To prove the condition (1) it suffices to exhibit a basis of $\NS(F)$ whose classes represents the required curves, slightly modifying the set of generators considered in the proof of Proposition \ref{prop: proj models '} (we keep the same notation of that proof). To this purpose let us consider the $\Z$-basis of $\NS(F)$ given by $$\{z,h-z,z+w_2, z-w_3, z-\sum_{i=3}^4w_i, z-\sum_{i=3}^5w_i,  z-\sum_{i=3}^6w_i, z+w_2+\hat{w}\}.$$ Then $E_1:=z$ and $E_2:=h-z$ are genus 1 curves such that $E_1E_2=2$, hence $\varphi_{|h|}=\varphi_{|E_1+E_2|}:F\ra\mathbb{P}^1\times \mathbb{P}^1$ is a double cover. Calling $R_i$, $i=1,\ldots, 6$ the remaining curves, one observes that $R_i^2=-2$, $R_iE_j=1$, $j=1,2$, $i=1,\ldots, 6$, which implies that the $R_i's$ correspond to rational curves, which are mapped to $(1,1)$ curves over $\mathbb{P}^1\times \mathbb{P}^1$ and which split in the double cover (otherwise their intersection with the $E_j$'s should be 2, since the surface is a double cover of $\mathbb{P}^1\times\mathbb{P}^1$). The intersection properties among the $R_i$'s can be directly computed using the basis given above.

To show (2), observe that the two classes $E_i$ define the two elliptic fibration $\mathcal{E}_i:=\varphi_{|E_i|}$ whose class of the fiber is $E_i$ and $E_i^2=0$, $E_1E_2=2$. A common section $\mathcal{O}$ 
is such that $\mathcal{O}^2=-2$ and $\mathcal{O}E_j=1$. Denote $R_6$ the zero section of the fibrations; and $R_i$, $i=1,\ldots, 5$ the sections $\mathcal{S}_i$. Putting $\mathcal{S}_6=E_1+E_2-\mathcal{S}_1$ one sees that the intersection properties given in point \eqref{eq: int sections 4'} are equivalent to \eqref{eq: int Ri's}. Moreover, the fact that the section $\mathcal{S}_i$'s generate the Mordell--Weil groups, implies that $\{E_1,\mathcal{O},\mathcal{S}_1\ldots, \mathcal{S}_6\}$ is a $\Z$-basis for the N\'eron--Severi group and vice versa. Since one can replace $\mathcal{S}_6$ with $E_2$, one obtains that $$\{E_1,E_2,\mathcal{S}_k, k=1,\ldots, 5,\mathcal{O}\}=\{E_1,E_2,R_i,\ldots i=1,\ldots, 6\}$$ is a basis of $\NS(F).$\end{proof}

\begin{prop}\label{prop: vice versa <4>+G'}
Let $W$ be a $K3$ surface with $\rho(W)=8$ and admitting two to one map $f:W\ra \mathbb{P}^1\times \mathbb{P}^1$ and such that there exist two  the genus 1 fibrations $\mathcal{E}_1$ and $\mathcal{E}_2$ which have a common zero section $\mathcal{O}$ and common generators of the Mordell--Weil group $\mathcal{S}_i$ which satisfy \eqref{eq: int sections 4'}. 
 
 Then $\NS(W)\simeq (\langle 4 \rangle\oplus E_7(-2))'$ and in particular $W$ is a generalized Nikulin surface.\end{prop}
\begin{proof}
Let us assume that $W$ admits a model as double cover of $\mathbb{P}^1\times \mathbb{P}^1$ and that $\mathcal{E}_i:=\varphi_{|E_i|}$ are the associated genus 1 fibrations. Let $\mathcal{O}$ be their common zero section, $\mathcal{S}_i$, $i=1,\ldots 5$, be the common generators of their Mordell--Weil group and $\mathcal{S}_6^{(j)}$ be the generator of $\mathrm{MW}(\mathcal{E}_j)$ with $j=1,2$. Then $\{E_1,E_2,\mathcal{O},\mathcal{S}_i,\ i=1,\ldots, 5\}$ is a $\Z$-basis of a lattice which is isometric to $(\langle 4\rangle\oplus E_7(-2))'$ by the proof of Lemma \ref{lem: decrip F (4+G)'}. 
\end{proof}

We now investigate the relations between the geometry of $F$ and of the manifold of $K3^{[2]}$-type $X$ admitting a symplectic involution $\iota$ such that $F$ is the fixed surface of $\iota$.

By Proposition \ref{prop:relation NS(F), NS(X), NS(W)}, it follows that $\NS(X)\simeq \langle 2 \rangle \oplus E_8(-2)$ and that it is embedded in $H^2(X,\Z)$ with the embedding $j_1$ (see \cite[Table 3]{CGKK} for the definition).

This case was already studied a bit in \cite[Section 4]{CGKK}.
Recall that manifolds of $K3^{[2]}$-type of degree $2$ admitting a symplectic involution were described  in \cite{C}. 
The projective model is an EPW sextic $Z\subset \PP^5$  invariant with respect to a linear involution. The involution induces on $\PP^5$ a linear involution of type  $(+,+,+,+,-,-)$.
As shown in \cite[Section 4]{CGKK}, the fixed $\PP^3$ intersects that EPW sextic along the union of a quadric $Q$ and a Kummer quartic $K$.
We can see that the fixed $K3$ surface $F$ maps $2:1$ to the quadric and is ramified along the intersection  $Q\cap K$.

\begin{prop}\label{prop 10.10}
    A generalized Nikulin surface $F$ with $\NS(F)\simeq (\langle 4\rangle\oplus E_7(-2))'$ is a double cover of a smooth quadric branched along its intersection with a Kummer quartic.
    Moreover, a general such double cover is a generalized Nikulin surface.
\end{prop}
\begin{proof} By \cite[Section 4]{CGKK} a generalized Nikulin surface $F$ with $\NS(F)\simeq (\langle 4\rangle\oplus E_7(-2))'$ is a double cover of a quadric $Q$ branched on $Q\cap K$ for a Kummer surface $K$. It remains to show that any double cover of a quadric branched on the intersection with a Kummer surface is a generalized Nikulin surface with N\'eron--Severi group which primitively contains $(\langle4\rangle\oplus E_7(-2))'$ (and that generically the N\'eron--Severi group is exactly $(\langle4\rangle\oplus E_7(-2))'$). 
 Since the dimension of the moduli space of $( \langle4\rangle\oplus E_7(-2))'$ polarized $K3$ surfaces is 
 $12$ it is enough to prove that the dimension of the family of double covers of quadrics branched along an intersection with a Kummer surface is at most $12$.
 The last follows from the fact that the family of quadrics (respectively of projective Kummer surfaces) is $9$-dimensional  (respectively $3$-dimensional).

   Another computational approach to the proof of Proposition \ref{prop 10.10} is to extend a symmetric determinantal
   description of a Kummer surface (see \cite[Section 2]{IKKR}) to get an equation of an EPW sextic so of a Lagrangian space.
\end{proof}

By merging Proposition \ref{prop: vice versa <4>+G'} and Proposition \ref{prop 10.10}, we see the following.
\begin{corollary}
Let $W$ be a $K3$ surface as in \ref{prop: vice versa <4>+G'}, then it is the double cover of a quadric branched along its intersection with a Kummer quartic.\end{corollary}

\subsection{The case $\NS(F)=(\langle 4\rangle\oplus E_7(-2))^*$ and $\NS(X)=\langle 2\rangle\oplus E_8(-2)$, embedding $j_2$.}\label{13.3}\label{subsec: F degree 4 star}

By Corollary \ref{cor: isom between NS(F) and NS(S),NS(Z)}, $\NS(F)\simeq \langle 2\rangle\oplus \langle -2\rangle^{\oplus 7}$.
\begin{prop}\label{prop: double cover dP2} Let $\NS(F)\simeq (\langle 4\rangle\oplus E_7(-2))^*$. Then $F$ is a double cover of a del Pezzo surface of degree 2 and it is a $4:1$ cover of $\mathbb{P}^2$. 

Vice versa, if $W$ is a $K3$ surface which is a general double cover of a del Pezzo surface of degree 2, then $\NS(W)\simeq (\langle 4\rangle\oplus E_7(-2))^*$ and so $W$ is a generalized Nikulin surface.
\end{prop}
\begin{proof} We already observed that $\NS(F)\simeq \langle 2\rangle\oplus \langle -2\rangle^{\oplus 7}$, and called $L$ the first generator, one sees that $\varphi_{|L|}:F\ra\mathbb{P}^2$ is a degree 2 cover which contracts 7 rational curves, so that $F$ is the double cover of the blow up $V$ of $\mathbb{P}^2$ in 7 points, which are assumed to be in general position. Then, $V$ is a del Pezzo surface of degree 2 and $\Pic(V)=\langle \ell, e_1,\ldots , e_7\rangle$ where $\ell$ is the pullback on $V$ of the class of a line in $\mathbb{P}^2$ and $e_i$ are the exceptional divisors. The pullback of these classes to $F$ are called $L$ and $E_i$ respectively and their non trivial intersections are $L^2=2$, $E_i^2=-2$. The class $\eta:=3L-\sum_{i=1}^7E_i$ is the pullback of the anticanonical class on $V$.  Its orthogonal complement in $\NS(F)$ is generated by $\{E_i-E_{i+1},i=1,\ldots,6,\  L-E_4-E_5-E_6\}$ which is isometric to $E_7(-2)$. This implies that $\eta$ coincides with the class $h$ considered before, i.e.~ with the ample generator of the orthogonal complement to $E_7(-2)$ in $\NS(F)$. 
Recall that the anticanonical divisor $-K_V=3\ell-\sum_{i=1}^7e_i$ gives a $2:1$ map $\varphi_{|-K_V|}:V\ra\mathbb{P}^2$ which is a double cover branched on a quartic curve. The composition of the two double covers $F\ra V$ and $V\ra \mathbb{P}^2$ provides a $4:1$ map to $\mathbb{P}^2$. 

Vice versa, the condition that a $K3$ surface is a general double cover of a del Pezzo surface of degree 2, i.e.~ of a blow up of $\mathbb{P}^2$ in 7 points in general position, implies that its N\'eron--Severi group is $\langle 2\rangle\oplus \langle -2\rangle^{\oplus 7}\simeq \langle \left(4\rangle\oplus E_7(-2)\right)^{\star}$.\end{proof}

By Proposition \ref{prop:relation NS(F), NS(X), NS(W)}, $F$ is the fixed locus of a symplectic involution on a manifold of $K3^{[2]}$-type $X$ with $\NS(X)\simeq \langle 2\rangle\oplus E_8(-2)$, embedded in $H^2(X,\Z)$ via $j_2$.

If $H$ is the divisor which spans $\langle 2 \rangle$ in $\NS(X)$, then $H$ gives a $6:1$ map from $X$ to a quadric  $Q\subset \mathbb{P}^5$. The two fixed spaces are $\mathbb{P}^2_{+}$ and $\mathbb{P}^2_{-}$ contained in the quadric. 
By Table \eqref{eq: X as twisted moduli space of sheaf}, $X$ is isomorphic to $S_2^{[2]}$, where $S_2$ is a $K3$ surface with N\'eron--Severi group isometric to $$ \langle 2\rangle \oplus \langle -2\rangle^7\simeq \left(\langle 4\rangle\oplus E_7(-2)\right)^\star.$$ In particular, $S_2$ is as in Proposition \ref{prop: double cover dP2}: it admits a model as double cover of a del Pezzo surface of degree 2, $V$, and a model as quartic in $\mathbb{P}^3$. The cover involution $\gamma$ of $S_2\ra V$ acts trivially on the N\'eron--Severi group of $S_2$ and in particular on the class which gives the embedding of $S_2$ in $\mathbb{P}^3$. So, $X=S_2^{[2]}$ admits a non--symplectic involution $\gamma^{[2]}$ and the Beauville involution $\beta$. The composition $\gamma^{[2]}\circ \beta$ is the symplectic involution defined on $X$ (by the sublattice $E_8(-2)$), as described in \cite[Proposition 3.14]{CGKK}. Observe that in \cite[Proposition 3.14]{CGKK} we proved that $S_2\simeq F$, and not just that they have the same N\'eron--Severi group, as stated in Corollary \ref{cor: isom between NS(F) and NS(S),NS(Z)}.

\subsection{The case $\NS(F)=\left(\langle 8\rangle\oplus E_7(-2)\right)'$ and $\NS(X)=\langle 4\rangle\oplus E_8(-2)$, embedding $j_1$.}\label{subsec:d=4,j1} By Proposition \ref{prop: proj models '}, if $\NS(F)=\left(\langle 8\rangle\oplus E_7(-2)\right)'$, then $F$ admits a model as a complete intersection of three quadrics in $\mathbb{P}^5$ with 7 hyperplane sections which split into the union of two rational curves. Here we present also other properties of the same model, which characterize it completely.

In order to analyze generalized Nikulin surfaces with a degree $8$ polarization as above, we
consider specializations of a well-known construction which relates $\langle
8\rangle$-polarized $K3$ surfaces with $\langle 2\rangle$-polarized $K3$ surfaces:
the $K3$ surfaces with Picard number $1$ and degree $8$, i.e.~ complete intersections
of three quadrics $Q_i$, $i=1,2,3$, in $\mathbb{P}^5$, can also be described as
moduli spaces of stable sheaves on $K3$ surfaces which are double covers of $\mathbb{P}^2$.
The corresponding branch sextic is then the discriminant of the net of quadrics
generated by $Q_1$, $Q_2$, $Q_3$. Conversely, given the branch sextic one can
reconstruct the net of quadrics, if one fixes a theta characteristic on the sextic,
or equivalently a $B$-field, to obtain the complete intersection as a moduli space of stable twisted sheaves. In Section~\ref{subsubsec: net and discriminant sextic} we consider many specializations of this construction to $K3$ surfaces with higher Picard number and in particular to generalized Nikulin surfaces. 

\begin{prop}\label{prop: alternative description 8+G'}
Let $\NS(F)=\left(\langle 8\rangle\oplus E_7(-2)\right)'$, then the image of $\varphi_{|H|}:F\ra\mathbb{P}^5$ admits 6 hyperplane sections each of which splits into the union of two genus 1 curves $C_1^{(i)}$ and $C_2^{(i)}$, $i=1,\ldots, 6$, of degree 4 and a quadric hypersurface which splits into the union of two genus 3 curves $D_1$ and $D_2$. Moreover, each $C_i^{(j)}$ is a bisection for the genus 1 fibration induced by any other $C_h^{(k)}$ and  $$D_1C_1^{(i)}=D_2C_2^{(i)}=3\ \ \ D_1C_2^{(i)}=D_2C_1^{(i)}=5.$$

Vice versa, if $W$ is a $K3$ surface whose N\'eron--Severi group is generated by $h$, $C_1^{(i)}$, $i=1,\ldots,6 $ and $D_1$ where $h$ is a very ample class such that  $\varphi_{|h|}(W)$ is a complete intersection of three quadrics in $\mathbb{P}^5$ with 6 hyperplanes which split into the union of two genus 1 curves, one of which is $C_1^{(j)}$, $i=1,\ldots,6$ and a quadric section which splits into the union of  two genus 3 curves, one of which is $D_1$, and $$C_{1}^{(j)}C_1^{(k)}=2,\ \ j\neq k,\ \ D_1C_1^{(i)}=3,$$ then $\NS(W)=\left(\langle 8\rangle\oplus E_7(-2)\right)'$ and in particular $W$ is a generalized Nikulin surface.
\end{prop}
\proof The proof is analogous to the ones of the Lemma \ref{lem: decrip F (4+G)'} and Proposition \ref{prop: vice versa <4>+G'}: it suffices to find a basis of $(\langle 8\rangle\oplus E_7(-2))'$ which is generated by the classes described in the statement. Let us recall that $u=\frac{h-w_1-w_4}{2}$ and let us consider the basis 
$$\left\{h,h+w_5,u, u+w_1,u-w_2-w_3,u+w_4+\hat{w}, u+w_1+w_3+w_4+w_5-\hat{w},u+w_1-w_6-\hat{w}\right\},$$
and put $D_1=h+w_5$ and $C_1^{(i)}$, $i=1,\ldots, 6$ equals to the last six elements of the basis. Then one shows that both $C_1^{(i)}$ and $h-C_1^{(j)}=:C_2^{(i)}$ are effective and that the intersection properties among the curves $C_i^{(j)}$ are the required ones. Similarly $D_1$ and $2h-D_1=:D_2$ are effective and the intersection properties are the required ones.\endproof

Observe that the rational curves appearing in Proposition \ref{prop: proj models '} are components of the reducible fibers of the genus 1 fibrations defined by the classes $C_1^{(j)}$.

\begin{rem}\label{rem: singular memebers in the net of quadrics}{\rm The curves $C_i^{(j)}$ are degree 4 and genus 1, contained in a hyperplane of $\mathbb{P}^5$.
Since the genus 1 curves are complete intersections of 2 quadrics in $\mathbb{P}^3$, this implies that in the net of quadrics defining $\varphi_{|h|}(F)$ there are 6 quadrics which admit hyperplanes section which split in two $\mathbb{P}^3$'s, i.e.~ there are six quadrics which are singular in a $\mathbb{P}^1$ and so  which are cones over a quadric in $\mathbb{P}^3$.}
\end{rem}

By Proposition \ref{prop:relation NS(F), NS(X), NS(W)}, a generalized $K3$ surface $F$ such that $\NS(F)\simeq\left(\langle 8\rangle\oplus E_7(-2)\right)'$ is the fixed locus of a symplectic involution on a manifold of $K3^{[2]}$-type $X$ such that $\NS(X)=\langle 4\rangle\oplus E_8(-2)$. Let us denote $H$ the polarization of degree 4 invariant for the involution. The map $\varphi_{|H|}$ gives a $1:1$ map to $\PP^9$, and the associated involution has fixed points $\PP^5_+$ and $\PP^3_-$, by \cite[Section 5A]{CGKK}. The fixed surface $F$ maps $1:1$ to a complete intersection of three quadrics $S_{2,2,2}$ in $\PP^5_+$ as above.\\

We now relate the family of the $\left(\left(\langle 8\rangle\oplus E_7(-2)\right)'\right)$-polarized $K3$ surfaces with a different family of $K3$ surfaces, in order to describe its property and in particular to show its unirationality.

Attached to each net of quadrics in $\mathbb{P}^5$ there is a sextic in $\mathbb{P}^2_{(\alpha:\beta:\gamma)}$ which is the zero locus of the determinant of the matrix $\alpha Q_1+\beta Q_2+\gamma Q_3$ where, for $i=1,2,3$, $Q_i$ is the symmetric matrix which represents the quadric $Q_i\subset\mathbb{P}^5$ and $Q_1,Q_2,Q_3$ are three quadrics which generates the net. So the sextic is the discriminant locus of the net and represents the quadrics which are singular. Hence, to each complete intersection of 3 quadrics in $\mathbb{P}^5$ one can associate another $K3$ surface, which is the double cover of $\mathbb{P}^2$ branched on the discriminant locus of the net of quadrics. In particular, one can reinterpret this relation in terms of a moduli space of stable sheaves on a $K3$ surface: let $W_8$ be a $K3$ surface which is the complete intersection of three very general quadrics in $\mathbb{P}^5$; then $\NS(W_8)=\langle 8\rangle=\Z \ell$ and $M_v(W_8)$, where $v=(2,h,2)\in H^*(W_8)$, is a $\langle 2\rangle$-polarized $K3$ surface, and indeed the one which is the double cover of $\mathbb{P}^2$ branched on the discriminant sextic. The geometric properties of the discriminant sextic of this double cover are related with the geometric properties of the net of quadrics defining $W_8$. This construction can be specialized to $K3$ surfaces with Picard number greater than 1, as we do in Section \ref{subsubsec: net and discriminant sextic}.

Let $Z$ be a $K3$ surface which is very general among the ones which admit a model as double cover of $\mathbb{P}^2$ branched on a sextic with 6 nodes, such that there exists a conic passing through all these nodes, i.e.~ $\NS(Z)=\langle \ell, \gamma, r_1,\ldots r_6\rangle$ such that $\ell^2=2$, $\ell r_i=0$, $r_ir_j=-2\delta_{ij}$, $\gamma^2=-2$, $\ell\gamma=2$, $\gamma r_i=1$. In Section \ref{subsubsec: net and discriminant sextic} we prove the following proposition, needed to prove Corollary \ref{8'}.
\begin{prop}\label{propB}
 A general $K3$ surface with N\'eron--Severi group $\left(\langle 8\rangle\oplus E_7(-2)\right)'$ is the moduli space of stable twisted sheaves on a $K3$ surface $Z$ as described above.
 The twisted Mukai vector is $v_B=(2,2B+\ell,1)$ and the Brauer class admits a $B$-lift with $B\ell=\frac{1}{2}$, $B^2=0$,
 $Br_i=B\gamma=\frac{1}{2}$.
\end{prop}

A $(1,2)$ divisor in $\mathbb P^2 \times \mathbb P^5$ gives a two dimensional systems of  quadrics on $\PP^5$ defined as fibers of the fibration over $\PP^2$.
\begin{corr} \label{8'}
A very general generalized Nikulin surface $F$ such that $\NS(F)\simeq \left(\langle 8\rangle\oplus E_7(-2)\right)'$ is the base locus of a net of quadrics given by a $(1,2)$-divisor $D$ in $\mathbb P^2 \times \mathbb P^5$ such that, denoted $p$ the restriction of the projection $\mathbb P^2 \times \mathbb P^5\to \mathbb P^2$ to $D$, there are six points $P_i\in\mathbb{P}^2$ such that $p^{-1}(P_i)$ contains a $\mathbb{P}^3$ for $i=1,\ldots, 6$ and these six points are contained in a conic. 

Moreover, the generalized Nikulin surfaces whose N\'eron--Severi lattice is $\left(\langle 8\rangle\oplus E_7(-2)\right)'$ are parameterized by a unirational variety.
\end{corr}

\begin{proof}
 Let us consider the base locus of the net of quadrics described above. After fixing the configuration of the $\PP^3$'s, the corresponding $(1,2)$ divisors are hyperplanes $H$ in the $(1,2)$-embedding of $\mathbb P^2 \times \mathbb P^5$ containing the span of the Veronese's of the configuration of $\PP^3$'s. 
 The space of $(1,2)$-divisors as above is denoted by $\mathcal{V}$.
Hence, the family $\mathcal{V}$ is unirational.  
    
To show that the family constructed is indeed the family of the $\left(\left(\langle 8\rangle\oplus E_7(-2)\right)'\right)$-polarized $K3$ surfaces we first perform a dimension count and then we show that the generic members coincide.

 We choose 6 points on a conic (dimension $11=6\cdot 2-1$),  then 6 general $\mathbb{P}^3$'s in the 6 $\mathbb{P}^5$ fibers (dimension $48=6\cdot8$). The $\mathbb{P}^3$'s span in general 6 skew $\mathbb P^9$'s in $\mathbb P^{62}=\mathbb P(H^0(\oo(1,2), \mathbb P^2\times \mathbb P^5))$ (the fact that the $\mathbb P^9$'s are skew so that they span a $\mathbb P^{59}$ is checked by a random choice using Macaulay2), we hence have a $2$-dimensional space of additional choices of the $(1,2)$-divisor. We then need to subtract 6 since each of our divisors contains two pencils of $\mathbb{P}^3$'s in each of the distinguished $\mathbb P^5$'s. We hence get a $11+48+2-6=55$-dimensional family, i.e.~  a $7$-codimensional family of $(1,2)$-divisors. This leads to a 7-codimensional subfamily of $K3$ surfaces inside the $19$-dimensional families of polarized $K3$ surfaces of  degree 8, i.e.~ to a $12$-dimensional subfamily of $K3$ surfaces in the family of the $K3$ surfaces which are base locus of a net of quadrics.

    Since a quadric in $\mathbb P^5$ has rank at most 4 if and only if it contains a $\mathbb{P}^3$, we deduce that the discriminants of the nets of quadrics represented by $(1,2)$-divisors from the family $\mathcal V$ satisfy the assumption of Proposition \ref{propB}
    
     Now looking at a random example, using Macaulay2, we can see that the intersection of the net of quadrics satisfying the assumptions is a smooth surface.  
    So it follows from  Section \ref{13.6.1} Case 1, that the geometric conditions determine the degrees of intersection of a $B$-lift with the exceptional curves to be $B.r_i=\frac{1}{2}$ (mod $\mathbb{Z}$). Moreover, Case 2 shows that $B.\ell=\frac{1}{2}$ (mod $\mathbb{Z}$) for the considered quadric bundle (the condition $B^2=0$ can be obtained by choosing an appropriate $B$-lift keeping the other conditions).
It follows from Proposition \ref{propB}, that in this way we obtain generalized Nikulin surfaces.
We conclude that our construction leads to a $12$-dimensional (hence maximal)  irreducible family of generalized Nikulin surfaces whose very general element has N\'eron--Severi lattice $\left(\langle 8\rangle\oplus E_7(-2)\right)'$.
\end{proof}

\subsection{The case $\NS(F)=\langle 8\rangle\oplus E_7(-2)$ and $\NS(X)=\left(\langle 16\rangle\oplus E_8(-2)\right)'$, embedding $\widetilde{j}$.} By Proposition \ref{prop: proj models direct sum} and Remark \ref{rem: vice versa proj models direct sum} a generalized Nikulin surface $F$ is $\left(\langle 8\rangle\oplus E_7(-2)\right)$-polarized if and only if it admits a model as a complete intersection of three quadrics in $\mathbb{P}^5$ with 7 quadric sections which splits in the union of two genus 3 curves of degree 8. 

As in Section \ref{subsec:d=4,j1}, we attach to $F$ a $K3$ surface $Z$ which is the double cover of $\mathbb{P}^2$ branched on a sextic which is the discriminant locus of the net of quadrics defining $F$.
By the results in Section \ref{subsubsec: net and discriminant sextic}, one obtains a description of the relations between $F$ and $Z$. Let $Z$ be a $K3$ surface which is very general among the ones which admit a model as double cover of $\mathbb{P}^2$ branched on a sextic such that there exist 7 conics which split in the double cover, i.e.~ 7 conics such that their intersection with the branch locus has even multiplicity in each intersection point, which is equivalent to require that $\NS(Z)\simeq \langle \ell, \gamma_1,\ldots \gamma_7\rangle$ such that $\ell^2=2$, $\ell\gamma_i=2$, $\gamma_i^2=-2$, and $\gamma_i\gamma_j=0$ if $j\neq i$, $i,j=1,\ldots 6$ or $(i,j)=(1,7), (3,7), (5,7)$ and $\gamma_7\gamma_j=2$ if $j=2,4,6$. By Proposition \ref{prop: proj models direct sum} and Remark \ref{rem: vice versa proj models direct sum}, we obtain that $\NS(Z)\simeq \langle 2\rangle\oplus E_7(-2)$, i.e.~ that $Z$ is a generalized Nikulin surface of degree 2. In Section \ref{subsubsec: net and discriminant sextic} we will prove the following.
\begin{prop}\label{prop: 8+G and 2+G related by twisted moduli space}
 A very general generalized Nikulin surface $F$ with N\'eron--Severi group $\langle 8\rangle\oplus E_7(-2)$ is the moduli space of stable twisted sheaves on a generalized Nikulin surface $Z$ with $\NS(Z)\simeq \langle 2\rangle\oplus E_7(-2)$.
 The twisted Mukai vector is $v_B=(2,2B+\ell,1)$ and the Brauer class admits a $B$-lift with $B\ell=\frac{1}{2}$, $B^2=0$,
 $B\gamma_i=\frac{1}{2}$.
\end{prop}

\subsection{The complete intersections of three quadrics in $\mathbb{P}^5$ and the discriminant sextic of the net}\label{subsubsec: net and discriminant sextic}
We already recalled that to a $K3$ surface $S$ which is the complete intersection of three quadrics $Q_1,Q_2,Q_3\in\mathbb{P}^5$, such that  $\NS(S)\simeq \langle 8\rangle=\Z h$, one can associate a $K3$ surface $Z$, which is the double cover of $\mathbb{P}^2$ branched in the sextic $\det(x_0Q_1+x_1Q_2+x_3Q_3)=0$, i.e.~ on the discriminant sextic of the net. One can reconstruct the $K3$ surface $S$ using $Z$ if one fixes an even theta characteristic on the sextic curve; we consider a symmetric resolution of the theta characteristic using free sheaves on $\PP^2$.

This procedure can be interpreted as follows: $Z$ is a moduli space of stable sheaves on $S$ and vice versa, $S$ is a moduli space of stable twisted  sheaves on $Z$ (in this latter case one has to fix a $B$-field, which corresponds to choosing a theta characteristic on the sextic).
Let us explain this procedure: on $S$ we consider the Mukai vector $v=(2,h,2)$ in $H^*(S,\Z)$. Then $Z\simeq M_v(S)$ has N\'eron--Severi group isometric to $\langle 2\rangle\simeq (v^{\perp}\cap (H^0(S,\Z)\oplus \NS(S)\oplus H^4(S,\Z)))/v$.

Vice versa, if $\NS(Z)=\langle 2\rangle\simeq \Z\ell$, then we can fix a $B$-field such that $B^2=0$, $B\ell=1/2$. Up to an isometry of $\Lambda_{K3}$, one can assume that $\ell=u_1^{(1)}+u_2^{(1)}$ and that $B=\frac{1}{2}(u_1^{(1)}+u_1^{(2)})\in \frac{1}{2}H^2(Z,\Z)$ where $u_i^{(j)}$, $i=1,2$, $j=1,2,3$ form the basis of the $j$-th copy of $U$ in $\Lambda_{K3}$. The twisted N\'eron--Severi group is now generated by  $(0,\ell,0)$, $(2,2B,0)$ and $(0,0,1)$ and hence it is $\left[\begin{array}{ccc}2&1&0\\1&0&-2\\0&-2&0\end{array}\right]$. We fix the vector $v_B:=(1,1,1)$ in the twisted extended N\'eron--Severi group as Mukai vector, i.e.~ $v_B=(2,\ell+2B,1)$ in the Mukai lattice. The N\'eron--Severi group of $M_v(Z,B)$ is $v_B^{\perp}/v_B$ in the twisted extended N\'eron--Severi group. 
Since the orthogonal to the vector $(1,1,1)$ in $\left[\begin{array}{ccc}2&1&0\\1&0&-2\\0&-2&0\end{array}\right]$ is generated by $\langle (1,1,1), (0,-2,1)\rangle$ one obtains that $\NS(M_v(Z,B))=\langle 8\rangle$ and so one can identify $M_v(Z,B)$ with a $K3$ surface with a model as complete intersection of three quadrics in $\mathbb{P}^5$. Observe that the choice of $B$ and $v_B$ determines the transcendental lattice of $M_v(Z,B)$. Recall that if the Picard number of a K3 surface is less than 11, its N\'eron--Severi group determines uniquely its transcendental lattice. In the following we will consider surfaces with Picard number less than 9, we compute their N\'eron--Severi group and this determines uniquely also their transcendental lattice.

\subsubsection{K3 surfaces with Picard number 2}\label{13.6.1}
As done in \cite{GvG} one can specialize the previous construction to $K3$ surfaces with higher Picard number. So we consider: $S$ a K3 surface with a polarization $h$ of degree 8, $Z$ a K3 surface with a polarization $\ell$ of degree 2 and a B-field  on $Z$ such that $$B^2=0,\ \ B\ell=\frac{1}{2}.$$
In view of our applications we are interested in $K3$ surfaces $S$ with one of the following properties: a) there exists a hyperplane section which splits in two genus 1 curves, meeting in 4 points; b) there exists a quadric section which splits in two genus 3 curve, meeting in 12 points.

{\it {\bf Case a).}} One has $\NS(S)\simeq \left[\begin{array}{cc}8&4\\4&0\end{array}\right]$ and denotes $\langle h,e\rangle$ the corresponding generators of $\NS(S)$.
Considering the same Mukai vector as the one of the general case, i.e.~ $v=(2,h,2)$, one obtains that $Z=M_v(S)$ has the following N\'eron--Severi group $\NS(Z)=\left[\begin{array}{rr}2&0\\0&-2\end{array}\right],$ which corresponds to a $K3$ surface $Z$ which is the double cover of $\mathbb{P}^2$ branched on a sextic with one node. We denote $\langle \ell, r\rangle$ the basis of the N\'eron--Severi group of $Z$.

We want to reconstruct $S$ from $Z$ as a moduli space of stable twisted sheaves on $Z$. Since now we added the class $r$ to $\NS(Z)$, one has a priori two possibilities, either $Br=0$ or $Br=\frac{1}{2}$. \begin{enumerate}\item Assume $Br=0$. Then the twisted extended N\'eron--Severi group is $\left[\begin{array}{rrrr}2&0&1&0\\
0&-2&0&0\\1&0&0&-2\\0&0&-2&0\end{array}\right]$ and chosen the vector $(1,0,1,1)$ in it as Mukai vector, one obtains that the N\'eron--Severi group of the associated surface $S$ is $\left[\begin{array}{rr}-2&0\\0&8\end{array}\right]$.
\item Assume $Br=\frac{1}{2}$. Then the twisted extended N\'eron--Severi group is $\left[\begin{array}{rrrr}2&0&1&0\\
0&-2&1&0\\1&1&0&-2\\0&0&-2&0\end{array}\right]$ and chosen the vector $(1,0,1,1)$ in it as Mukai vector, one obtains that the N\'eron--Severi group of the associated surface $S$ is $\left[\begin{array}{rr}0&4\\4&8\end{array}\right]$.
\end{enumerate}
The previous computations imply that there are two different geometric reasons for which the discriminant sextic of a net of quadrics has a singular point: either the complete intersection of the quadrics in the net is singular, or it admits a hyperplane section which splits (in this latter case, one of the quadric of the net is singular along a line, but the complete intersection is smooth).
We are interested in the second case, so to reconstruct $Z$ by $S$ we must assume $Br=1/2$.

Observe that, once chosen the classes $\ell$ and $B$ as above, an admissible choice for $r$ is $r=u_2^{(2)}+e_1^{(1)}$ where we will denote $e_i^{(j)}$, $i=1,\ldots, 8$, $j=1,2$ a basis of $E_8(-1)\oplus E_8(-1)$.

{\it {\bf Case b)}.} One has $\NS(S)\simeq \left[\begin{array}{cc}8&8\\8&4\end{array}\right]$ and denotes $\langle h,g\rangle$ the corresponding generators of $\NS(S)$.
Considering the same Mukai vector as the one of the general case, i.e.~ $v=(2,h,2)$, one obtains that if $Z=M_v(S)$, $\NS(Z)=\left[\begin{array}{rr}2&2\\2&-2\end{array}\right],$ which corresponds to a $K3$ surface $Z$ which is the double cover of $\mathbb{P}^2$ branched on a sextic which admits an everywhere tangent conic.  We denote $\langle \ell, \gamma\rangle$ the basis of the N\'eron--Severi group of $Z$.

We want to reconstruct $S$ from $Z$ as a moduli space of stable twisted sheaves on $Z$. Since now we added the class $\gamma$ to $\NS(Z)$, one has a priori two possibilities, either $B\gamma=0$ or $B\gamma=\frac{1}{2}$. \begin{enumerate}\item Assume $B\gamma=0$. Then the twisted extended N\'eron--Severi group is $\left[\begin{array}{rrrr}2&2&1&0\\
2&-2&0&0\\1&0&0&-2\\0&0&-2&0\end{array}\right]$ and chosen the vector $(1,0,1,1)$ in it as Mukai vector, one obtains that the N\'eron--Severi group of the associated surface $S$ is $\left[\begin{array}{rr}-2&4\\4&8\end{array}\right]$.
\item Assume $B\gamma=\frac{1}{2}$. Then the twisted extended N\'eron--Severi group is $\left[\begin{array}{rrrr}2&2&1&0\\
2&-2&1&0\\1&1&0&-2\\0&0&-2&0\end{array}\right]$ and chosen the vector $(1,0,1,1)$ in it as Mukai vector, one obtains that the N\'eron--Severi group of the associated surface $S$ is $\left[\begin{array}{rr}8&8\\8&4\end{array}\right]$.
\end{enumerate}
The previous computations imply that there are two different geometric reasons for which the discriminant sextic of a net of quadrics has an everywhere tangent conic: either the complete intersection of the quadrics admits a hyperplane section which splits in two rational normal curves, or it admits a quadric section which splits in two genus 3 curves.
The geometric setting that we are considering is the latter, so to reconstruct $Z$ from $S$ we must assume $B\gamma=1/2$.
Chosen $\ell$ and $B$ as above, an admissible choice for $\gamma$ is $\gamma=u_1^{(1)}+u_2^{(1)}+e_1^{(1)}+e_1^{(2)}$.

\subsubsection{K3 surfaces with Picard rank 3}
We specialize again the previous surfaces to $K3$ surfaces with higher Picard number. In particular, we are  interested in $K3$ surfaces $S$ which are complete intersection of quadrics in $\mathbb{P}^5$ and admit both the previous properties, i.e.~ for which  there exists a hyperplane section which splits in two genus 1 curves, meeting in 4 points, and  there is a quadric section which splits in two genus 3 curves, meeting in 12 points. 

In this case one has $\NS(S)\simeq \left[\begin{array}{ccc}8&4&8\\4&0&3\\8&3&4\end{array}\right]$ and we denote $\langle h,e,g\rangle$ the corresponding generators of $\NS(S)$.
Let us observe that $g-e$ is the class of a rational curve whose intersection with $h$ is 4. In particular on the $K3$ surface $S$ there is a hyperplane section which splits in the union of two rational curves (this agrees with the 2 different descriptions we gave of $F$ when $\NS(F)=\left(\langle 8\rangle\oplus E_7(-2)\right)'$ in Propositions \ref{prop: proj models '} and \ref{prop: alternative description 8+G'}).

Considering the same Mukai vector as the one of the general case, i.e.~ $v=(2,h,2)$, one obtains that $Z=M_v(S)$ has the following N\'eron--Severi group $\NS(Z)=\left[\begin{array}{rrr}2&0&2\\0&-2&1\\2&1&-2\end{array}\right],$ which corresponds to a $K3$ surface $Z$ which is the double cover of $\mathbb{P}^2$ branched on a sextic with one node and admitting an everywhere tangent conic and passing through the node. We denote $\langle \ell, r,\gamma\rangle$ the basis of the N\'eron--Severi group of $Z$.

As above, one is interested in reconstructing $S$ from $Z$ in terms of moduli space of stable twisted sheaves on $Z$. Since now we added a class to $\NS(Z)$, one has to consider the intersection between the $B$-field and the new class. We already observed that in our geometric context $Br$ has to be $1/2$ (otherwise the associated $K3$ surface $S$ would be singular). One can freely chose $B\gamma$ to be equal to 0 or to $1/2$, indeed we observed that the $K3$ surface $S$ has both the properties associated to these choices. More explicitly, \begin{itemize}\item if one assumes that $Br=1/2$ and $B\gamma=0$, the N\'eron--Severi group of the $K3$ obtained by choosing $(1,0,0,1,1)$ as Mukai vector in the twisted N\'eron--Severi group is $\left[\begin{array}{ccc}8&4&4\\4&0&3\\4&3&-2\end{array}\right]$, 
i.e.~ the $K3$ surface is a complete intersection of three quadrics in $\mathbb{P}^5$ for which there is a hyperplane section which splits in two genus 1 curves of degree 4 and another hyperplane section which splits in the union of two rational curves of degree 4.
\item if one assumes that $Br=B\gamma=1/2$ the N\'eron--Severi group of the $K3$ surface obtained choosing $(1,0,0,1,1)$ as Mukai vector in the twisted N\'eron--Severi group is 
$\left[\begin{array}{ccc}8&4&8\\4&0&3\\8&3&4\end{array}\right]$, 
i.e.~ the $K3$ surface is a complete intersection of three quadrics in $\mathbb{P}^5$ for which there is a hyperplane section which splits in two genus 1 curves of degree 4 and a quadric section which splits in the union of two genus 3 curves  of degree 8.\end{itemize}
The two N\'eron--Severi groups (the ones obtained for $B\gamma=0$ and for $B\gamma=\frac{1}{2}$) are isometric, hence there is no a difference among these two choices.
\subsubsection{Higher Picard rank}
One can iterate the previous process many times, and in particular, for the case of our interest, one can consider the $K3$ surface $S$, which is the complete intersection of three quadrics in $\mathbb{P}^5$ such that there are six hyperplanes sections each of  which split in the union of two genus 1 curves and an hyperplane section which splits in the union of two rational curves (or equivalently a quadric section which splits in the union of two genus 3 curve of degree 8). 
The moduli space of stable sheaves  on $S$ with Mukai vector $(2,h,2)$ is the $K3$ surface $Z$ which is the double cover of $\mathbb{P}^2$ branched on a sextic which has 6 singular points and a there exists a conic through all these singular points. We will denote $\langle \ell,r_i,\gamma\rangle$, $i=1,\ldots, 6$ a base of the N\'eron--Severi group of $Z$ such that $\ell^2=2$, $\ell r_i=0$, $\ell \gamma=2$, $r_i^2=\gamma^2=-2$, $r_ir_j=0$ and $r_i\gamma=1$. 

Vice versa, one can obtain $S$ as moduli space of stable twisted sheaves on $Z$ by considering a $B$-field such that $B\ell=\frac{1}{2}$, $B^2=0$, $Br_i=1/2$ for all $i=1,\ldots 6$ and $B\gamma=1/2$ (or, equivalently $B\gamma=0$). Observe that an admissible choice for $B$ and the basis $\langle \ell,r_i,\gamma\rangle$ is $\ell=u_1^{(1)}+u_2^{(1)}$, $B=\frac{1}{2}(u_1^{(1)}+u_1^{(2)})$ (as above), $\gamma=u_1^{(1)}+u_2^{(1)}+u_1^{(2)}+e_1^{(1)}+e_1^{(2)},\ r_i=u_2^{(2)}+e_{2i+1}^{(1)},\ i=1,2,3,$ and $r_j=u_2^{(2)}+e_{2j-5}^{(2)},\ j=4,5,6.$  This proves Proposition \ref{propB}.

Similarly one proves Proposition \ref{prop: 8+G and 2+G related by twisted moduli space} by iterating the process to complete intersection of three quadrics which admits 7 quadric sections each of which splits in two genus 3 curves. More precisely, one considers $Z$ whose N\'eron--Severi group is generated by $\langle \ell,\gamma_i\rangle$, $i=1,\ldots, 7$ such that 
$\ell^2=2$, $\ell\gamma_i=2$, $\gamma_i^2=-2$, and $\gamma_i\gamma_j=0$ if $j\neq i$, $i,j=1,\ldots 6$ or $(i,j)=(1,7), (3,7), (5,7)$ and $\gamma_7\gamma_j=2$ if $j=2,4,6$ 
 and one constructs $S$ as the moduli space of stable twisted sheaves $M_v(Z,B)$ with $v$ as above and $B$ a $B$-field with the properties $B^2=0$, $B\ell=B\gamma=\frac{1}{2}$. An admissible choice for the involved vectors is  $\ell=u_1^{(1)}+u_2^{(1)}$, $B=\frac{1}{2}(u_1^{(1)}+u_1^{(2)})$ (as above), $$\begin{array}{ll}\gamma_1=u_1^{(1)}+u_2^{(1)}+u_1^{(3)}-u_2^{(3)}+u_1^{(2)}+e_1^{(1)},&
 \gamma_2=u_1^{(1)}+u_2^{(1)}+u_1^{(3)}-u_2^{(3)}+e_3^{(1)}\\
 \gamma_3=u_1^{(1)}+u_2^{(1)}+u_1^{(3)}-u_2^{(3)}+u_1^{(2)}+e_5^{(1)},&
 \gamma_4=u_1^{(1)}+u_2^{(1)}+u_1^{(3)}-u_2^{(3)}+e_7^{(1)},\\
 \gamma_5=u_1^{(1)}+u_2^{(1)}+u_1^{(3)}-u_2^{(3)}+u_1^{(2)}+e_1^{(2)},&
\gamma_6=u_1^{(1)}+u_2^{(1)}+u_1^{(3)}-u_2^{(3)}+e_3^{(2)},\\
\gamma_7=u_1^{(1)}+u_2^{(1)}+u_1^{(3)}-u_2^{(3)}+2u_2^{(2)}+e_5^{(2)}.\end{array}$$ 
\subsection{The case $\NS(F)=(\langle 12\rangle\oplus G)^\bullet$, $\NS(X)=\langle 6\rangle\oplus E_8(-2)$, embedding $j_3$.} Observe that in Proposition \ref{prop: proj models bullet} the geometric description of the surfaces with N\'eron--Severi group $(\langle 2d\rangle\oplus G)^\bullet$ is provided for every admissible $d\neq 6$. Here we fill this gap.
By Proposition \ref{prop:relation NS(F), NS(X), NS(W)}, if $\NS(F)=(\langle 12\rangle\oplus G)^\bullet$, then $F$ is the fixed $K3$ surface of a symplectic involution on $X$, where $\NS(X)=\langle 6\rangle\oplus E_8(-2)$, embedded in $H^2(X,\Z)$ via $j_3$.  The hyper-K\"ahler manifold $X$ is the  Fano variety of lines on a cubic 4-fold, the involution is induced by an involution on the cubic and the explicit description is known. In particular the fixed surface is considered in \cite[Proposition 3.15]{CGKK} where it is shown that $F$ is a complete intersection of two divisors of bidegree type $(2,1)$ and $(0,3)$ in $\mathbb{P}^1\times \mathbb{P}^3$, i.e.~ $F$ is a divisor of type $(2,1)$ in $\mathbb{P}^1\times V_3$ where $V_3$ is a cubic hypersurface in $\mathbb{P}^3$. 

\section{Related topics}\label{section: related problems}

\subsection{The residual involution}\label{subset:the residuals involution}
It follows from \cite[Section 2]{CGKK} that a manifold of $K3^{[2]}$-type $X$ admitting a symplectic involution $i$ is isomorphic to a moduli space of stable twisted sheaves on some $K3$ surface $S$ (so stable sheaves on a twisted $K3$ surface $(S,\alpha)$).
From \cite[Proposition 1.4]{BO}, there exists an action $\iota$ on $D(S,\alpha)$ that induces the above action $i$ on the moduli space $X. $ 
Denote by $F$ the fixed locus of such a symplectic involution $i$.
 
By \cite[Theorem 1.1]{BO}  we expect that the equivariant category $D(S,\alpha)_{i}$ is equivalent to a derived category of twisted sheaves on the fixed surface $(F,\beta)$ for some Brauer class $\beta\in H^2(F,\mathcal{O}_F^\star)_{tors}$ (cf.~\cite[Section 1.3]{BO}). 

By a result of Elagin, \cite[Theorem 1.3]{El}, there exists a residual involution (see \cite{KP} for an introduction to residual involutions) on the derived category $D(F,\beta)$ that induces an involution $j$
on its Mukai lattice $$\tilde{H}(F,\ZZ)=H^0(F,\ZZ)\oplus H^2(F,\ZZ)\oplus H^4(F,\ZZ).$$

This involution was explicitly described in \cite[Section 7.2]{BO} in the case when $X\simeq S^{[2]}$, and hence such that the fixed surface is a Nikulin surface. It is natural to ask for the same description in the more general context in which $X$ is not necessarily $S^{[2]}$ and therefore the fixed surface $F$ is a generalized Nikulin surface (and not a Nikulin surface).
In particular, one can be interested in describing possible Brauer classes $\beta$ and 
residual involutions on $D(F,\beta)$.

The aim of this section is to present some results in this direction.

We recall that when the fixed $K3$ surface $F$ is a Nikulin surface, then the Nikulin lattice $N$ is contained in $\NS(F)$  and
the dual action $j$ on the Mukai lattice $\tilde{H}(F,\ZZ)$  given in \cite[Section 7.2]{BO} and it is 
$(0,E_i,0)\to (0,-E_i,1)$ and $(1,0,0)\to (1,\hat{E},-2)$ where $E_i$ for $i=1,\ldots, 8$ are the exceptional curves in the Nikulin lattice and $\hat{E}:=(E_1+\dots +E_8)/2$. 
\begin{prop}\label{61}
 Let $j$ be the residual involution defined above on a Nikulin surface $F$. The anti-invariant part of $j$ on $\tilde{H}(F,\ZZ)$ is isometric to $E_8(-2)$ and it is contained in  $H^0(F)\oplus \NS(F)\oplus H^4(F)$.
\end{prop}
\begin{proof}
The lattice $\Lambda_{K3}$ is the unique even unimodular lattice with signature $(3,19)$. We fix a specific construction of it (which is not the standard one, but still it is useful for our purpose).

Let us consider the lattice $U^{\oplus 3}\oplus N$, where $N$ is the Nikulin lattice. We denote $u_{i}^{(j)}$, $i=1,2$ a basis of the $j$-th copy of $U$, so that $u_1^{(j)}u_2^{(j)}=1$ and $u_i^{(j)}u_h^{(k)}=0$ if $j\neq k$. Recall that 
a $\Z$-basis of $N$ is given by $\{E_1,E_2,E_3,E_4,E_5,E_6, E_7,\hat{E}\}$.

Consider another copy of the lattice $N$, generated by $\epsilon_i$, $i=1,\ldots,8$ such that $\epsilon_i^2=-2$ and $\epsilon_i\epsilon_j=0$ and by the class $\hat{\epsilon}=-\sum_i\epsilon_i/2$.

The lattice $U^{\oplus 3}\oplus N\oplus N$ with basis $\{u_i^{(j)}, E_h,\hat{E},\epsilon_k,\hat{\epsilon}\}$ is an even lattice with signature $(3,19)$, whose discriminant group is $(\Z/2\Z)^{\oplus 12}$.
To obtain an overlattice which is unimodular it suffices to properly add 6 classes which are gluing vectors: let us consider the following ones
$$(-E_i-E_{i+1}+\epsilon_i+\epsilon_{i+1})/2,\ i=1,2,3,4,5,6.$$
A $\Z$-basis of the overlattice obtained is
\begin{equation}\label{eq: basis LambdaK3}u_i^{(j)}, E_1,\ldots , E_7,\hat{E}, \frac{-E_1-E_2+\epsilon_1+\epsilon_2}{2}, \frac{-E_2-E_3+\epsilon_2+\epsilon_3}{2}, \frac{-E_3-E_4+\epsilon_3+\epsilon_4}{2},\ \end{equation} \begin{equation*}\frac{-E_4-E_5+\epsilon_4+\epsilon_5}{2},\frac{-E_5-E_6+\epsilon_5+\epsilon_6}{2},\frac{-E_6-E_7+\epsilon_6+\epsilon_7}{2}, \epsilon_7,\hat{\epsilon}.\end{equation*}

Computing the intersection form on this basis, one obtains an even unimodular lattice of signature $(3,19)$, which is necessarily isometric to $\Lambda_{K3}$. So from now we fix the set of vectors \eqref{eq: basis LambdaK3} as basis of $\Lambda_{K3}$.

The Mukai lattice is isometric to $\Lambda_{K3}\oplus U(-1)$ where we denote $w_1$ and $w_2$ the vectors generating $U(-1)$ (they correspond to the vector $(1,0,0)$ and $(0,0,1)$ in the standard notation of the Mukai vectors).

We are now interested in an isometry $j$ of the Mukai lattice whose action is the following:
$$j(E_i)=-E_i+w_2,\ j(w_2)=w_2\,\ j(w_1)=w_1+\hat{E}-2w_2,$$
and which is the identity on the orthogonal complement to $\langle w_1,w_2,E_i,i=1,\ldots, 8\rangle$.
This allows one to deduce the action on the $\Z$-basis of the Mukai lattice given by $w_1$, $w_2$ and the vectors in \eqref{eq: basis LambdaK3}: $$j(w_1)=w_1+\hat{E}-2w_2,\ j(w_2)=w_2,\ j(\epsilon_i)=\epsilon_i,\ j(u_i^{j})=u_i^{j},\ j(\hat{E})=-\hat{E}+4w_2,\ j(\hat{\epsilon})=\hat{\epsilon},$$ $$j\left((-E_i-E_{i+1}+\epsilon_i+\epsilon_{i+1})/2\right)=(E_i+E_{i+1}+\epsilon_i+\epsilon_{i+1})/2-w_2=\left((-E_i-E_{i+1}+\epsilon_i+\epsilon_{i+1})/2\right)+E_i+E_{i+1}-w_2.$$

The map $j$ is an order 2-isometry, which has 8 eigenvalues $-1$ and 16 eigenvalues $+1$. 
The anti-invariant lattice is generated by $E_i-E_{i+1}$, $i=1,\ldots, 6$, $w_2-E_1-E_7$, $w_2-\hat{E}+2E_7$. Computing the intersection form on this basis one finds a lattice with signature $(0,8)$, whose intersection form can be divided by 2 obtaining a unimodular even lattice, so the anti-invariant lattice is $E_8(-2)$. Since $j$ is the identity on the orthogonal complement to $U\oplus N\simeq \langle w_1,w_2,E_i,i=1,\ldots,7, \hat{E}\rangle$, the anti--invariant part is clearly contained in $U\oplus N\subset H^0(F)\oplus \NS(F)\oplus H^4(F)$. 
\end{proof}
\begin{rem}\label{rem: different U+N decompostions}
By considering more closely the action of the map $j$ on the sublattice $U\oplus N=\langle w_1,-w_2,E_i,i=1,\ldots, 7, \hat{E}\rangle$, one obtains two different decompositions of the same lattice as direct sum of the hyperbolic plane and the Nikulin lattice, which will be useful in the following. We describe these two decompositions.
The orthogonal to the anti-invariant part $E_8(-2)$ in $U\oplus N=\langle w_1,w_2,E_i,\hat{E}\rangle$ is generated by $\{-w_2,2w_1-w_2+\hat{E}\}$ and isometric to $U(2)$. 
The vector $v=2w_1-2w_2+\hat{E}$ is the unique vector with square $v^2=4$ in $U(2)$, up to a sign. 
Observe that $$w_1=\left(v+\varepsilon\right)/2\in U\oplus N$$ where $$\varepsilon=\sum_{i=1}^6(E_i-E_{i+1})+(w_2-E_1-E_7)+(w_2-\hat{E}+2E_7)=2w_2-\hat{E}\in E_8(-2).$$ So the minimal primitive sublattice of $U\oplus N$ which contains both $E_8(-2)$ and $v$ is an overlattice of index 2 of $\langle 4\rangle\oplus E_8(-2)$ and computing its discriminant form one shows that it is isometric to $U\oplus E_7(-2)$. Observe that the first summand in this orthogonal decomposition is not the same copy of $U$ that we have been considering before (i.e.~  $\langle w_1,w_2\rangle$). To stress this fact we write this lattice as $U'\oplus E_7(-2)$ (where $U'$ is the hyperbolic plane with certain generators which are not $w_1,w_2$). Observe that the class $z=-(-w_2)+2w_1-w_2+\hat{E}=2w_1+\hat{E}$ is the unique vector with square $z^2=-4$ in $U(2)$, up to a sign and that there is an overlattice of index 4 of $\langle z\rangle\oplus E_7(-2)$ which is embedded in $U\oplus N$. Denoted $(\langle z\rangle\oplus E_7(-2))'$ such an overlattice, one first observes that it is isometric to a Nikulin lattice (denoted $N'$ to stress that it is not the copy of the Nikulin lattice generated by the $E_i$'s) and then one obtains $U\oplus N\simeq U'\oplus (E_7(-2)\oplus\langle z\rangle)'\simeq U'\oplus N'$.
\end{rem}

It is natural to study the residual involution for a general generalized Nikulin surface and to do that one needs to study the action on families of manifolds of $K3^{[2]}$-type with involution and this can be done as in \cite[Section 8.4]{BP}.

 As before let $F$ be the Nikulin surface that is the fixed locus of an induced involution on $S^{[2]}$.
From \cite[Section 7.3]{BO} the invariant category is equivalent to $D(F)$. 
One can deform $D(F)$ together with the residual involution. First, we consider a deformation of $(S^{[2]}, \iota^{[2]})$ over a small disc $B$ (from Table \ref{eq: X as twisted moduli space of sheaf} we can choose it by considering a family of moduli spaces of twisted sheaves $M_v(S_t,\alpha_t)$ with involution $j_t$). 
We obtain the family $\mathcal{F}\to B$ of twisted $K3$ surfaces  such that the residual involution to $j_t$ acts on each fiber $D(F_t,\beta_t)$ separately.
Now from \cite[Proposition 3.9]{BP} generalized for twisted sheaves one could prove that there exists an etale cover $C$ of $B$ (also of the moduli space of the corresponding lattice-polarized $K3$ surfaces as considered in Corollary \ref{prop: Ns e T of projective F}) such that the involution lifts to an involution of the total family $\mathcal{F}_C\to C$ acting fiberwise, i.e.~on $D(\mathcal{F}_C, \beta)$ where $\beta$ is a Brauer class on the total family. 
Then one can use \cite[Section 4.3]{HMS} (generalized for twisted $K3$ surfaces) and see that for a $K3$ surface $F_t$ the category $D(F_t,\beta_t)$ deforms together with the involution when the anti-invariant part in the cohomology $\tilde{H}(F_t,\ZZ)$ remains algebraic.
By dimension count we obtain by such deformations the action on a very general $D(F_t,\beta_t)$ (meaning we obtain very general examples with $\rk \NS(F_t)=8$). Thus the anti-invariant part of the action on $\tilde{H}(F_t,\ZZ)$ is $E_8(-2)$ and for very general $t$ the generalized Picard lattice $\tilde{H}^{1,1}(F_t,\ZZ)$ contains the anti-invariant part.
It is natural to expect the following.

    \begin{prob}
Let $X$ be a manifold of $K3^{[2]}$-type with a symplectic involution and fixed $K3$ surface $F$. Let $j$ be a residual involution on $D(F)$ inducing an action
 $j$ on $\tilde{H}(F,\ZZ)$. Is it true that the anti-invariant sublattice of $j$ is isometric to $E_8(-2)$ and that it is contained in $\tilde{H}^{1,1}(F,\ZZ)$?
    \end{prob}

\begin{rem}
Let us describe the Hodge structures of the categories $D(F_t,\beta_t)$ in $\tilde{H}(F,\ZZ)\otimes \mathbb C$ explicitly when $F$ is a Nikulin surface (we mean after choosing a trivialisation of the local system). Hence we have $N\subset \NS(F)$ generated by the $E_i$'s and a copy of $U=\langle w_1,-w_2\rangle$ which spans $H^0(F,\Z)\oplus H^4(F,\Z)$ in $\tilde{H}(F,\Z)$ as in Remark \ref{rem: different U+N decompostions}. Moreover, we have the isometry $j$ acting on $U\oplus N$ as in proof of Proposition \ref{61}.
In Remark \ref{rem: different U+N decompostions}, we show 
such that we have two decompositions $U\oplus N=U'\oplus N'$ in such a way that the orthogonal complement to $U'$ in $\tilde{H}(F,\ZZ)$, that is necessarily isometric to $\Lambda_{K3}$,  
contains $E_7(-2)$ 
and the $j$-invariant $(-4)$-vector $z$ (perpendicular to $v$ in the invariant $U(2)\subset U\oplus N$).  The very general periods $\omega_t$ of $D(F_t,\beta_t)$ are contained in
$\Lambda_{K3}\otimes \mathbb C \subset \tilde{H}(F,\ZZ)\otimes \mathbb C $. Those periods which are invariant with respect to $j$ and not necessarily perpendicular to the $(-4)$-vector $z$ correspond by the Torelli theorem  to $K3$ surfaces $F_t$ such that $\NS(F_t)$ contains primitively $E_7(-2)$.
\end{rem}

\subsubsection{The residual involution on the double EPW quartic}\label{14.1}
Let us consider a generalized Nikulin surface $F$
with N\'eron--Severi lattice $\langle2\rangle\oplus E_7(-2)$ and let $h$ be the generator of  $\langle 2\rangle$. 
 By Table \ref{eq: X as twisted moduli space of sheaf},  $F$ is isomorphic to a component of the fixed locus of a symplectic involution $g$ on $X=M_{(2,\sum n_i,-4)}(S_2)$, where $S_2$ is a $K3$ surface constructed as the resolution of a $7$-nodal quartic and the exceptional $(-2)$-classes are denoted by $n_i$ (the N\'eron--Severi lattice is $\langle4\rangle\oplus \langle-2\rangle^{\oplus 7}$). {Recall that $\NS(X)\simeq (\langle 4\rangle\oplus E_8(-2))'$.} 
 From the isomorphism $D(S_2)_g\simeq D(F,\alpha)$ (for some Brauer class $\alpha$ \cite{BO}),  by duality (see \cite{El}) we obtain  $D(S_2)\simeq D(F,\alpha)_{g^{\vee}}$ and the induced map $\varphi\colon \tilde{H}(S_2,\ZZ)\to \tilde{H}(F,\ZZ)$ as in \cite[Section 7.3]{BO}. We expect that $\alpha$ is trivial so $S_2$ could be a fixed locus of an involution on some moduli space of stable sheaves on $F$.

\begin{prob}\label{prob14.0}
    Is $S_2$ a fixed surface  of some involution on some moduli space of stable sheaves on $F$?
    Describe the map $\varphi$ and find the square $0$ vector $v_1$ mapping to $(0,h,0)$.
    \end{prob}
The previous problem is related to a construction we performed in \cite{CGKK}.
We consider on $F$ the Mukai vector $v=(0,h,0)$, then the moduli space $M_v(F)$ has Picard lattice $U\oplus E_7(-2)=(\langle 4\rangle\oplus E_8(-2))'$ and observe that the stability condition can be chosen using \cite[Theorem 1.3]{BM}.
Hence, we obtain the symmetric double EPW quartic $M_v(F)$ considered in 
\cite[Section 5]{CGKK}.
As observed in \cite[Section 5.1]{CGKK}, $M_v(F)$ admits a symplectic involution. We expect  this involution to be related to the residual involution obtained by the construction above of $F$ as the fixed locus.
 
Finally, recall also that $M_v(F)$ is isomorphic to the base of a $\mathbb{P}^1$ fibration on the Hilbert scheme of $(1,1)$-conics contained in a symmetric Verra fourfold $V$ (see \cite[Section 5A]{CGKK}). In this way it
is also the moduli space of stable twisted sheaves on a $K3$ surface $S$ with polarization $H$ of degree $2$ that is the double cover of $\PP^2$ branched along the discriminant sextic of the quadric fibration on $V$.
So $M_{(0,h,0)}(F)$ is isomorphic to $M_{(0,H,0)}(S,B)$ where $B$ is a $B$ lift of a $2$-torsion Brauer class such that  $B^2=0$ and $B.H=0$.
 It is natural to ask the following.
\begin{prob}\label{prob14.1}
   Is $F$ isomorphic to the  surface $S$? 
    \end{prob}
    Note that a priory $F$ could be associated to the discriminant locus of another Verra fourfold associated to $M_v(F)$. The problem is thus to understand the discriminant locus of the symmetric $V$ (cf.~Section \ref{Verra}).

\subsubsection{The residual involution and {six-fold of $K3^{[3]}$-type}}
$ $
Observe that if $E_8(-2)$ is primitively embedded in a lattice $\Gamma$, then there exists an isometry of $\Gamma$ which acts as $-\id$ on the embedded copy of $E_8(-2)$ and as the identity on its orthogonal complement. This is because the action of $-\id$ on $E_8(-2)$ induces the identity on the discriminant group of $E_8(-2)$, and then can be glued with the identity on the orthogonal complement of $E_8(-2)$ in $\Gamma$.

Let us now consider $F$ a very general generalized Nikulin surface with $\NS(F)=E_7(-2)$ and observe that $H^0(F,\Z)\oplus H^4(F,\Z)\oplus \NS(F)\subset H^*(F,\Z)$ is a lattice isometric to $U\oplus E_7(-2)$. Recall that 
$U\oplus E_7(-2) \simeq (E_8(-2)\oplus \langle4\rangle)'$, see \ref{lemma: properties of G}. Therefore there is an isometry $i$ on $U\oplus E_7(-2)\simeq (E_8(-2)\oplus \langle4\rangle)'$ which acts as $-\id$ on $E_8(-2)$ and as the identity on the vector $v$ which generates the sublattice $\langle 4\rangle$ orthogonal to $E_8(-2)$.
This extends to an isometry, still denoted $i$, which  acts on $$H^*(F,\Z)\simeq U\oplus H^2(F,\Z)$$ as $-\id$ on $E_8(-2)\hookrightarrow U\oplus \NS(F)\hookrightarrow H^*(F,\Z)$ and as the identity on its orthogonal complement. 

Then we choose $v$ as Mukai vector and a $v$-generic stability condition such that $M_v(F)$ is a six dimensional smooth hyperk\"ahler manifold of $K3^{[3]}$-type (it exists by \cite{BM} as discussed before). 

The construction above implies that $\NS(M_v(F))\simeq E_8(-2)$, which is the orthogonal to $v$ in $U\oplus \NS(F)$.
The involution $i$ fixes the vector $v$, therefore it defines an isometry on $H^2(M_v(F),\Z)$, which acts as $-\id$ on $\NS(M_v(F))\simeq E_8(-2)$ and as the identity on the transcendental lattice of $M_v(F)$.
In particular, $i$ is an Hodge isometry of the second cohomology group of $M_v(F)$ and hence it is related with the presence of a birational involution $b$ on $M_v(F)$. The birational involution $b$ induces an isometry on $H^2(M_v(F),\Z)$ which is the composition of $i$ and certain monodromy operator. Observe that a priori $b$ can be biregular or not, but one can show that there exists a class $W\in H^2(M_v(F),\Z)$ which has divisibility 2 and self intersection $-4$ and is contained in $\NS(X)\simeq E_8(-2)$. So $i(W)=-W$, which implies that $i$ is not an isometry preserving a K\"ahler class. 

\begin{prob}
    Does the involution $j$ induce a birational involution
    on the sixfold $M_v(F)$? 
\end{prob}

 We can see that $E_8(-2)$ contains $(-4)$-classes of divisibility $2$ so classes of exceptional divisors of  contractions.
 It follows that the involution does not preserve the movable cone and we cannot use the Torelli theorem \cite{BL,MenetTorelli} since all invariant classes are perpendicular to the $(-4)$-class $W$ so cannot be K\"ahler classes.

Observe that the existence of a good Mukai vector $v$, preserved by the involution $i$ is very much related with the choice considering $6$-dimensional manifolds. Indeed, let us consider a $K3$ surface $F$ with $\NS(F)\simeq E_7(-2)$ and so $\T_F\simeq U(2)^{\oplus 3}\oplus E_8(-1)\oplus \langle -4\rangle$. Let us consider a Mukai vector $v\in H^*(F,\Z)$ such that $v^2=2d$ and recall that $\T_F\simeq \T_{M_v(F)}$. If we require that $H^2(M_v(F),\Z)$ admits a Hodge isometry which acts as $-\id$ on a copy of $E_8(-2)$ and as the identity on the orthogonal complement, we are requiring that $\NS(M_v(F))=E_8(-2)$. So $H^2(M_v(F),\Z)\simeq U^{\oplus 3}\oplus E_8(-1)\oplus E_8(-1)\oplus \langle -2d\rangle$ has to be an overlattice of finite index of $$\NS(M_v(F))\oplus \T_{M_v(F)}\simeq E_8(-2)\oplus U(2)^{\oplus 3}\oplus E_8(-1)\oplus \langle -4\rangle.$$ This forces $d=2$, so that $v^2=4$.

\subsection{N-generalized Nikulin surfaces}\label{subset: N generalized}
Until now we considered manifolds of $K3^{[2]}$-type with a symplectic involution and we considered the fixed surface. A natural generalization of this problem is to consider a manifold of  $K3^{[n]}$-type for $n>2$ with a symplectic involution and consider its fixed locus. We mainly consider the case $n=3$.

Let $X_n$ be a  manifold of $K3^{[n]}$-type admitting a symplectic involution $\iota_X$. In Theorem \ref{Thm1.2} we proved that if $n=2$, the transcendental lattice of the terminalization $Y$ of $X_2/\iota_X$ is Hodge isometric to the transcendental lattice of the $K3$ surface in the fixed locus.  Observe that if $X_2$ is general among the manifolds of $K3^{[2]}$-type admitting a symplectic involution $\iota_X$, then the transcendental lattice $\T_Y$ coincides with $\pi_*(H^2(X_2,\Z))\simeq (\pi_*(\T_{X_2}))_{\rm{sat}}$, where $\pi:X_2\ra X_2/\iota_X$ is the quotient map, $\pi_*$ is the topological push--forward which coincides with the ones on algebraic cycles when both are defined and $(\pi_*(\T_{X_2}))_{\rm{sat}}$ is the saturation of $(\pi_*(\T_{X_2}))$ in $\pi_*(H^2(X_2,\Z))$. So we can rephrase Theorem \ref{Thm1.2} as follows: let $F$ be a very general generalized Nikulin surface, then the transcendental lattice $\T_F$ of $F$ is Hodge isometric to $\pi_*(H^2(X_2,\Z))\simeq \left(\pi_*(\T_{X_2})\right)_{\rm{sat}}$.\\

If $n\geq 3$, then $X_n/\iota_X$ is already terminal and so $Y$ coincides with $X_n/\iota_X$. In particular either $H^2(Y,\Z)$ is $\pi_*(H^2(X_n,\Z))$ or it is an overlattice of finite index of it. 
\begin{lemma}\label{lem: pi* higher dimension} Let $X_n$ be a  manifold of $K3^{[n]}$-type admitting a symplectic involution $\iota_X$ and $\pi:X_n\ra X_n/\iota_X$ be the quotient map. Then $$\pi_*(H^2(X_n,\Z))\simeq U(2)^{\oplus 3}\oplus E_8(-1)\oplus \langle -4(n-1)\rangle.$$
\end{lemma}
\proof The symplectic involution on  manifolds of $K3^{[n]}$-type are always induced by natural symplectic involution, i.e.~ there exists a specialization of $(X_n,\iota_X)$ to $(S^{[n]},\iota_S^{[n]})$ where $S$ is a $K3$ surface and $\iota_S$ is a symplectic involution on it, see \cite{KMO}. Therefore, the cohomological action of $\iota_X$ is the same as the one of $\iota_S^{[n]}$: in particular, it fixes the exceptional divisor of the Hilbert--Chow morphism and acts on $U^{\oplus 3}\oplus E_8(-1)^{\oplus 2}$ switching the two copies of $E_8(-1)$ and as the identity on $U^{\oplus 3}$. Recalling that $$H^2(X_n,\Z)=U^{\oplus 3}\oplus E_8(-1)^{\oplus 2}\oplus \langle -2(n-1)\rangle$$ where the last summand is generated by  half the exceptional divisor, by computations analogous to the ones in \cite{Mor}, one obtains $$\pi_*(H^2(X_n,\Z))\simeq U(2)^{\oplus 3}\oplus E_8(-1)\oplus \langle -4(n-1)\rangle,$$ since $\pi_*$ multiply by 2 the form of the invariant sublattice and identifies the two copies of $E_8(-1)$ to one.\endproof

\subsubsection{The case $n=3$}
Let us now concentrate one the case $n=3$. The fixed locus of $\iota_X$ consists of $64$ isolated points and 8 $K3$ surfaces, $Z_i$, $i=1,\ldots, 8$, see \cite{KMO}.
With ideas similar to the ones applied to study the generalized Nikulin surfaces, we now identify the surfaces which appear in the fixed locus by their N\'eron--Severi groups and transcendental lattices. First we deform the pair $(X_3,\iota_X)$ of a  manifold of $K3^{[3]}$-type with a symplectic involution to $(S^{[3]},\iota_S^{[3]})$, where $S$ is a $K3$ surface with a symplectic involution $\iota_S$. Then the surfaces in the fixed locus of $\iota_{S}^{[3]}$ are isomorphic to the Nikulin surface $\widetilde{S/\iota_S}$. Let us denote $E_1,\dots, E_8$ the curves spanning the Nikulin lattice in $\widetilde{S/\iota_S}$.

Let $S^{[3]}\to S^{(3)}$ be the Hilbert--Chow contraction with exceptional divisor $\Delta$. It is known by \cite{G} that the base $B$ of the contraction is non-normal along a surface $D$ isomorphic to $S$ and the normalization is bijective and is isomorphic to $S\times S $ (with the conductor supported on the diagonal).

The fixed surface $Z_i\subset S^{[3]}$ can be seen as the set of points $P_i+Q+\iota_S(Q)$ such that $P_i,Q\in S $ and $P_i$ is a point on $S$ fixed by $\iota_S$. The image of $P_i$ in $S/\iota_S$ is singular and it is resolved by the exceptional divisor $E_i$ in $\widetilde{S/\iota_S}$. 
One sees that the exceptional curve $E_i\subset Z_i$ is distinguished and proves the following.
    
\begin{prop}\label{prop: fixed locus 6 dimnesion}
   The intersection  $\Delta.Z_i$ in $S^{[3]}$ is $\sum_{j=1}^8E_j+2E_i$ on $Z_i$.
   For the general deformation $(X,Z_i')$ of $(S^{[3]},Z_i)$ obtained by deforming the symplectic involution, the N\'eron--Severi lattice of
   $Z_i'$ is the orthogonal to $Z_i.\Delta$ in the N\'eron--Severi lattice of $Z_i$ (i.e.~ in the Nikulin lattice).

The transcendental lattice of the surfaces $Z_i'$ is $U(2)^{\oplus 3}\oplus E_8(-1)\oplus \langle -8\rangle$.
\end{prop}

\begin{proof}
The idea is to show that $\Delta$ intersects $Z_i$ in $E_i$ with multiplicity $3$ and $E_j$ with multiplicity 1, for $j\neq i$.
The contraction $\Delta \to B$ is a topologically locally trivial fibration over $B-D$ with general fibers $\PP^1$ and the fibers over $D$ are isomorphic to the cone over a twisted cubic.
The local situation around $Z_i\cap \Delta$ is worked out in \cite[page 13 diagram 8]{F}, where it is proved that 
the multiplicity of the intersection $\Delta .Z_i$ is $3$.

As in Proposition \ref{cor: very general F has NS=G}, to compute the N\'eron--Severi group of $Z_i'$, it suffices to compute the orthogonal complement to $\sum_{j=1}^8E_j+2E_i$ inside the Nikulin lattice $N$. It is a rank 7 negative definite lattice of discriminant $(\Z/2\Z)^{\oplus 6}\oplus (\Z/8\Z)$. The transcendental lattice is an overlattice of finite index (possibly 1) of the transcendental lattice of a very general Nikulin surface $K3$ surface, plus the class $\sum_{j=1}^8E_j+2E_i$ (since it is no longer algebraic). Observe that also the class $(\sum_{j=1}^8E_j+2E_i)/2$ is contained in the N\'eron--Severi group (and hence in the second cohomology group) of a Nikulin surface. Therefore $\T_{Z_i'}$ is an overlattice of finite index of $$U(2)^{\oplus 3}\oplus E_8(-1)\oplus  \left\langle \left(\sum_{j=1}^8E_j+2E_i\right)/2\right\rangle\simeq U(2)^{\oplus 3}\oplus E_8(-1)\oplus  \langle -8\rangle,$$ whose discriminant coincides with the one of the N\'eron--Severi group. This implies that $$\T_{Z_i'}\simeq U(2)^{\oplus 3}\oplus E_8(-1)\oplus  \langle -8\rangle.$$ 
\end{proof}

\begin{cor}\label{cor: conjecture dimension 6}
 Let $X_3$ be a very general  manifold of $K3^{[3]}$-type admitting a symplectic involution $\iota_X$, $\pi:X_3\ra X_3/\iota_X$ be the quotient map and $Z$ be one of the surface fixed by $\iota_X$. Then there is the following lattice-theoretic isometry: $\T_Z\simeq \pi_*(H^2(X_3,\Z))\simeq \left(\pi_*(\T_{X_3})\right)_{\rm{sat}}$.\end{cor}
\begin{rem}
The previous result is similar to the one in Theorem \ref{Thm1.2}, but in this case, i.e.~ for $n=3$, we are able to prove the existence of an isometry between the lattices $\T_Z$ and  $\left(\pi_*(\T_{X_3})\right)_{\rm{sat}}$, but we do not know if it is a Hodge isometry.
\end{rem}

\subsubsection{Results in higher dimension}

Let $3\leq n\leq 9$, $I$ a set of ordered and distinct indices in $\{1,\ldots,8\}$ such that $|I|= n-2$.
Let $X_n$ be a very general  manifold of $K3^{[n]}$-type with a symplectic involution $\iota$. Deform $(X,\iota)$ to $(S^{[n]},\iota_S^{[n]})$ for a $K3$ surface $S$ with a symplectic involution $\iota_S$ whose fixed points are called $P_i$. Then the image $Z_I$ in $S^{[n]}$ of the surface $\{(s,\iota(s), P_{i_1},\ldots,P_{i_{n-2}})\mbox{ with }s\in S, i_j\in I\}$ is a Nikulin surface such that $Z_I\cap \Delta$ is $$\sum_{j=1}^8 E_j+2E_{i_1}+\ldots +2E_{i_{n-2}}.$$ 

Then, by the same arguments considered in Proposition \ref{prop: fixed locus 6 dimnesion} and Corollary \ref{cor: conjecture dimension 6}, we conclude that there are surfaces in the fixed locus of $\iota$ on $X_n$ whose N\'eron--Severi group is isometric to the orthogonal complement to $\sum_{j=1}^8 E_j+2E_{i_1}+\ldots +2E_{i_{n-2}}$ in the Nikulin lattice and whose transcendental lattice is $U(2)^{\oplus 3}\oplus E_8(-1)\oplus \langle -4(n-1)\rangle$. In particular, there are surfaces $Z_I$ in the fixed locus of $\iota_X$ on $X$ such that $$\T_{Z_i}\simeq \pi_*(H^2(X_n,\Z))\simeq \left(\pi_*(\T_{X_n})\right)_{\rm{sat}}.$$

\end{document}